\documentclass[10pt, a4paper, headsepline, dateoff]{scrartcl} 
\usepackage{scrpage2}    % change headings, see Franzis book, p. 500)
\usepackage{scrtime}     % give time by \thistime
\typearea[5mm]{23}    % [BCOR Bindungskorrektur]{DIV Einteilung der Seite}
 \automark[subsection]{section}  %

\usepackage{scrtime}     % give time by \thistime
\usepackage[german,english]{babel}

\usepackage{ifthen}       % used for ``draft'' mode

\usepackage{graphicx}                % for graphics
\usepackage{amsmath, amssymb}%, theorem}
\usepackage{amsthm}
\usepackage{accents}     % for \accentset, \ring...

\usepackage{bbm}         % blackboard (see: /usr/share/texmf/mytex/bbm.dvi)

\usepackage{bbding}      % for \Checkmark and \XSolidBrush

\usepackage{mathrsfs}    % use ``real'' caligraphic letters for \mathcal
\numberwithin{equation}{section}

\let\OLDthebibliography\thebibliography
\renewcommand\thebibliography[1]{
  \OLDthebibliography{#1}
  \setlength{\parskip}{0pt}
  \setlength{\itemsep}{0pt plus 0.3ex}
}

\newenvironment{myenumerate}
{
  \begin{enumerate}
  \setlength{\itemsep}{1pt}
  \setlength{\parskip}{0pt}
  \setlength{\parsep}{0pt}
}
{
\end{enumerate}
}

\newenvironment{myitemize}{
\begin{itemize}
  \setlength{\itemsep}{1pt}
  \setlength{\parskip}{0pt}
  \setlength{\parsep}{0pt}
}{\end{itemize}}

\newcommand{\itemref}[1]{\noindent \eqref{#1}}

\theoremstyle{plain}            % body italics
\newtheorem{theorem}[equation]{Theorem}

\newtheorem{proposition}[equation]{Proposition}
\newtheorem{lemma}[equation]{Lemma}

\theoremstyle{definition}       % body roman
\newtheorem{definition}[equation]{Definition}

\newtheorem{example}[equation]{Example}

\newtheorem{remark}[equation]{Remark}

\newcommand{\Sec}[1]{Section~\ref{sec:#1}}
\newcommand{\Secs}[2]{Sections~\ref{sec:#1} and~\ref{sec:#2}}
\newcommand{\SecS}[2]{Sections~\ref{sec:#1}--~\ref{sec:#2}}

\newcommand{\Thm}[1]{Theorem~\ref{thm:#1}}
\newcommand{\Thms}[2]{Theorems~\ref{thm:#1} and~\ref{thm:#2}}

\newcommand{\Thmenum}[2]{Theorem~\ref{thm:#1}~(\ref{#2})}
\newcommand{\Thmenums}[3]{Theorem~\ref{thm:#1}~(\ref{#2}) and~(\ref{#3})}

\newcommand{\Ex}[1]{Example~\ref{ex:#1}}

\newcommand{\Lem}[1]{Lemma~\ref{lem:#1}}

\newcommand{\Prp}[1]{Proposition~\ref{prp:#1}}

\newcommand{\Prpenum}[2]{Proposition~\ref{prp:#1}~(\ref{#2})}

\newcommand{\PrpenumS}[3]{Proposition~\ref{prp:#1}~(\ref{#2})--(\ref{#3})}

\newcommand{\Rem}[1]{Remark~\ref{rem:#1}}

\newcommand{\Remenum}[2]{Remark~\ref{rem:#1}~(\ref{#2})}

\newcommand{\Def}[1]{Definition~\ref{def:#1}}
\newcommand{\Defs}[2]{Definitions~\ref{def:#1} and~\ref{def:#2}}

\newcommand{\Defenum}[2]{Definition~\ref{def:#1}~(\ref{#2})}

\newcommand{\abs}[2][{}]{\lvert{#2}\rvert_{{#1}}}    % abs value
\newcommand{\abssqr}[2][{}]{\lvert{#2}\rvert^2_{#1}} % abs squared
\newcommand{\normsymb}{\|}

\newcommand{\bignormsymb}[1]{#1|\!#1|}

\newcommand{\norm}[2][{}]{\normsymb{#2}\normsymb_{{#1}}}    % norm
\newcommand{\normsqr}[2][{}]{\normsymb{#2}\normsymb^2_{#1}} % norm squared
\newcommand{\bignorm}[2][{}]{\bignormsymb{\bigl}{#2}\bignormsymb{\bigr}_{#1}}
\newcommand{\iprod}[3][{}]{\langle{#2},{#3}\rangle_{#1}}  % inner product
\newcommand{\bigiprod}[3][{}]{\bigl\langle{#2},{#3}\bigr\rangle_{#1}}

\newcommand{\set}[2]{\{ \, #1 \, | \, #2 \, \} }      % set { #1 | #2 }
\newcommand{\bigset}[2]{\bigl\{ \, #1 \, \bigl|\bigr. \, #2 \, \bigr\} }

\newcommand{\map}[3]{ #1 \colon #2 \longrightarrow #3}    % maps
\newcommand{\embmap}[3]{ #1 \colon #2 \hookrightarrow #3} % embedding map

\newcommand{\bd}  {\partial}             % symbol for boundary of a set
\newcommand{\clo}[2][]{\overline{{#2}}^{#1}} %  symbol for closure

\newcommand{\restr}[1]{{\restriction}_{#1}} % symbol for map restriction

\newcommand{\card}[1]{\lvert#1\rvert}   % from AMS proceedings file
\newcommand{\dd}    {\, \mathrm d}    % not optimal: no \, if at beginning
\DeclareMathOperator{\dom}    {dom}
\DeclareMathOperator{\ran}    {ran}
\DeclareMathOperator{\id}     {id}   % identity map
\newcommand{\de} {\mathord{\mathrm d}} % exterior derivative

\newcommand{\specsymb} {\sigma} % symbol for spectrum

\newcommand{\spec}[2][{}]   {\specsymb_{\mathrm{#1}}(#2)}

\newcommand{\essspec}[1]{\spec[ess] {#1}}
\newcommand{\disspec}[1]{\spec[disc]{#1}}

\newcommand{\eps}{\varepsilon} % shortcut
\renewcommand{\phi}{\varphi}   % shortcut
\renewcommand{\rho}{\varrho}   % shortcut
\renewcommand{\theta}{\vartheta}% shortcut

\DeclareMathOperator{\myRe}     {Re}   % more professional solution!
\renewcommand{\Re}     {\myRe}
\DeclareMathOperator{\myIm}     {Im}   % more professional solution!
\renewcommand{\Im}     {\myIm} % poor man's solution

\newcommand{\conj}[1]{\overline {{#1}}}  % symbol for complex conjugation

\newcommand{\R}{\mathbb{R}} % symbol for real numbers
\newcommand{\C}{\mathbb{C}} % symbol for complex numbers
\newcommand{\N}{\mathbb{N}} % symbol for natural numbers
\newcommand{\1}{\mathbbm 1}                    % blackboard 1

\newcommand{\e}{\mathrm e}  %Euler number
\newcommand{\im}{\mathrm i} % complex unit

\newcommand{\orth}{\bot}                    % symbol for orthogonality

\newcommand{\normdersymb}{\mathrm{n}}  % symbol for normal derivative
\newcommand{\normder}{\partial_\normdersymb}  % symbol for normal derivative

\newcommand{\wt}{\widetilde}           % shortcut
\newcommand {\qf}[1]{\mathfrak{#1}}    % font for quadratic forms

\newcommand{\HS}{\mathscr H}           % symbol for Hilbert space
\newcommand{\HSaux}{\mathscr G}        % symbol for auxiliary Hilbert space

\newcommand{\Lin}[2][{}]{\mathscr L_{#1}({#2})}% symbol for bdd linear
\newcommand{\Sobsymb} {\mathsf H}      % symbol for Sobolev space
\newcommand{\sobsymb} {\mathsf h}      % symbol for discrete 
\newcommand{\Sobnsymb} {\ring{\mathsf H}}   % symbol for Sobolev space
\newcommand{\Contsymb} {\mathsf C}     % symbol for cont. space
\newcommand{\Lsymb}    {\mathsf L}     % symbol for int L-spaces
\newcommand{\lsymb}    {\ell}          % symbol for int l-spaces

\newcommand{\Sobspace}[1][1]{\Sobsymb^{#1}} 
\newcommand{\sobspace}[1][1]{\sobsymb^{#1}} 
\newcommand{\Sobnspace}[1][1]{\Sobnsymb^{#1}}

\newcommand{\Contspace}[1][{}]{\Contsymb^{#1}}     % symbol for cont. space
\newcommand{\Lpspace}[1][p]    {\Lsymb_{#1}}     % symbol for int L-spaces
\newcommand{\lpspace}[1][p]    {\lsymb_{#1}}     % symbol for int L-spaces
\newcommand{\Lsqrspace}    {\Lpspace[2]}     % symbol for int L-spaces
\newcommand{\lsqrspace}    {\lpspace[2]}          % symbol for int l-spaces

\newcommand{\BdOpsymb} {\mathscr B}       % symbol for bounded linear operators
\newcommand{\BdOp}[2][{}]{\BdOpsymb_{#1}({#2})}
\newcommand{\Ci} [2][{}]{\Contspace [\infty]_{#1} ({#2})}
\newcommand{\Lsqr}[2][{}]{\Lsqrspace^{#1}({#2})} % L_2^{#1}(#2)-spaces
\newcommand{\lsqr}[2][{}]{\lsqrspace^{#1}({#2})}   % l_2(#1)-spaces

\newcommand{\Sob}[2][1]{\Sobspace [#1]({#2})}         % Sobolev space
\newcommand{\Sobn}[2][1]{\Sobnspace [#1]({#2})}  % Sob space ^0

\newcommand{\Sobx}[3][1]{\Sobspace [#1]_{{#2}}({#3})} % Sobolev space
\newcommand{\sob}[2][1]{\sobspace [#1]({#2})} % discrete Sobolev space

\newcommand{\Neu}{{\mathrm N}}              % symbol for Neumann bd cond
\newcommand{\Dir}{{\mathrm D}}              % symbol for Dirichlet bd cond
\newcommand{\laplacian}[2][{}]{\Delta_{{#2}}^{{#1}}} 
\newcommand{\laplacianD}[1]{\laplacian[\Dir]{#1}} % symb f Dir-Laplacian
\newcommand{\laplacianN}[1]{\laplacian[\Neu]{#1}} % symb f Neu-Laplacian
\newcommand{\Krein}{{\mathrm K}}              % symbol for Krein/soft ext.

\newcommand{\err}{\mathrm o}  % not used ... better font!
\newcommand{\Err}{\mathrm O}

\newcommand{\quadtext}[1]{\quad\text{#1}\quad}
\newcommand{\qquadtext}[1]{\qquad\text{#1}\qquad}

\newcommand{\HDir}{H^\Dir}   % Dirchlet operator (with label)
\newcommand{\RDir}{R^\Dir}   % Resolvent of Dirichlt operator (with label)

\newcommand{\HNeu}{H}   % Neumann operator (without label)
\newcommand{\RNeu}{R}   % Neumann operator (without label)

\newcommand{\LS}{\mathscr N}           % symbol for ``loesungs''space
\newcommand{\dplus}{\mathop{\dot+}}

\newcommand{\HKrein}{H^\Krein}   % Krein operator (with label)

\newcommand{\discr}[1]{\dddot #1}         % label for discrete obj
\newcommand{\dlapl}[1][{}]{\discr \Delta_{#1}}

\newcommand{\oneadj}{{1*}}
\newcommand{\adjone}{{*1}}

\newcommand{\weakop}[1]{\breve {#1}}   % symbol for ``weak operator'' 

\newcommand{\wHDir}{\weakop H^\Dir}   % weak Dirchlet operator (with label)
\newcommand{\wRDir}{\weakop R^\Dir}  % weak Resolvent of Dirichlt operator

\newcommand{\wHNeu}{\weakop H}   % Neumann operator (without label)
\newcommand{\wRNeu}{\weakop R}   % Neumann resolvent (without label)

\newcommand{\SDir}{S}          % Dir solution operator (without label)
\newcommand{\eSDir}{\clo    S} % ext. Dir solution operator (without label)

\newcommand{\wLambda}{\weakop \Lambda}  % weak DtN operator

\newcommand{\WS}{\mathscr W}          % symbol for W-space in bd3

\newcommand{\Hmax}{H^{\max}}  % max operator
\newcommand{\Hmin}{H^{\min}}   % min operator
\newcommand{\Deltamax}{\Delta^{\max}}  % max Laplace operator

\newcommand{\wtHDir}{\wt H^\Dir}   % Dirchlet operator (with label)
\newcommand{\wtHNeu}{\wt H}   % Neumann operator (without label)
\newcommand{\wtRDir}{\wt R^\Dir}   % tilde Dirchlet operator (with label)
\newcommand{\wtRNeu}{\wt R^\Dir}  % tilde Neumann operator

\newcommand{\DtN}{Dirichlet-to-Neumann}   % unified notation
\newcommand{\NtD}{Neumann-to-Dirichlet}   % unified notation

\newcommand{\dec}{\mathrm{dec}} % label for decoupled

\newcommand{\yes}{\Checkmark}
\newcommand{\no}{\XSolidBrush}

\begin{document}

%------------------------------------------------------------
% Title page
%------------------------------------------------------------

\title{Boundary pairs associated with quadratic forms}

\author{Olaf Post}

\dedication{\normalsize{School of Mathematics, Cardiff University,
    Senghennydd Road, Cardiff, CF24 4AG, Wales, UK\newline %
    \emph{On
      sabbatical leave from:} Department of Mathematical Sciences,
    Durham University, England, UK\newline %
    \emph{From April 2015:}
    Fachbereich 4 -- Mathematik, Universit\"at Trier, 54286 Trier,
    Germany\newline %
    E-mail: \texttt{olaf.post@uni-trier.de}}}

%------------------------------------------------------------
% Subject classifications
%------------------------------------------------------------

%\subjclass{35P20, 58G18, 47F05}

%------------------------------------------------------------
% Abstract.
%------------------------------------------------------------
\maketitle

\begin{abstract}
  We introduce an abstract framework for elliptic boundary value
  problems in a variational form.  Given a non-negative quadratic form
  in a Hilbert space, a boundary pair consists of a bounded operator,
  the \emph{boundary operator}, and an auxiliary Hilbert space, the
  \emph{boundary space}, such that the boundary operator is bounded
  from the quadratic form domain into the auxiliary Hilbert space.

  These data determine a Neumann and Dirichlet operator, a Dirichlet
  solution and a \DtN\ operator.  The basic example we have in mind is
  a manifold with boundary, where the quadratic form is the integral
  over the squared derivative, and the boundary map is the restriction
  of a function to (a subset of) the boundary of the manifold.
  
  As one of the main theorems, we derive a resolvent formula relating
  the difference of the resolvents of the Neumann and Dirichlet
  operator with the Dirichlet solution and the \DtN\ operator.  From
  this, we deduce a spectral characterisation for a point being in the
  spectrum of the Neumann operator in terms of the family of \DtN\
  operators.  The relation can be generalised also to Robin-type
  boundary conditions.

  We identify conditions expressed purely in terms of boundary pairs,
  which allow us to relate our concept with existing concepts such as
  boundary triples.  We illustrate the theory by many examples
  including Jacobi operators, Laplacians on spaces with (non-smooth)
  boundary and the Zaremba (mixed boundary conditions) problem.
\end{abstract}

%\keywords{Boundary triples, abstract boundary value problems, abstract
%  Dirichlet problem, variational Dirichlet-to-Neumann operator,
%  Krein-type resolvent formulas, spectral characterisation, Zaremba
%  problem}

\tableofcontents

%----------------------------------------------------------------------
% aaaa
\section{Introduction}
\label{sec:intro}
%
%----------------------------------------------------------------------

The following article provides a unified language and approach to very
different types of ``boundary value problems'' based on quadratic
forms, which (to our believe) is very natural and starts with minimal
assumptions on the model.

%----------------------------------------------------------------------
\subsection{The basic example for a boundary pair}
\label{sec:intro.ex}
%----------------------------------------------------------------------
Probably the best way to illustrate the main subject of this article
is to start with our main basic example, from which we borrow the
names for the abstract setting: Let $X$ be a manifold with
sufficiently smooth boundary $\bd X$ (e.g.\ Lipschitz, see
\Sec{lapl.mfd} for details).  We set
\begin{equation*}
  \HS := \Lsqr X, \quad \qf h(u):=\normsqr{\de u}, \quad
  u \in \dom \qf h = \HS^1 := \Sob X,
\end{equation*}
where $\Sob X$ is the completion of the set of smooth function with
respect to the norm given by $\normsqr[\HS^1] u := \normsqr[\HS] u +
\qf h(u)$.  The operator $\HNeu$ associated with $\qf h$ is the usual
\emph{Neumann} Laplacian on $X$.  As boundary space and operator we
set
\begin{equation*}
  \HSaux := \Lsqr Y \quadtext{and} \Gamma u := u \restr Y,
\end{equation*}
respectively, where $Y=\bd X$ (or $Y$ is a suitable subset of $\bd
X$).  By a Sobolev trace estimate, $\map \Gamma{\HS^1} \HSaux$ is
bounded, its kernel is dense in $\HS=\Lsqr X$, as well as its range
$\HSaux^{1/2}:=\Sob[1/2] Y$ (the space of those elements $\phi \in
\Lsqr Y$ having an extension $\wt \phi \in \Sob[1/2]{\bd X}$, i.e.,
$\wt \phi \restr Y=\phi$) is dense in $\HSaux$.

The operator $\HDir$ associated with the form $\qf h^\Dir := \qf h
\restr {\ker \Gamma}$ is the usual \emph{Dirichlet} Laplacian (or, if
$Y \subsetneq \bd X$, a mixed boundary value problem, also called
\emph{Zaremba problem}).  The Sobolev space $\HS^1$ decomposes into
the sum of $\HS^{1,\Dir}:=\ker \Gamma$ and $\LS^1(z)$, where
$\LS^1(z)$ is the \emph{space of weak solutions}, i.e.,
\begin{equation}
  \label{eq:def.ls.intro}
  \LS^1(z) 
  := \bigset{h \in \HS^1}
  {\qf h(h,f) = z \iprod[\HS] h f
    \quad\forall f \in \ker \Gamma=:\HS^{1,\Dir}},
\end{equation}
and the sum is direct if $z \in \C$ is \emph{not} in the spectrum of
the Dirichlet operator $\HDir$.  In the latter case, we can invert the
boundary operator on the weak solution space and call
\begin{equation}
  \label{eq:dsol.intro}
  \map{\SDir(z) := (\Gamma \restr {\LS^1(z)})^{-1}} {\HSaux^{1/2}} {\HS^1}
\end{equation}
the \emph{(weak) Dirichlet solution operator} or \emph{Poisson
  operator}, because $h=\SDir(z) \phi$ is the (unique) weak solution
of the Dirichlet problem $h \in \LS^1(z)$ with $\Gamma h=\phi$.

The \DtN\ operator $\Lambda(z)$ is now defined as the operator
associated with the sesquilinear form
\begin{equation}
\label{eq:dtn.intro}
  \qf l_z(\phi,\psi)
  := (\qf h - z\qf 1)(\SDir(z) \phi, \SDir(-1) \psi),
\end{equation}
i.e., $\qf l_z(\phi,\psi)=\iprod[\HSaux]{\Lambda(z)\phi} \psi$ for all
$\psi \in \HSaux^{1/2}$ and $\phi \in \dom \Lambda(z) \subset
\HSaux^{1/2}$.  Applying the (first) Green's identity for sufficiently
regular $\phi$ one can see that $\Lambda(z)\phi = \normder h \restr
Y$, i.e., that $\Lambda(z)$ acts as the usual \DtN\ operator
associating the normal derivative of the Dirichlet solution of a
boundary value $\phi$ (see~\eqref{eq:why.dtn} below).
 
As one of the main theorems, we derive a resolvent formula and relate
the spectrum of $\HNeu$ with the one of the operator pencil
$\Lambda(\cdot)$, see \Thm{krein.intro} below.

%----------------------------------------------------------------------
\subsection{The abstract theory of boundary pairs}
\label{sec:intro.bd2}
%----------------------------------------------------------------------
In order to develop an analogue abstract theory, we only need the
following data:
\begin{myitemize}
\item a closed, non-negative quadratic form $\qf h$ with domain
  $\HS^1:=\dom \qf h$ in a Hilbert space $\HS$; where $\HS^1$ is
  endowed with its natural norm, see~\eqref{eq:norm.qf};
\item an  operator $\map \Gamma {\HS^1} \HSaux$ (the
  \emph{boundary operator}) into another Hilbert space (the
  \emph{boundary space}).
\end{myitemize}
%----------------------------------------------------------------------
\begin{definition}
  \label{def:bd2.intro}
  We say that $(\Gamma,\HSaux$) is a \emph{boundary pair associated
    with $\qf h$} if $\Gamma$ is bounded, 
  \begin{equation*}
    \text{$\HS^{1,\Dir}:=\ker \Gamma$ is dense in $\HS$ and}
    \qquad
    \text{$\HSaux^{1/2}:=\ran \Gamma$ is dense in $\HSaux$.}
  \end{equation*}

  We say that the boundary pair is \emph{unbounded} if the operator
  $\Gamma$ is \emph{not} surjective.
\end{definition}
%----------------------------------------------------------------------
In particular, our example above leads to an unbounded boundary pair
as the range of the Sobolev trace map $\Gamma$ is $\Sob[1/2] Y$, a
space strictly smaller than $\Lsqr Y$.  It can be seen that a boundary
pair is unbounded iff the \DtN\ operator $\Lambda(z)$ is unbounded
(see \Thm{dn}).  Moreover, in our example, the Dirichlet solution
operator (for $z=0$) defined in~\eqref{eq:dsol.intro} is also called
\emph{Poisson operator}, and this operator extends to a bounded
operator $\map{\eSDir(z)}{\HSaux=\Lsqr Y}{\HS=\Lsqr X}$ onto the
corresponding $\Lsqrspace$-spaces; we call such boundary pairs
\emph{elliptically regular} (see \Def{ell.pos.intro} below).

Let us now formulate one of our main results of this article (see
\Prp{dn.z.inv} and \Thms{krein.res}{krein2}):
%----------------------------------------------------------------------
\begin{theorem}
  \label{thm:krein.intro}
  Let $(\Gamma,\HSaux)$ be an elliptically regular boundary pair, and
  let $z \in \C \setminus (\spec \HDir \cup \spec \HNeu)$ then the
  \DtN\ operator $\Lambda(z)$ has a bounded inverse
  $\map{\Lambda(z)^{-1}}\HSaux \HSaux$ and we have
  \begin{equation}
    \label{eq:krein.intro}
    \RNeu(z) - \RDir(z) = \eSDir(z) \Lambda(z)^{-1} \eSDir(\conj z)^*,
  \end{equation}
  relating the difference of the resolvents $\RNeu(z)=(\HNeu-z)^{-1}$
  and $\RDir(z)=(\HDir-z)^{-1}$ of the Neumann and Dirichlet operator,
  respectively, with the (extended) Dirichlet solution operator
  $\map{\eSDir(z)}\HSaux \HS$ and the \DtN\ operator $\Lambda(z)$.  As
  a consequence, we deduce a spectral characterisation for a point
  being in the spectrum of the Neumann operator, namely
  \begin{equation}
    \label{eq:spec.rel.intro}
    \lambda \in \spec \HNeu \quadtext{iff}
    0 \in \spec {\Lambda(\lambda)},
  \end{equation}
  provided $\lambda \notin \spec \HDir$.
\end{theorem}
%----------------------------------------------------------------------
One can extend this result in a weaker form also for non-elliptically
regular boundary pairs (see \Thms{krein.res}{krein1}); we give an
example of such a boundary triple below.  Resolvent formulas as
in~\eqref{eq:krein.intro} have been shown in many previous works, see
e.g.~\cite[Cor.~3.14]{behrndt-micheler:14}, \cite{agw:14},
\cite[Thm.~6.16, Cor.~6.17]{behrndt-langer.in:12},
\cite[Thm.~7.26]{dhms.in:12}, \cite{dhms:09, behrndt-de-snoo:09},
\cite[Sec.~2]{bgw:09}, \cite[Thm.~1.29]{bgp:08},
\cite[Thm.~2.1]{posilicano:08}, \cite[Thm.2.8]{behrndt-langer:07},
\cite{posilicano:01}, to name a few publications.  We would also like
to stress that one can easily consider Robin-type boundary conditions
instead of Neumann boundary conditions, see \Rem{robin.intro}.  In our
approach, the self-adjointness of \emph{both} operators $\HNeu$ and
$\HDir$ is automatically given.

The spectral relation~\eqref{eq:spec.rel.intro}, a consequence of the
resolvent formula, extends known results for (ordinary) boundary
triples to our general context (see e.g.~\cite[Thm.~2.8]{malamud:10},
\cite[Cor.~2.3]{bgw:09}, \cite{posilicano:08, bgp:08, bmn:02}, going
back to ideas already contained in~\cite{vishik:52}
and~\cite{grubb:68}).  We would like to stress that our approach here
allows the \emph{natural} boundary operators and \DtN\ operator in
this context.

The family of \DtN\ (sesquilinear) forms $(\qf l_z)_{z \in \C \setminus
  \spec \HDir}$ defined in~\eqref{eq:dtn.intro} fulfils the following
important relation
\begin{equation}
  \label{eq:dtn.z.w.intro}
  \frac{\qf l_z(\phi,\psi)-\qf l_w(\phi,\psi)}{z-w}
  = - \iprod[\HS]{\SDir(z) \phi}{\SDir(\conj w) \psi}
\end{equation}
for $\phi,\psi \in \HSaux^{1/2}$ and $z,w \in \C \setminus \spec
\HDir$ (see \Thm{dn.z.qf}).  In particular, for $\conj w=z$, this
formula implies that
\begin{equation}
  \label{eq:dtn.nevanlinna}
   - \frac {(\Im \qf l_z)(\phi,\phi)}{\Im z} 
   = \normsqr[\HS]{\SDir(z) \phi}
   =: \qf q_z(\phi) \ge 0
\end{equation}
for $\phi \in \HSaux^{1/2}$ and $z \in \C \setminus \R$, i.e., $\Im
z>0$ implies that $-\Im \qf l_z \ge 0$; We call such a family $(-\qf
l_z)_z$ a \emph{sesquilinear-form-valued Nevanlinna function}.

%----------------------------------------------------------------------
\subsubsection*{Elliptically regular boundary pairs}
%----------------------------------------------------------------------
In order to relate our concept of boundary pairs with various concepts
of boundary triples in the next subsection, we identify additional
properties, expressed in terms of the quadratic form $\qf q:= \qf
q_{-1}$, respectively the Dirichlet solution operator $\SDir :=
\SDir(-1)$ (see \Defs{ell.bd2}{pos.bd2}):
%----------------------------------------------------------------------
\begin{definition}
  \label{def:ell.pos.intro}
  A boundary pair is \emph{elliptically regular} (resp.\
  \emph{positive}) if $\qf q$ is a bounded (resp.\ uniformly positive)
  quadratic form, i.e., if there is a constant $c \in (0,\infty)$ such
  that
  \begin{equation*}
    \qf q(\phi)
    :=\normsqr[\HS]{\SDir \phi} \le c^2 \normsqr[\HSaux] \phi
    \qquad
    (\text{resp.}\; 
    \qf q(\phi)
    \ge c^2 \normsqr[\HSaux] \phi)
  \end{equation*}
  for all $\phi \in \HSaux^{1/2}=\ran \Gamma$.
\end{definition}
%----------------------------------------------------------------------
One can equivalently use the form $\qf q_z$ in \Def{ell.pos.intro}
(with $z$-depending constant $c$) for some other value $z \in \C
\setminus \spec \HDir$ instead of the form $\qf q:= \qf q_{-1}$ for
$z=-1$ (see \eqref{eq:dsol.z.norm.hs}).  A simple consequence of the
above formula~\eqref{eq:dtn.nevanlinna} is that a boundary pair is
elliptically regular (resp.\ positive) iff $-\Im \qf l_z$ is a bounded
(resp.\ uniformly positive) form for $\Im z>0$.  In particular,
elliptic regularity and positivity of a boundary pair can be seen from
the family $(\qf l_z)_z$ of \DtN\ forms.  One purpose of this article
is to provide further equivalent characterisation of elliptic
regularity and positivity in terms of boundary pairs (see
\Thms{ell.bd2}{pos.bd2}).

%----------------------------------------------------------------------
\begin{remark}
  \label{rem:why.ell.reg.intro}
  Let us mention that the notation ``elliptic regularity'' is inspired
  by our basic example of a Laplacian with smooth boundary $Y=\bd X$
  in \Sec{intro.ex}: the notation relates to the fact that for
  elliptic regular boundary pairs, $\dom \HDir$ and $\dom \HNeu$ are
  both ``nice'' spaces, i.e., they have enough regularity in the sense
  that $\dom \HDir$ and $\dom \HNeu$ are subsets of $\Sob[2] X$.  In
  particular, if $u \in \dom \HDir$ then the normal derivative
  $\Gamma' u=\normder u \restr Y$ is an element of the Hilbert space
  $\HSaux=\Lsqr Y$ (see e.g.~\cite{lions-magenes:68}).  Moreover, we
  have
  \begin{equation*}
    \qf h(u,h)-\iprod[\HS]{\HDir u} h
    = -\iprod[\HS]{(\HDir + 1)u} h
  \end{equation*}
  for $h=S\phi=S(-1)\phi$, as $\qf h(u,h)=-\iprod u h$
  by~\eqref{eq:def.ls.intro} for $z=-1$.  We can rewrite the left hand
  side as $\iprod[\HSaux] {\Gamma'u} \phi$, where $\Gamma'u=\normder u
  \restr Y$ is the restriction of the normal derivative onto the
  boundary $Y$ and $\Gamma h=\phi$ is the boundary value of the
  Dirichlet solution $h$.  Substituting $u=\RDir v = (\HDir+1)^{-1} v
  \in \dom \HDir$, we obtain
  \begin{equation}
    \label{eq:dsol.ell}
    \iprod[\HSaux] {\Gamma' \RDir v} \phi 
    = -\iprod[\HS] v {\SDir \phi}.
  \end{equation}
  hence~\eqref{eq:dsol.ell} defines a bounded functional $\phi \mapsto
  \iprod[\HSaux]{\Gamma' \RDir v} \phi = \iprod[\HS] v {\SDir \phi}$
  with bound $\norm[\HSaux]{\Gamma' \RDir v}$ and $\SDir$ can be
  extended to a bounded operator $\map \eSDir \HSaux \HS$.  The latter
  fact is equivalent with the elliptic regularity of the boundary pair
  (see \Thm{ell.bd2}).  For an equivalent characterisation of elliptic
  regularity in terms of the range of $\Gamma' \RDir$, see
  \Remenum{bd2.ell}{bd3-ell}.
\end{remark}
%----------------------------------------------------------------------
\emph{We would like to stress that we will show the elliptic
  regularity in our examples mostly by applying the definition, and
  not as in the previous remark.} Especially for the Zaremba problem
or manifolds with Lipschitz boundary (see \Secs{lapl.mfd}{zaremba}),
the regularity of the Dirichlet operator is rather delicate, and our
approach avoids questions about the operator domains almost
completely.

%----------------------------------------------------------------------
\subsubsection*{A non-elliptically regular boundary pair: The Zaremba
  problem}
%----------------------------------------------------------------------
There are interesting examples of \emph{non-elliptically regular}
boundary pairs: Consider the \emph{Zaremba problem} of \Sec{intro.ex}
(see \Sec{zaremba} for more details): Let $X$ be a compact Riemannian
manifold with smooth boundary $\bd X$ and $Y$ a ``good'' proper subset
of $\bd X$ (e.g. $\bd Y$ is a smooth submanifold in the manifold $\bd
X$).  The corresponding boundary pair $(\Gamma,\HSaux)$ (with
$\HS^1=\Sob X$, $\Gamma f = f \restr Y$ and $\HSaux=\Lsqr Y$) is
\emph{not} elliptically regular: the Neumann operator of this boundary
pair is the usual Neumann Laplacian on $X$ while the Dirichlet
operator is the \emph{Zaremba Laplacian} (Dirichlet conditions on $Y$
and Neumann conditions on $\bd X \setminus Y$).  Functions in the
domain of the Zaremba operator do \emph{not} all belong to $\Sob[3/2]
X$ (a fact which we dub here ``non-elliptically regular''), hence the
normal derivative $\Gamma'$ of such a function is \emph{not} an
element of the boundary Hilbert space $\HSaux=\Lsqr Y$.  By a similar
argument as in \Rem{why.ell.reg.intro} in the elliptically regular
case, one can then conclude that the Dirichlet solution operator
cannot have a bounded extension onto $\HSaux \to \HS$, and hence the
boundary pair is not elliptically regular (see also
\Remenum{bd2.ell}{bd3-ell}).

We would like to stress that this example cannot be treated with the
quasi-boundary triple method developed by Behrndt and
Langer~\cite{behrndt-langer:07} (see \Def{quasi.bd3}), but our method
still gives a weaker version of \Thm{krein.intro} (see
\Thms{krein.res}{krein1}).

There is also an interesting equivalent character of ``elliptic
regularity'' of a boundary pair in terms of an optimal convergence
speed for Robin-type resolvents in the works of Brasche et
al~\cite{brasche-demuth:05, benamor-brasche:08, bbb:11} (see
\Remenum{eq.ell}{brasche}).  Elliptic regularity in a slightly
different context appeared also in the works of Arlinskii
(\cite{arlinskii:00, arlinskii.in:12}, see
\Remenum{eq.ell}{arlinskii}).

%----------------------------------------------------------------------
\subsection{Purpose of this article and outlook}
\label{sec:purpose}
%----------------------------------------------------------------------

Let us explain here why we believe that our approach is useful:
\begin{myitemize}
\item Boundary pairs give a \emph{simple and unified language}
  bringing together \emph{very different approaches} such as boundary
  triples, Weyl-Titchmarsh functions, Jacobi operators, elliptic
  boundary value problems (even with low regularity as for the Zaremba
  problem or with non-smooth Robin boundary conditions), \DtN\
  operators, boundary conditions for differential form Laplacians,
  non-negative form perturbations, Dirichlet forms, discrete
  Laplacians;

\item The concept of boundary pairs uses only \emph{very little
    information} on the model, namely only quadratic form domains, but
  still allows to develop a \emph{reasonable spectral analysis} of the
  problem.  In particular, one can \emph{avoid} discussions about
  \emph{operator domains} and the rather \emph{delicate analysis of
    regularity} on Lipschitz domains or for the Zaremba problem.

\item To our knowledge, the \emph{elliptic regularity} condition for
  boundary pairs in \Def{ell.pos.intro} and its consequences (see
  \Thm{ell.bd2.impl}) have not yet been recognised as an important
  feature.

\item The form-based approach fits perfectly into the
  \emph{two-Hilbert space convergence scheme} developed
  in~\cite[Ch.~4]{post:12}; especially in the case of
  parameter-depending spaces such as manifolds shrinking to a metric
  graph.
  
\item We see our approach as a starting point for ongoing research,
  and this article is meant to provide the basic tools for the
  following aspects:

  \begin{myitemize}
  \item We provide conditions under which a boundary pair \emph{fits
      into existing concepts} such as \emph{boundary triples} (see
    \Sec{rel.bd3} below); and in what sense it is more general than
    existing concepts.

    Moreover, we believe that the elliptic regularity and positivity
    expressed as in \Def{ell.pos.intro} will lead to further fruitful
    research.

  \item As an example of convergence of operators in different Hilbert
    spaces and boundary pairs, we mention the forthcoming publication
    on \emph{graph-like manifolds}~\cite{behrndt-post:pre14}).
    Similarly, we will use these techniques to show convergence of
    \emph{Dirichlet-Laplacians on domains shrinking to a graph}.

  \item Our approach allows an easy \emph{decomposition} or
    \emph{coupling technique} (see \Sec{coupl.bd2}
    or~\cite[Sec.~3.9]{post:12}), expressing global objects such as
    the spectrum or resolvent in terms of local building blocks if the
    space can be decoupled according to a graph structure (see also
    the example in \Sec{dtn.mg}).

  \item Given a form-valued Nevanlinna function $(-\qf l_z)_{z \in \C
      \setminus [0,\infty)}$, \emph{reconstruct the boundary pair} up
    to unitary equivalence in the spirit of~\cite{langer-textorius:77}
    and \cite{dhms.in:12} (uniqueness will only be possible if $\HNeu
    \cap \HDir$ is simple, i.e., if $\HNeu \cap \HDir$ has no
    non-trivial self-adjoint part).

  \item Relate the concept with the theory of \emph{Dirichlet forms}
    (cf.~\cite{bbb:11} and references therein); assuming that
    $\HS=\Lsqr X$ and $\HSaux = \Lsqr Y$, and $\Gamma$ is compatible
    with the (order-theoretic) lattice structure of the
    $\Lsqrspace$-spaces.
  \end{myitemize}
\end{myitemize}
It is clear that such a general concept cannot avoid deep analysis on
certain classes of problems such as elliptic regularity questions for
partial differential operators (\emph{``There is no free lunch
  \dots''}).  But we believe that we can provide interesting new links
between very different subjects; e.g.\ the property of a boundary pair
to be elliptically regular is equivalent to an optimal convergence
rate for Robin-type resolvents as mentioned earlier (see also
\Remenum{eq.ell}{brasche}).

%----------------------------------------------------------------------
\subsection{Relation with boundary triples and similar concepts}
\label{sec:rel.bd3}
%----------------------------------------------------------------------

The concept of boundary pairs associated with a non-negative quadratic
form is in some sense a generalisation of the concept of
\emph{boundary triples} (also called \emph{boundary value spaces})
associated with a closed operator $\Hmax$.  We will provide the
details of the following results in a forthcoming publication,
see~\cite{post:pre14a}.  Nevertheless, we believe that it is useful to
state the concepts and results already here.

Let $(\Gamma,\HSaux)$ be a boundary pair associated with a quadratic
form $\qf h$.  We define its \emph{minimal operator} by $\Hmin :=
\HDir \cap \HNeu$.  We assume for now on in this section that $\Hmin$
is densely defined (i.e., $\dom \HDir \cap \dom \HNeu$ is dense in
$\HS$).  In this case, the \emph{maximal operator} $\Hmax$ defined by
$\Hmax:=(\Hmin)^*$ is again a (single-valued) operator.
%----------------------------------------------------------------------
\begin{definition}[{cf.~\cite[Def.~3.4.3]{post:12}}]
  \label{def:ass.bd3}
  We say that $(\Gamma,\Gamma',\HSaux)$ is a \emph{boundary triple}
  associated with a quadratic form $\qf h$ if $(\Gamma,\HSaux)$ is a
  boundary pair associated with $\qf h$ such that $\Hmin=\HDir \cap
  \HNeu$ is densely defined, and if $\map{\Gamma'}\WS \HSaux$ is a
  bounded operator, where $\WS$ is dense in $\HS^1=\dom \qf h$ and
  continuously embedded, such that the (first) \emph{Green's (first)
    identity}
  \begin{equation}
    \label{eq:green}
    \qf h(f,g) 
    = \iprod[\HS] {\Hmax f} g + \iprod[\HSaux]{\Gamma' f} {\Gamma g}
  \end{equation}
  holds for all $f \in \WS$ and $g \in \HS^1$.  We also call
  $(\Gamma,\Gamma',\HSaux)$ a \emph{boundary triple associated with
    $(\Gamma,\HSaux)$ (and $\qf h$)}.  We say that the boundary triple
  is \emph{maximal} if the space $\WS$ cannot be enlarged without
  violating the boundedness of $\Gamma'$ or the validity of Green's
  identity in $\HSaux$.
\end{definition}
%----------------------------------------------------------------------
For a given boundary pair $(\Gamma,\HSaux)$ associated with $\qf h$
such that $\Hmin$ is densely defined, there always exists an
associated \emph{maximal} boundary triple $(\Gamma,\Gamma',\HSaux)$
(similar as in~\cite[Def.~2.1]{arlinskii:00}
or~\cite[Def.~3.12]{arlinskii.in:12}).  We will prove this and related
results in~\cite{post:pre14a}.

In the manifold example above,~\eqref{eq:green} is the usual Green's
(first) identity where $\Gamma'$ is the normal derivative restricted
to the boundary.  A simple consequence of Green's identity is a
justification of the name ``Dirichlet-to-Neumann'': Let $z \in \C
\setminus \spec \HDir$ and let $\phi$ be sufficiently ``regular'' in
the sense that $h=\SDir(z) \phi \in \WS$, then
\begin{equation}
  \label{eq:why.dtn}
  \qf l_z(\phi, \psi)
  = (\qf h - z \qf 1)(\SDir(z) \phi, g)
  = \iprod{(\Hmax - z) h} g + \iprod[\HSaux]{\Gamma' h} {\Gamma g}
  = \iprod [\HSaux]{\Gamma' h} \psi,
\end{equation}
where $g=S(-1)\psi$ (or any other function $g \in \HS^1$ with $\Gamma
g = \psi$).  Note that as $\LS^1(z) \subset \ker(\Hmax - z)$, we have
$(\Hmax - z)h=0$.  In particular, we have shown that $\phi$ is in the
domain of the associated \DtN\ operator, and
\begin{equation}
  \label{eq:why.dtn2}
  \Lambda(z) \phi = \Gamma' h,
\end{equation}
i.e., the \DtN\ operator associates to a boundary value $\phi$ the
``normal derivative'' $\Gamma'h$ of the solution of the Dirichlet
problem $h=\SDir(z)\phi$.

%----------------------------------------------------------------------
\begin{remark}
  \label{rem:robin.intro}
  The choice of ``Neumann'' and ``Dirichlet'' conditions for boundary
  pairs is not as restrictive as it looks like at first glance: one
  can replace the quadratic form $\qf h$ by $\qf h_L$, where $\qf
  h_L(u):= \qf h(u) + \iprod{L \Gamma u} {\Gamma u}$ for some suitable
  operator $L$ in $\HSaux$ (e.g.\ bounded).  It can then be shown that
  $(\Gamma,\HSaux)$ is also a boundary pair associated with $\qf h_L$
  (for suitable operators $L$), and that the associated Neumann
  operator is of \emph{Robin-type}, i.e., functions in its domain
  fulfil $\Gamma' u + L \Gamma u =0$, while the Dirichlet operator is
  the same as before (see \Sec{robin}).
\end{remark}
%----------------------------------------------------------------------

%----------------------------------------------------------------------
\subsubsection*{Quasi- and generalised boundary triples}
%----------------------------------------------------------------------
 
Let us now shortly review different concepts of \emph{boundary
  triples} and relate them to boundary pairs.  We will provide the
proofs in a forthcoming publication, see~\cite{post:pre14a}.  For a
recent treatment of \mbox{(quasi-)}boundary triples and related
concepts we refer to the nice surveys~\cite{dhms.in:12} and
\cite{behrndt-langer.in:12} and the references therein (see also
\cite{schmuedgen:12, bgp:08, behrndt-langer:07, dhms:06, bmn:02,
  arlinskii:96, derkach-malamud:95}).  Note that the following
concepts are formulated on the operator level:
% ----------------------------------------------------------------------
\begin{definition}
  \label{def:quasi.bd3}
  Let $ \Hmin$ be a closed, densely defined and symmetric operator
  in $\HS$ and set $\Hmax = (\Hmin)^*$.
  \begin{myenumerate}
  \item A triple $(\Gamma_0,\Gamma_1,\HSaux)$ is a
    \emph{quasi-boundary triple} associated with $\Hmax$
    (see~\cite[Def.~2.1]{behrndt-langer:07} or
    \cite[Def.~6.10]{behrndt-langer.in:12}) if there is a subspace
    $\WS$ of $\dom \Hmax$, dense in $\dom \Hmax$ (with its graph norm),
    such that
    \begin{myenumerate}
    \item
      \label{quasi.bd3.i}
      $\map{(\Gamma_0,\Gamma_1)}{\WS} {\HSaux \oplus \HSaux}$ has
      dense range, where $(\Gamma_0,\Gamma_1) f:=(\Gamma_0 f ,\Gamma_1
      f)$ \emph{(``joint dense range'')}.
    \item
      \label{quasi.bd3.ii}
      $H_0 := \Hmax \restr{\ker \Gamma_0}$ is self-adjoint
      \emph{(``self-adjointness'')},
    \item
      \label{quasi.bd3.iii}
      \emph{Green's (second) identity}
      \begin{equation}
        \label{eq:green.op}
        \iprod[\HS]{\Hmax f} g - \iprod[\HS] f {\Hmax g}
        = \iprod[\HSaux]{\Gamma_0 f}{\Gamma_1 g}
        - \iprod[\HSaux]{\Gamma_1 f}{\Gamma_0 g}
      \end{equation}
      holds for $f,g \in \WS$.
    \end{myenumerate}
    
  \item The triple $(\Gamma_0,\Gamma_1,\HSaux)$ is called a
    \emph{generalised boundary triple} associated with $\Hmax$ (in the
    sense of \cite[Def.~6.1]{derkach-malamud:95}) if it fulfils
    Green's (second) identity~\eqref{eq:green.op} and if the joint
    dense range condition~\eqref{quasi.bd3.i} is replaced by the
    surjectivity of $\Gamma_0$, i.e., by $\Gamma_0(\WS) = \HSaux$.

  \item
  \label{ord.bd3}
    The triple $(\Gamma_0,\Gamma_1,\HSaux)$ is (here) called an
    \emph{ordinary boundary triple} associated with $\Hmax$ if it
    fulfils Green's (second) identity~\eqref{eq:green.op}, if
    $\WS=\dom \Hmax$ and if the ``joint dense range''
    condition~\eqref{quasi.bd3.i} is replaced by the
    \emph{``joint surjectivity''} condition, i.e., by
    $(\Gamma_0,\Gamma_1)(\WS)=\HSaux \oplus \HSaux$.

  \end{myenumerate}
\end{definition}
%----------------------------------------------------------------------
A generalised boundary triple is a quasi-boundary triple (see
\cite[Cor.~3.7]{behrndt-langer:07}
and~\cite[Lem.~6.1]{derkach-malamud:95}), and obviously, an ordinary
boundary triple is a generalised and quasi-boundary triple.  We will
prove the following converse results on boundary pairs with boundary
triples in~\cite{post:pre14a}:
%----------------------------------------------------------------------
\begin{theorem}[\cite{post:pre14a}]
  \label{thm:bd3.intro}
  Let $(\Gamma,\HSaux)$ be a boundary pair associated with $\qf h$ and
  let $(\Gamma,\Gamma',\HSaux)$ be its maximal associated boundary
  triple (as in \Def{ass.bd3}).  Then the following holds:
  \begin{myenumerate}
  \item
    \label{bd3.intro.i}
    The boundary pair $(\Gamma,\HSaux)$ is elliptically regular iff
    $(\Gamma \restr {\WS},\Gamma',\HSaux)$ is a quasi-boundary triple
    associated with $\Hmax$;

  \item
    \label{bd3.intro.ii}
    The boundary pair $(\Gamma,\HSaux)$ is bounded iff $(\Gamma \restr
    {\WS},\Gamma',\HSaux)$ is a generalised boundary triple associated
    with $\Hmax$.

  \item
    \label{bd3.intro.iii}
    The boundary pair $(\Gamma,\HSaux)$ is bounded and positive iff
    $(\Gamma\restr {\WS},\Gamma',\HSaux)$ is an ordinary boundary
    triple.
  \end{myenumerate}
\end{theorem}
%----------------------------------------------------------------------

In~\cite{dhms:06}, the more general notation of \emph{boundary
  relation} has been introduced (basically, multi-valued operators and
boundary maps are allowed).  One of the results of~\cite{dhms:06}
states a characterisation of boundary relations by their
\emph{Nevanlinna} ``functions'' $-\Lambda(\cdot)$ (see also the
survey~\cite{dhms.in:12}).  In particular, generalised boundary
triples have a Nevanlinna function $-\Lambda(\cdot)$ with values in
$\BdOp \HSaux$ (the space of \emph{bounded} operators) such that $-\Im
\Lambda(z)$ is \emph{injective}, while ordinary boundary triples have
again values in $\BdOp \HSaux$ and \emph{uniformly positive} operators
$-\Im \Lambda(z)$ (so-called \emph{uniformly strict Nevanlinna}
functions), see~\cite{dhms:06} for more such characterisations.  One
can now characterise \emph{elliptically regular} boundary pairs in
terms of $\Lambda(\cdot)$ and reconstruct $(\Gamma,\HSaux)$ from it in
the spirit of~\cite{langer-textorius:77} (see the outlook list of
\Sec{purpose} above).

\paragraph{To summarise:} Our concept of boundary pairs is more general as
it includes \emph{non-elliptically regular} boundary pairs; this
concept is not covered by quasi-boundary triples.  On the other hand,
our concept is only suitable for \emph{non-negative} operators $\HNeu$
and $\HDir$ (although there is a natural extension of the theory to
\emph{semi-bounded} or \emph{sectorial} operators).

Let us illustrate the different properties of boundary pairs and its
relation to boundary pairs in a tabular: we also refer to (classes of)
examples showing that the above properties of boundary pairs are all
independent except the following: A boundary pair $(\Gamma,\HSaux)$
with finite-dimensional boundary space is automatically bounded (i.e.,
$\ran \Gamma=\Gamma(\HS^1)=\HSaux$) and positive; and a bounded
boundary pair is automatically elliptically regular.
%----------------------------------------------------------------------
\begin{center}
  \begin{tabular}[h]{|c|c|c|c|p{0.42\textwidth}||c|c|c|}
    \hline
    \multicolumn{4}{|l|}{boundary pair is}& Examples &
    \multicolumn{3}{|l|}{ass.\ boundary triple is}\\
    \texttt{fin} & \texttt{bdd} & \texttt{ell} & \texttt{pos} & &
    \texttt{o-bd3} & \texttt{g-bd3} & \texttt{q-bd3}\\
    \hline \hline
    \yes & \yes & \yes & \yes & \Sec{2dim} (Sturm-Liouville)  & % bd2
    \yes & \yes & \yes \\                   % bd3 
%    \yes & \yes & \yes & \no & \Sec{}\\ geht nicht
%    \yes & \yes & \no & \yes & \Sec{}\\ geht nicht
%    \yes & \yes & \no & \no & \Sec{}\\  geht nicht
    \no & \yes & \yes & \yes & metric graphs with infinitly many vertices,
    \newline see~\cite{post:12} & % bd2
    \yes & \yes & \yes \\                   % bd3 
    \no & \yes & \yes & \no & \Ex{bdd.non-pos}:
                   modified manifold example &% bd2\\
    \no  & \yes & \yes \\                   % bd3
%    \no & \yes & \no & \yes & \Sec{}\\  geht nicht
%    \no & \yes & \no & \no & \Sec{}\\   geht nicht
    \no & \no & \yes & \yes & \Ex{bd2.ell.pos} (Jacobi) &% bd2
    \no  & \no & \yes \\                       % bd3
    \no & \no & \yes & \no & \Secs{lapl.mfd}{dtn.mg}, \Thm{zaremba.bd2} (mfds)
             &% bd2
    \no  & \no & \yes \\                   % bd3
    \no & \no & \no & \yes & \Ex{non-ell.bd2} (Jacobi) &% bd2
    \no  & \no & \no \\                   % bd3
    \no & \no & \no & \no & \Thm{zaremba2.bd2} (Zaremba) &% bd2
    \no  & \no & \no \\                   % bd3
    \hline
  \end{tabular}
  \parbox{0.9\textwidth}{\yes --- yes; \no --- no.  \texttt{fin}:
    finite dimensional boundary space $\HSaux$; \texttt{bdd}: bounded
    boundary pair ($\ran \Gamma=\HSaux)$; \texttt{ell}: elliptically
    regular boundary pair; \texttt{pos}: positive boundary
    pair.  \texttt{o-bd3}: the associated boundary triple is an
    ordinary boundary triple; \texttt{g-bd3}: \dots generalised
    boundary triple; \texttt{q-bd3}: \dots quasi-boundary triple.}
\end{center}
%----------------------------------------------------------------------

%----------------------------------------------------------------------
\subsection{Extension theory and related other works}
\label{sec:ext.th.intro}
%----------------------------------------------------------------------

Let us briefly comment on other concepts, in particular, the concept
of extension theory: the characterisation of all possible
(self-adjoint) extensions of a given symmetric operator.

%----------------------------------------------------------------------
\subsubsection*{Grubb's extension theorem via forms}
%----------------------------------------------------------------------
There is a close relation of our form approach with extension theory,
especially for non-negative operators: Grubb proved
in~\cite[Sec.~1]{grubb:70} that all non-negative and self-adjoint
extensions of a minimal, closed, uniformly positive operator $\Hmin$ lie
in between the \emph{Friedrichs extension} $\HDir$ (defined as the
operator associated with the closure of the quadratic form $\qf
h^\Dir(f):=\iprod {\Hmin f} f$) and the so-called \emph{Krein} or
\emph{soft extension} defined via the quadratic form
\begin{equation}
%  \label{eq:krein.qf}
  \dom \qf h^\Krein := \dom \qf h^\Dir \dplus \ker \Hmax, \qquad
  \qf h^\Krein(u):= \qf h^\Dir(u^\Dir),
\end{equation}
where $u^\Dir$ is the projection of $u \in \dom \qf h^\Krein$ onto
$\dom \qf h^\Dir$ along $\ker \Hmax$ ($\Hmax:=(\Hmin)^*$).  The
associated operator $\HKrein$ acts on $\dom \HKrein = \dom \Hmin
\dplus \ker \Hmax$ as $\HKrein u = \Hmin u^\Dir$ (a result of Krein,
see~\cite[Thm.~14]{krein:47i}) and is self-adjoint and non-negative.

To be more precise, Grubb's result is as follows
(cf.~\cite[Cor.~1.3]{grubb:70}): If $\qf h$ is a closed non-negative
quadratic form in $\HS$ and $\qf h \restr {\dom \qf h^\Dir}=\qf
h^\Dir$, such that its domain lies in between $\dom \qf h^\Dir$ and
$\dom \qf h^\Krein$ and $\qf h^\Krein(u) \le \qf h(u)$ for all $u \in
\dom \qf h$, then the associated operator $H$ is an extension of
$\Hmin$ ($\Hmin \subset H \subset \Hmax$), and all positive extensions
of $\Hmin$ appear in this way.  One can also characterise the
extension in terms of the quadratic form $\qf t$, the restriction of
$\qf h$ to the space $\ker \Hmax$. This form then corresponds to our
family of \DtN\ forms at $z=0$.

There are many more approaches to self-adjoint extensions of positive
closed operators, see e.g.~\cite{malamud:92b, alonso-simon:80,
  ando-nishio:70, birman:56, vishik:52, krein:47i}.

Grubb also contributed to the Zaremba (mixed) boundary problem
in~\cite{grubb:11}; she provides Krein-type formula (with the
Dirichlet Laplacian as a reference operator, hence it is in our
elliptically regular setting, see \Thm{zaremba.bd2}); and also
Weyl-type asymptotics for the resolvent difference).

%----------------------------------------------------------------------
\subsubsection*{Arlinskii's boundary pairs}
%----------------------------------------------------------------------
The notion of \emph{boundary pair} appears also in works of (see also
the survey~\cite{arlinskii.in:12} and references therein), but with
the Krein extension as Neumann operator and the additional condition
that $\ran \Gamma=\HSaux$.  In other words, given a minimal operator
$\Hmin$, Arlinskii defines a boundary pair $(\Gamma,\HSaux)$
associated with $\qf h^\Krein$ such that $\Gamma(\dom \qf
h^\Krein)=\HSaux$, i.e., the boundary pair is \emph{bounded} in our
notation.  Arlinskii deals with sectorial operators, and not with
non-negative operators, as we do.  In our opinion, Arlinskii is mainly
interested in characterising all possible variational extensions of
the associated minimal operator for sectorial operators (we just
mentioned Grubb's result for non-negative operators
above~\cite{grubb:70}).  Arlinskii~\cite{arlinskii:99} also associates
a boundary triple with a boundary pair in the same spirit as we do in
\Sec{rel.bd3}.

Lyantse and Storozh~\cite{lyantse-storozh:83} use a similar notion of
boundary pairs for operators corresponding basically to \emph{bounded}
boundary pairs in our notation.  There is another approach for first
order systems in~\cite{mogilevskii:12}.  Malamud and
Mogilevskii~\cite{malamud-mogilevskii:02} discuss the extension theory
for dual pairs of operators or even relations and also provide a
Krein-type formula for the resolvents.

%----------------------------------------------------------------------
\subsubsection*{Posilicano's approach}
%----------------------------------------------------------------------
Posilicano~\cite{posilicano:08} (see also~\cite{posilicano:04,
  posilicano:01}) considers (in our notation) a self-adjoint extension
$\HDir$ and a bounded operator $\map {\Gamma'} {\dom \HDir} \HSaux$
($\dom \HDir$ with its graph norm) which is surjective and has dense
kernel.  He then describes all self-adjoint extensions of the
associated minimal operator $\Hmin := \HDir \restr {\ker \Gamma'}$ via
a Krein-type formula.  He also applies his concept to Laplacians on
open subsets $X \subset \R^n$ with smooth boundary, but the
surjectivity condition on $\Gamma'$ enforces $\HSaux=\Sob[1/2] {\bd
  X}$. Other results on Krein-type resolvent formulae are also
considered in~\cite{agw:14, dhms:09, behrndt-de-snoo:09,
  posilicano:01}.

%----------------------------------------------------------------------
\subsubsection*{Arendt's and ter Elst's $\Gamma$-elliptic forms}
%----------------------------------------------------------------------
Arendt and ter Elst~\cite{arendt-ter-elst:12} define a generalised
notion for sesquilinear forms and associated operators, using an
operator playing the role of our boundary operator $\map \Gamma
{\HS^1}\HSaux$.  Namely, they say that a sesquilinear form $\qf a$ is
is \emph{$\Gamma$-elliptic} if there exist $\alpha>0$ and $\omega \in
\R$ such that
\begin{equation}
  \label{eq:g-ell}
  \Re \qf a(u) + \omega \normsqr[\HSaux]{\Gamma u}
  \ge \alpha \normsqr[\HS^1] u
\end{equation}
holds for all $u \in \HS^1$.

To such a $\Gamma$-elliptic form $\qf a$, one can associate an
operator $A$ on $\HSaux$ by setting $\phi \in \dom A$ and $A\phi =
\psi$ iff there exists $u \in \HS^1$ such that $\Gamma u=\phi$ and
$\qf a(u,v)= \iprod[\HSaux] \psi {\Gamma v}$ for all $v \in \HS^1$.
We say that \emph{$A$ is the operator associated with
  $(\HS^1,\Gamma,\qf a)$}.

We can apply this definition in the following way: Let
$(\Gamma,\HSaux)$ be a boundary pair associated with $\qf h$.  Set now
$\qf a= \qf h - z \qf 1$, then it is not difficult to see that the
operator associated with $(\HS^1,\Gamma,\qf h-z\qf 1)$ is the \DtN\
operator $\Lambda(z)$.

The notation ``$\Gamma$-elliptic'' does not refer to the Hilbert space
$\HS$, and hence, there is no direct relation to our notion of
``elliptic regularity'' in the sense of \Def{ell.pos.intro}, but it
can easily be seen that if $\Re z<0$, then $\qf h - z\qf 1$ is
$\Gamma$-elliptic (with $\alpha=\min\{1,-\Re z\}$ and $\omega=0$).  On
the other hand, if $(\Gamma,\HSaux)$ is an elliptically regular
boundary pair, and if $0 \notin \spec \HDir$, then $\qf h$ is
$\Gamma$-elliptic.

Arendt and ter Elst use this abstract concept to define a \DtN\
operator even on very rough domains, see~\cite{arendt-ter-elst:11},
assuming e.g.\ that $\Gamma$ is only a \emph{closed} map in $\HS^1 \to
\HSaux$.  They also relate the concept with operator semi-groups
(see~\cite{arendt-ter-elst:12b}).  In an older paper, Greiner 
considers perturbations of semi-groups by boundary conditions in the
setting of Banach spaces (see~\cite{greiner:87}).

%----------------------------------------------------------------------
\subsubsection*{Brasche et al's non-negative form perturbations}
%----------------------------------------------------------------------
A different approach using the notion of Dirichlet forms is used in
the works of Brasche et al~\cite{brasche-demuth:05,
  benamor-brasche:08, bbb:11}, see
also~\cite[Ex.~3.6]{posilicano:01}).  Their concept is called
\emph{non-negative form perturbations} and can equivalently be given
(in our notation) by a non-negative closed quadratic form $\qf h$ in
$\HS$ with domain $\HS^1:=\dom \qf h$, an auxiliary Hilbert space
$\HSaux$ and an identification operator $\Gamma$, \emph{closed} as
operator $\HS^1 \to \HSaux$ and densely defined in $\HS^1$ with dense
range $\ran \Gamma$ (\cite[Ex.~2.1 and Lem.~2.2]{bbb:11}).  It is more
general than our concept since $\Gamma$ is not assumed to be bounded
as operator $\HS^1 \to \HSaux$ and since $\ker \Gamma$ is not assumed
to be dense in $\HS$ (we have one example where we also drop the
latter density condition, see \Sec{large.bd2}).  We would like to
stress that Brasche et al only consider (what we call) the \DtN\
operator $\Lambda=\Lambda(-1)$ at $z=-1$ and not \emph{families} of
\DtN\ operators $(\Lambda(z))_z$ as we do.

One of the main examples in~\cite{bbb:11} is the Laplacian on $\R$
with infinitely many delta interactions on it (in other words,
Robin-type perturbations of the Neumann Laplacian on $\R$, see
\Remenum{eq.ell}{brasche}), a case also treated in detail
in~\cite{kostenko-malamud:10} (see \Rem{delta.malamud} and
also~\cite{malamud-schmuedgen:12,cmp:13,bll:13a} for other recent
applications of boundary triple methods to delta-type Schr\"odinger
operators).  Brasche et al also introduce an interesting equivalent
characterisation of ``elliptic regularity'' of a boundary pair in
terms of an optimal convergence speed for Robin-type resolvents in the
works of Brasche et al~\cite{brasche-demuth:05, benamor-brasche:08,
  bbb:11} (see again \Remenum{eq.ell}{brasche}).

%----------------------------------------------------------------------
\subsubsection*{Applications of boundary triples to elliptic boundary
  value problems}
%----------------------------------------------------------------------

There have been many attempts to apply the concept of boundary triples
to elliptic boundary value problems on $X$.  Most attempts try to keep
the concept of boundary triples (namely the ``joint surjectivity
condition'' of \Defenum{quasi.bd3}{ord.bd3}), for example
by``regularising'' the boundary maps by applying some isomorphism
$\Sob[1/2]{\bd X}$ onto $\Lsqr{\bd X}$ (e.g. the square root of the
\DtN\ map $\Lambda^{1/2}$), see also \Sec{bdd.bd2} for an abstract
``regularisation'' for boundary pairs.  This approach has been used
e.g.\ in~\cite{malamud:10, bgw:09, pankrashkin:06b}).

Abstract formulations of elliptic boundary value problems have already
been considered by Grubb in~\cite{grubb:68,grubb:70} where she
considers self-adjoint extensions of non-negative (and more general)
operators, being elliptic on a subset $X$ of $\R^n$ with \emph{smooth}
boundary $\bd X$.  Recently, Gesztesy and
Mitrea~\cite{gesztesy-mitrea:11} (see also~\cite{gesztesy-mitrea:08,
  gesztesy-mitrea:09}) provided a detailed analysis of all extensions
of Laplacians on a subclass of Lipschitz domains, also providing
Krein-type formulae as in~\eqref{eq:krein.intro}.  They are not using
the language of boundary triples, but their results could easily be
embedded e.g.\ in the concept of boundary pairs or quasi-boundary
triples.  The latter has been successfully applied to elliptic boundary
value problems, even in the setting of Lipschitz domains in $\R^n$,
see~\cite{behrndt-micheler:14} and the references therein.  There are
similar results by Grubb~\cite{grubb:08}, providing Krein-type
resolvent formulae also for certain non-smooth domains, using
pseudo-differential boundary operator methods adopted to the non-smooth
case.  Ryzhov~\cite{ryzhov:07} uses a similar concept as boundary
triples (in a weaker version) and also defines what we call \DtN\
operator.  He assumes what we call elliptic regularity (see
\cite[Prp.~1.16~(2)]{ryzhov:07}).

%----------------------------------------------------------------------
\subsection{Structure of this article}
%----------------------------------------------------------------------
\Sec{bd2.qf} contains the basic notion of a boundary pair, the
Dirichlet solution operator and the \DtN\ operator and properties
related with these operators.  In \Sec{bd2.add} we find additional
properties of boundary pairs needed in order to prove certain
Krein-type resolvent formulae and spectral relations in
\Sec{bd2.krein}.  \Sec{bd2.constr} consists of boundary pairs
constructed from others, such as the Robin-type perturbation in
\Sec{robin} and coupled boundary pairs in \Sec{coupl.bd2}; and a
construction how to turn an unbounded boundary pair in a bounded one
in \Sec{bdd.bd2}.  In \Sec{examples} we provide many examples
including Laplacians on intervals, Jacobi operators, Laplacians on
sets with Lipschitz boundary, a \DtN\ operator supported on an embedded
metric graph, the Zaremba problem and discrete Laplacians.

%----------------------------------------------------------------------
\subsection{Acknowledgements}
%----------------------------------------------------------------------
The author is indebted to many colleagues for helpful discussions and
comments while working on this project and for carefully reading
earlier versions.  These thanks go especially to Yury Arlinskii, Jussi
Behrndt, Gerd Grubb, Matthias Langer, Mark Malamud, Marco Marletta,
Till Micheler, Jonathan Rohleder, Karl-Michael Schmidt and Michael
Strauss.  The author acknowledges the financial support of the SFB~647
``Space --- Time --- Matter'' while the author was at the Humboldt
University at Berlin and first ideas for this article were developed.
The author enjoyed the hospitality at Cardiff University and thanks
for the financial support given by the Marie Curie grant
``COUPLSPEC''.  Last but not least the author would like to thank the
referees for many useful comments on the article.

%----------------------------------------------------------------------
% bbbb
\section{Boundary pairs associated with quadratic forms}
\label{sec:bd2.qf}
%
%----------------------------------------------------------------------

%----------------------------------------------------------------------
\subsection{Preliminaries: quadratic and sesquilinear forms, scales of
  Hilbert spaces}
%----------------------------------------------------------------------
\label{sec:prelim}

In this article, $\HS$, $\HSaux$ etc.\ denote Hilbert spaces with norm
and inner product denoted by $\norm[\HS] \cdot$, $\iprod[\HS] \cdot
\cdot$ etc.  Let $\HS^1$ be a subspace of $\HS$ with another inner
product $\iprod[\HS^1]\cdot \cdot$, such that the inclusion $\HS^1
\hookrightarrow \HS$ is bounded.  A \emph{sesquilinear form} $\map
{\qf h}{\HS^1 \times \HS^1}\C$ is a map linear in its first and
anti-linear in its second argument.  A sesquilinear form determines a
\emph{quadratic form} by $\qf h(u):= \qf h(u,u)$ (denoted by the same
symbol), its domain is denoted by $\dom \qf h = \HS^1$.  A quadratic
form uniquely determines its associated sesquilinear form $\map{\qf
  h}{\HS^1 \times \HS^1}\C$ by the polarisation identity $\qf h(u,v):=
\sum_{k=0}^3 \im^k \qf h(u+\im^kv)$.  Therefore, we will mostly speak
of ``forms'' in the sequel, meaning either ``quadratic form'' or
``sesquilinear form''.

A form is called \emph{non-negative} resp.\ \emph{(uniformly)
  positive} if $\qf h(u) \ge 0$ resp.\ $\qf h(u) \ge c \normsqr[\HS]
u$ for all $u \in \dom \qf h=\HS^1$ and some $c>0$.  
We sometimes write $\qf 1$ for the form $\qf 1(u):=\normsqr[\HS] u$.
A non-negative
form is called \emph{closed} if the norm defined by
\begin{equation}
  \label{eq:norm.qf}
  \normsqr[\qf h] u := \qf h(u) + \normsqr[\HS] u
\end{equation}
is equivalent with the norm on $\HS^1$, i.e., if $\HS^1$ is a Hilbert
space w.r.t.\ the norm $\norm[\qf h] \cdot$.

The \emph{adjoint} $\qf a^*$ of a form $\qf a$ is defined by $\qf
a^*(\phi,\psi):= \conj{\qf a(\psi,\phi)}$.  The \emph{real} resp.\
\emph{imaginary part} of a form $\qf a$ is defined as $\Re \qf a :=
\frac 12 (\qf a + \qf a^*)$ resp.\ $\Im \qf a := \frac 1{2 \im} (\qf a
- \qf a^*)$.  Let $D$ be an open subset of $\C$ invariant under
complex conjugation.  A family $(\qf l_z)_{z \in D}$ of sesquilinear
forms is called \emph{symmetric}, if $\qf l_z^* = \qf l_{\conj z}$.

If $\qf h$ is a closed, non-negative quadratic form, then there is a
unique self-adjoint, non-negative operator $H$ defined for those $u
\in \HS$ such that there is $v \in \HS$ with $\qf h(u,w)=\iprod[\HS] v
w$ for all $w \in \dom \qf h$, and we set $Hu := v$.  We also call $H$
the \emph{strong} operator associated with $\qf h$.

A self-adjoint, non-negative operator defines a \emph{scale of Hilbert
  spaces} $\HS^k:=\dom (H+1)^{k/2}$ with norm $\norm[\HS^k] u :=
\norm{(H+1)^{k/2} u}$ for $k \ge 0$ such that $\dom \qf h = \HS^1$.
For negative $k$, we set $\HS^{-k} := (\HS^k)^*$, where $(\cdot)^*$
refers to the pairing induced by the inner product of $\HS$.  In
particular, we can interpret $\HS^{-k}$ as the completion of $\HS$
with respect to the norm $\norm[\HS^{-k}] u = \norm{(H+1)^{-k/2}u}$;
for details on scales of Hilbert spaces, see~\cite[Sec.~3.2]{post:12}.
We sometimes use the notation $A^{k \to m}$ (or similar ones) to
indicate that $\map A {\HS^k}{\HS^m}$ is a bounded operator.  In this
sense, $H^{k \to k-2}$ is the extension/restriction of $H$, and a
bounded operator.  Especially, if $k=1$, we write $\weakop H := H^{1
  \to -1}$, called the \emph{weak} operator associated with $\qf h$,
and defined by $(\weakop H u) v := \qf h(u,v)$.  There is a very nice
summary about these facts in~\cite[App.~B]{gesztesy-mitrea:08}.

Finally, we denote by $\HS = \HS_1 \dplus \HS_2$ the \emph{topological
  sum} of $\HS_1$ and $\HS_2$, i.e., the sum is direct (but not
necessarily orthogonal), and $\HS_1$, $\HS_2$ are closed in $\HS$.

%----------------------------------------------------------------------
\subsection{Boundary pairs, Dirichlet solution operators and \DtN\
  operators}
%----------------------------------------------------------------------

We start with our basic object, using only quadratic form domains for
the moment.
%----------------------------------------------------------------------
\begin{definition}
%  \addtocounter{equation}{-1} %
  \label{def:bd2}
  Let $\qf h$ be a closed non-negative and densely defined quadratic
  form in the Hilbert space $\HS$ with domain $\HS^1 := \dom \qf h$.
  We endow $\HS^1$ with its natural norm given by~\eqref{eq:norm.qf}.
  Moreover, let
  \begin{equation*}
    \map \Gamma{\HS^1}\HSaux
  \end{equation*}
  be a bounded map, where $\HSaux$ is another
  Hilbert space. We denote the norm of the operator $\Gamma$ by
  $\norm[1 \to 0] \Gamma$.
  \begin{myenumerate}
  \item We say that $(\Gamma,\HSaux)$ is an \emph{(ordinary) boundary
      pair} (or $\Gamma$ is a \emph{boundary map}) associated with the
    quadratic form $\qf h$ if the following conditions are fulfilled:
    \begin{myenumerate}
    \item $\HS^{1,\Dir}:=\ker \Gamma$ is dense in $\HS$.
    \item $\HSaux^{1/2} :=\ran \Gamma$ is dense in $\HSaux$.
    \end{myenumerate}
    If the first condition is not fulfilled, i.e., if $\ker \Gamma$ is
    not dense in $\HS$ then we say that the boundary space is
    \emph{large in $\HS$} or shortly, that $(\Gamma,\HSaux)$ is a
    \emph{generalised boundary pair}.

  \item If $\HSaux^{1/2} \subsetneq \HSaux$ (i.e., if the boundary map
    $\Gamma$ is not surjective), then we call the boundary pair
    $(\Gamma, \HSaux)$ \emph{unbounded}.  Otherwise, if the boundary
    map is \emph{surjective}, then we call the boundary pair
    \emph{bounded}.

  \item We call the self-adjoint and non-negative operator $\HNeu$
    associated with $\qf h$ (see~\cite[Thm.~VI.2.1]{kato:66}) the
    \emph{Neumann operator}.  Its resolvent is denoted by $\RNeu(z) :=
    (\HNeu-z)^{-1}$ for $z \in \C \setminus \spec \HNeu$ and we set
    $\RNeu:=\RNeu(-1)$.  The associated scale of Hilbert spaces is
    denoted by $\HS^{k,\Neu}:=\dom \HNeu^{k/2}$ with norm $\norm
    [k,\Neu] u := \norm{(\HNeu+1)^{k/2} u}$.

  \item
    \label{bd2.dir}
    Denote by $\qf h^\Dir := \qf h \restr{\HS^{1,\Dir}}$ the form
    restricted to $\HS^{1,\Dir}$. (Note that $\qf h^\Dir$ is a closed
    form since $\map \Gamma {\HS^1} \HSaux$ is bounded, hence $\dom
    \qf h^\Dir=\ker \Gamma$ is closed in $\HS^1$.)  Moreover, we call
    the self-adjoint and non-negative operator $\HDir$ associated with
    $\qf h^\Dir$ in $\HS$ the \emph{Dirichlet operator}.  We denote
    its resolvent by $\RDir(z) := (\HDir-z)^{-1}$ for $z \in \C
    \setminus \spec \HDir$ and set $\RDir:=\RDir(-1)$.  The associated
    scale of Hilbert spaces in $\HS^{0,\Dir}$ is denoted by
    $\HS^{k,\Dir}:=\dom (\HDir)^{k/2}$ with norm $\norm [k,\Dir] u :=
    \norm{(\HDir+1)^{k/2} u}$.

    If $(\Gamma,\HSaux)$ is a generalised boundary pair, then we
    denote by $\HS^{0,\Dir}$ the closure of $\ker \Gamma=\HS^{1,\Dir}$
    in $\HS$.   For consistency, we extend $\RDir(z)$ by $0$
    on $(\HS^{0,\Dir})^\orth$, and denote the extended resolvent by
    the same symbol, i.e., we set $\RDir(z) f := (\HDir-z)^{-1} f^\Dir
    \oplus 0$ for $f=f^\Dir \oplus f^\orth$.
  \end{myenumerate}
\end{definition}
%----------------------------------------------------------------------

%----------------------------------------------------------------------
\begin{remark}
  \label{rem:bd2}
  \indent
  \begin{myenumerate}
  \item
    \label{bd2.ran}
    If $\ran \Gamma$ is not dense in $\HSaux$, then we can replace
    $\HSaux$ by $\HSaux_0 := \clo{\ran \Gamma}$.

  \item
    \label{bd2.ker}
    In most of our examples, $\ker \Gamma$ is dense in $\HS$, and we
    mean by ``boundary pair'' an ``ordinary boundary pair''.  An
    example of a generalised boundary pair is presented in
    \Ex{basic.ex} and in \Sec{large.bd2}.

  \item We mostly work with \emph{unbounded} quadratic forms $\qf h$.
    In \Sec{large.bd2}, we present an example with a \emph{bounded}
    form $\qf h$ related to a \emph{discrete} Laplacian on a graph; in
    this case, $\HS=\HS^1$.

  \item Note that if the boundary pair $(\Gamma,\HSaux)$ associated
    with $\qf h$ is not bounded (i.e., $\ran \Gamma \subsetneq
    \HSaux$), then we can turn it into a \emph{bounded} boundary pair
    $(\wt \Gamma,\wt \HSaux)$ (i.e., $\ran \wt \Gamma = \wt \HSaux$)
    associated with $\qf h$ by changing the boundary space $\wt
    \HSaux$ and its norm, see \Prp{bd2.bdd}.  Nevertheless, after this
    modification, the boundary space is less natural in many
    applications, and no longer an ordinary $\Lsqrspace$-space as in
    the following basic example.
  \end{myenumerate}
\end{remark}
%----------------------------------------------------------------------

%----------------------------------------------------------------------
\begin{example}
  \label{ex:basic.ex}
  Let us illustrate the above setting by a prototype we have in mind
  (see also~\cite{bbb:11} and references therein): Assume that
  $(X,\mu)$ is a measured space.  As quadratic form we choose an
  ``energy form'', i.e., $\qf h(f) = \int_X \abssqr{\de f} \dd \mu$,
  where $\abssqr{\de f}$ is usually a sort of squared norm of a
  ``derivative'' on $X$.  Moreover, we assume that $Y \subset X$ is
  measurable (the ``boundary'' of $X$) and $\nu$ is a measure on $Y$.
  We set $\HS := \Lsqr{X,\mu}$ and $\HSaux := \Lsqr {Y,\nu}$.  As
  boundary map we choose $\Gamma f := f \restr Y$.  One has to check
  now that $\Gamma$ is bounded as operator $\HS^1 \to \HSaux$, i.e.,
  that there is a constant $C>0$ such that
  \begin{equation*}
    \int_Y \abssqr f \dd \nu
    \le C \int_X \bigl(\abssqr{\de f} + \abssqr f \bigr) \dd \mu.
  \end{equation*}
  If $\mu(Y)>0$, then $(\Gamma,\HSaux)$ is a \emph{generalised}
  boundary pair, and we may choose as measure $\nu$ the measure
  induced by $X$ (i.e., $\nu(B):=\mu(B)$ for measurable sets $B
  \subset Y$).  In this article, we mostly are interested in the case
  when $\mu(Y)=0$, i.e., when $\nu$ is supported on a set of
  $\mu$-measure $0$ only.  This leads to a boundary map for which
  $\ker \Gamma$ is dense in $\HS=\Lsqr{X,\mu}$, i.e., to an
  \emph{ordinary} boundary pair.
\end{example}
%----------------------------------------------------------------------

%----------------------------------------------------------------------
\begin{definition}
  \label{def:dsol}
  Let $\LS^1$ be the orthogonal complement of $\ker \Gamma$ in
  $\HS^1$.  We call the inverse
  \begin{equation*}
    \map {\SDir :=(\Gamma\restr{\LS^1})^{-1}}
            {\HSaux^{1/2}}{\LS^1 \subset \HS^1}
  \end{equation*}
  of the bijective map $\map \Gamma {\LS^1}{\HSaux^{1/2}:=\ran
    \Gamma}$ the \emph{(weak) Dirichlet solution map} (at the point
  $z=-1$).
\end{definition}
%----------------------------------------------------------------------

Clearly, $h=\SDir\phi$ with $\phi \in \HSaux^{1/2}$ is the \emph{weak
  solution of the Dirichlet problem}, i.e.,
\begin{equation*}
  (\qf h + \qf 1)(h,f)= \qf h (h,f) + \iprod h f
  = 0 \quad \forall f \in \HS^{1,\Dir},
  \qquad
  \Gamma h=\phi.
\end{equation*}

The Dirichlet solution operator $\SDir$ allows us to define a natural norm
on the range $\HSaux^{1/2}$ of $\Gamma$, namely we
set
\begin{equation}
  \label{eq:norm.1/2}
  \norm[\HSaux^{1/2}] \phi
  := \norm[\HS^1] {\SDir\phi},
\end{equation}
i.e.\ the norm of the boundary element $\phi$ is given by the
$\HS^1$-norm of its (weak) Dirichlet solution.  Clearly, the operator
$\map \SDir {\HSaux^{1/2}}{\HS^1}$ and its left inverse $\map \Gamma
{\LS^1} {\HSaux^{1/2}}$ are isometric.  In particular, $\HSaux^{1/2}$
is itself a Hilbert space (with its inner product induced by
$\norm[\HSaux^{1/2}] \cdot$).  Moreover, the natural inclusion
$\HSaux^{1/2} \hookrightarrow \HSaux$ is bounded, since
\begin{equation}
  \label{eq:hsaux.0.12}
 \norm[\HSaux] \phi 
 = \norm[\HSaux] {\Gamma \SDir \phi}
 \le \norm[1 \to 0]{\Gamma} \norm[\HS^1] {\SDir \phi}
  = \norm[1 \to 0]{\Gamma} \norm[\HSaux^{1/2}] \phi,
\end{equation}
where $\norm[1 \to 0] \Gamma$ is the norm of $\Gamma$ as operator
$\map \Gamma {\HS^1} \HSaux$.

%----------------------------------------------------------------------
\begin{proposition}
  \label{prp:s.dn.closed}
%   \addtocounter{equation}{-1} %
  Let $(\Gamma,\HSaux)$ be a boundary pair associated with $\qf h$,
  then we have:
  \begin{myenumerate}
  \item
    \label{s.closed}
    The Dirichlet solution operator $\SDir$ is closed and densely
    defined as operator in $\HSaux \to \HS^1$. Its domain is given by
    $\dom \SDir=\HSaux^{1/2}$.
  \item
    \label{l.closed}
    The quadratic form $\qf l$ defined by $\qf
    l(\phi):=\normsqr[1]{\SDir \phi}$ with $\dom \qf l=\HSaux^{1/2}$ is a
    closed quadratic form in $\HSaux$.  Moreover,
    \begin{equation}
      \label{eq:l.lower}
      \qf l(\phi)
      \ge \frac 1 {\normsqr[1 \to 0] \Gamma} \normsqr[\HSaux] \phi
    \end{equation}
    for $\phi \in \HSaux^{1/2}$.
  \end{myenumerate}
\end{proposition}
%----------------------------------------------------------------------
\begin{proof}
  \itemref{s.closed}~The operator $\SDir$ has (by definition) a bounded
  inverse, hence $\SDir$ is closed.

  \itemref{l.closed}~The lower bound on $\qf l$ follows immediately
  from~\eqref{eq:hsaux.0.12}.  In order to show that $\qf l$ is
  closed, let $(\phi_n)_n$ be a Cauchy sequence in $\HSaux^{1/2}$ with
  respect to $\qf l$, then $(\phi_n)_n$ is also a Cauchy sequence in
  $\HSaux$, hence converges in $\HSaux$ to an element $\phi \in
  \HSaux$.  Moreover, $(\SDir \phi_n)_n$ is a Cauchy sequence in
  $\HS^1$, hence also convergent to $h \in \HS^1$. Since $\SDir$ is
  closed it follows that $\phi \in \dom \SDir = \HSaux^{1/2}$ and
  $\SDir\phi=h$, i.e., $\phi \in \dom \qf l$ and $\qf l(\phi_n - \phi)
  \to 0$.
\end{proof}
%----------------------------------------------------------------------

Let us now associate a natural operator $\Lambda$ to a boundary pair
$(\Gamma,\HSaux)$.  It turns out (cf.~\eqref{eq:why.dtn2}) that
$\Lambda$ is the \DtN\ operator, i.e., $\Lambda \phi$ associates to a
suitable boundary value $\phi$ the ``normal derivative'' of the
associated solution of the Dirichlet problem $h=S \phi$.

%----------------------------------------------------------------------
\begin{definition}
  \label{def:dn}
  Let $\Lambda$ be the operator associated with the quadratic form
  $\qf l$.  Then $\Lambda$ is called the \emph{\DtN\ operator} (at
  the point $z=-1$) associated with the boundary map $\Gamma$ and the
  quadratic form $\qf h$.  We denote by $\HSaux^k$ the natural scale
  of Hilbert spaces associated with the self-adjoint operator
  $\Lambda$, i.e.\ we set
  \begin{equation*}
    \HSaux^k := \dom \Lambda^k, \qquad
    \norm[k] \phi := \norm[\HSaux]{\Lambda^k \phi}.
  \end{equation*}
\end{definition}
%----------------------------------------------------------------------
Note that $\normsqr[1/2] \phi = \normsqr[\HSaux]{\Lambda^{1/2} \phi} =
\qf l(\phi) = \normsqr[\HSaux^{1/2}] \phi$, i.e.\ the setting is
compatible with our previously defined norm in~\eqref{eq:norm.1/2}.
The exponents in the scale of Hilbert spaces $\HS^k$ and $\HSaux^k$
will be consistent with the regularity order of Sobolev spaces in our
main examples in \Sec{lapl.mfd}, a boundary pair associated with a
Laplacian on a manifold with (smooth) boundary.

In the following theorem, we denote the adjoints of $\map
\Gamma{\HS^1}\HSaux$ w.r.t.\ the inner products in $\HS^1$ and
$\HSaux$ by $\Gamma^\oneadj$.  Similarly, the adjoint of the operator
$\SDir$ viewed as (possibly unbounded) operator from $\HSaux$ into
$\HS^1$ with domain $\HSaux^{1/2}$ is denoted by $\SDir^\adjone$.  It
is easy to see that $\map {\Gamma^\oneadj=R \Gamma^*}\HSaux {\HS^1}$
and $\map {\SDir^\adjone=\SDir^*(\wHNeu+1)}{\HS^1}\HSaux$, where
$\map{\Gamma^*}{\HS^{-1}}\HSaux$ and $\map{\SDir^*}\HSaux{\HS^{-1}}$
are the duals with respect to the pairing $\map{\iprod[-1,1] \cdot
  \cdot}{\HS^{-1} \times \HS^1}\C$.

%----------------------------------------------------------------------
\begin{theorem}
  \label{thm:dn}
%   \addtocounter{equation}{-1} %
  Let $\Gamma$ be a boundary map associated with $\qf h$.
  \begin{myenumerate}
  \item 
    \label{dn.i}
    We have $(\HSaux^1={})\dom \Lambda = \dom \SDir^\adjone \SDir$
    and
    \begin{equation}
      \label{eq:dn.-1}
      \Lambda = \SDir^\adjone \SDir
      \ge \frac 1 {\normsqr[1 \to 0] \Gamma}.
    \end{equation}
    In particular, $\Lambda^{-1}=\Gamma \Gamma^\oneadj$ exists and is
    a bounded operator in $\HSaux$ with norm bounded by $\normsqr[1
    \to 0] \Gamma$.

  \item
    \label{dn.ii}
    We have $\normsqr[1 \to 0] \Gamma = 1/\inf \spec \Lambda$; in
    particular, the lower bounds in~\eqref{eq:l.lower}
    and~\eqref{eq:dn.-1} are optimal.

  \item
    \label{dn.iii}
    The boundary pair is unbounded iff $\Lambda$ is
    unbounded.
  \end{myenumerate}
\end{theorem}
%----------------------------------------------------------------------
\begin{proof}
  \itemref{dn.i}~The lower bound on $\Lambda$ follows
  from~\eqref{eq:l.lower}.  Moreover, by definition of the
  associated operator (see e.g.~\cite[Thm.~VI.2.1]{kato:66}) $\phi \in
  \dom \Lambda$ iff
  \begin{equation*}
    \qf l(\cdot,\phi)
    =\iprod[\HSaux^{1/2}] \cdot \phi
    =\iprod[\HS^1]{\SDir \cdot}{\SDir\phi}
  \end{equation*}
  extends to a bounded functional $\HSaux \to \C$, i.e., iff $\SDir \phi
  \in \dom \SDir^\adjone$.  Moreover,
  \begin{equation*}
    \iprod[\HSaux] \phi {\Lambda \phi}
    = \iprod[\HSaux^{1/2}] \phi \phi
    = \iprod[\HS^1] {\SDir \phi} {\SDir \phi}
    = \iprod[\HSaux] \phi {\SDir^\adjone \SDir \phi}
  \end{equation*}
  for $\phi \in \dom \Lambda$.  Since $\SDir$ is closed, densely
  defined and $\map {\SDir^{-1}=\Gamma}{\LS^1}\HSaux$ is bounded, it
  follows that $\SDir^\adjone$ is invertible and
  $(\SDir^\adjone)^{-1}=\Gamma^\oneadj$
  (cf.~\cite[Thm.~III.5.30]{kato:66}), hence $\Lambda^{-1}=\Gamma
  \Gamma^\oneadj$.

  \itemref{dn.ii}~From~\eqref{eq:dn.-1} we conclude immediately the
  inequality ``$\ge$''.  For the inequality ``$\le$'', note that there
  is a sequence $h_n \in \HS^1$ such that $\norm[\HS^1]{h_n}=1$ and
  $\norm{\Gamma h_n} \to \norm[1 \to 0] \Gamma$.  Moreover, we can
  assume that $h_n \in \LS^1$, since the component in $\ker
  \Gamma=\HS^{1,\Dir}$ does not contribute to the norm of $\Gamma$.
  Let $\phi_n := \Gamma h_n$, then we have
  \begin{equation*}
    \frac {\qf l(\phi_n)}{\normsqr{\phi_n}}
    = \frac {\normsqr[\HS^1] {h_n}}{\normsqr{\Gamma h_n}}
    \to \frac 1 {\normsqr[1 \to 0] \Gamma},
  \end{equation*}
  hence $\inf \spec \Lambda \le 1/\normsqr[1 \to 0] \Gamma$ by the
  variational characterisation of the spectrum of $\Lambda$.

  \itemref{dn.iii}~Assume that $\ran \Gamma=\HSaux$.  Since $\map {\Gamma
    \restr{\LS^1}}{\LS^1} \HSaux$ is bounded and bijective, its
  inverse $\SDir$ is bounded as well by the open mapping theorem.
  Hence $\Lambda= \SDir^\adjone \SDir$ is bounded.  On the other
  hand, if $\Lambda$ is bounded, then $\qf
  l(\phi)=\normsqr[\HS^1]{\SDir\phi}$ is a bounded and everywhere
  defined quadratic form.  In particular, $\map \SDir \HSaux {\HS^1}$
  is everywhere defined and bounded.  For $\phi \in \HSaux$ we then
  have $\phi = \Gamma \SDir \phi \in \ran \Gamma$, i.e., $\ran \Gamma
  = \HSaux$.
\end{proof}
%----------------------------------------------------------------------
A similar proof as the one for \Thmenum{dn}{dn.iii} shows that the
boundary is unbounded iff $\Lambda(z)$ is unbounded for some (any) $z
\in \C \setminus \spec \HDir$, see \Sec{dtn.z} for the definition of
$\Lambda(z)$.

For a \emph{bounded} boundary pair, the scale of Hilbert spaces
$\HSaux^k$ consists of one vector space only, and all norms are
equivalent, i.e.\
\begin{equation}
  \label{eq:scale.aux.eq}
  \norm[1 \to 0] \Gamma^{-2k} \norm[\HSaux] \phi
  \le \norm[\HSaux^k] \phi
  \le \norm \Lambda^k \norm[\HSaux] \phi.
\end{equation}

%----------------------------------------------------------------------
\subsection{The Dirichlet solution operator at arbitrary points}
%----------------------------------------------------------------------

Let us now extend the Dirichlet solution operator to arbitrary
spectral points $z \in \C \setminus \spec \HDir$.
%----------------------------------------------------------------------
\begin{definition}
  \label{def:weak.dsol.z}
  Let $z \in \C$.
  \begin{myenumerate}
  \item We call
    \begin{equation*}
      \LS^1(z)
      := \bigset{h \in \HS^1} 
            {\qf h(h,f) - z \iprod[\HS] h f = 0 \quad
              \forall \; f \in \HS^{1,\Dir}}
    \end{equation*}
    the set of \emph{weak solutions} in $z \in \C$ (with respect to
    the boundary pair $(\Gamma,\HSaux)$ and the quadratic form $\qf
    h$).
  \item Let $\phi \in \HSaux^{1/2}$.  We say that $h$ is a \emph{weak
      solution of the Dirichlet problem} at the point $z$, if $h \in
    \LS^1(z)$ and $\Gamma h = \phi$.
  \end{myenumerate}
\end{definition}
%----------------------------------------------------------------------
Note that $\LS^1(-1)=\LS^1$ (see \Def{dsol}).  Moreover, it is easy
to see that $\LS^1(z)$ is a closed subspace of $\HS^1$.

%----------------------------------------------------------------------
\begin{proposition}
%   \addtocounter{equation}{-1} %
  \label{prp:wsol.dec.z}
  Let $z \in \C \setminus \spec \HDir$.
  \begin{subequations}
    \begin{myenumerate}
    \item
      \label{wsol.dec.z.i}
      Let $h_1, h_2 \in \LS^1(z)$ be two weak solutions of the same
      Dirichlet problem $\Gamma h_1 =\Gamma h_2$, then $h_1=h_2$.

    \item
      \label{wsol.dec.z.ii}
      The spaces $\HS^{1,\Dir}$ and $\LS^1(z)$ are closed as subspaces
      of $\HS^1$ and we have the decomposition
      \begin{equation}
        \label{eq:dec.ker.sol}
        \HS^1 = \HS^{1,\Dir} \dplus \LS^1(z),
      \end{equation}
      (where ``$\dplus$'' denotes the direct topological sum, see the
      end of \Sec{prelim}) and
      \begin{equation}
        \label{eq:dec.z.proj}
        P(z) g := S \Gamma g + (z+1)\RDir(z) S \Gamma g
      \end{equation}
      is the projection of $g$ onto $\LS^1(z)$ with respect to the
      above decomposition.  The sum is orthogonal if $z=-1$.
    \end{myenumerate}
  \end{subequations}
\end{proposition}
%----------------------------------------------------------------------
The choice of $P(z)$ becomes more clear with
\Thmenum{dsol.z}{dsol.z.iii} (choosing $w=-1$).
%----------------------------------------------------------------------
\begin{proof}
  \itemref{wsol.dec.z.i}~Assume that $h_1$ and $h_2$ are two solutions of
  the Dirichlet problem with $\Gamma h_1=\Gamma h_2$.  Then $h:= h_1 -
  h_2 \in \HS^{1,\Dir} \cap \LS^1(z)$.  Since $\qf h^\Dir - z \qf 1$
  is non-degenerative on $\HS^{1,\Dir} \times \HS^{1,\Dir}$ for $z
  \notin \spec \HDir$, we conclude $h=0$.

  \itemref{wsol.dec.z.ii}~ The space $\HS^{1,\Dir}=\ker \Gamma$ is
  closed since $\Gamma$ is bounded, and $\LS^1(z)$ is easily seen to
  be closed as subspace of $\HS^1$, too.  Moreover, for $g \in \HS^1$,
  we set $g_z = P(z) g = h + (z+1)\RDir(z) h$ with $h := S \Gamma g$
  and $g^\Dir := g - g_z$.  Then $\Gamma g^\Dir = \Gamma g - \Gamma h
  = 0$ and
  \begin{align*}
    (\qf h - z \qf 1)(g_z,f)
    &= (\qf h - z \qf 1)(h + (z+1) \RDir(z) h, f)\\
    &= (\qf h +1)(h, f)
    - (z+1) \iprod h f
    + (z+1) (\qf h^\Dir -z)( \RDir(z) h, f)
    = 0
   \end{align*}
   for $f \in \HS^{1,\Dir}$.  Here, the first term vanishes since
   $\LS^1$ and $\HS^{1,\Dir}$ are orthogonal in $\HS^1$; and the
   second and third term cancel each other. Therefore, we have shown
   that $g = g^\Dir + g_z \in \HS^{1,\Dir}+\LS^1(z)$.  Finally, the
   sum is direct, since $\LS^1(z) \cap \HS^{1,\Dir}=\{0\}$ by the
   uniqueness of the weak Dirichlet solution,
   cf.~\eqref{wsol.dec.z.i}.
\end{proof}
%----------------------------------------------------------------------

We now define a ``solution'' operator $\SDir(z)$ as the inverse of the
boundary map $\Gamma$, i.e., $h=\SDir(z)\phi$ is the unique solution
of the weak Dirichlet problem $h \in \LS^1(z)$ and $\Gamma h=\phi$ for
$z \notin \spec \HDir$.

%----------------------------------------------------------------------
\begin{definition}
  \label{def:dsol.z}
  Let $\map{\SDir(z)} {\HSaux^{1/2}} {\HS^1}$ be given by
  \begin{equation*}
    \map {\SDir(z) := (\Gamma \restr{\LS^1(z)})^{-1}}
    {\HSaux^{1/2}} {\LS^1(z) \subset \HS^1}.
  \end{equation*}
\end{definition}
%----------------------------------------------------------------------

Let us now relate the Dirichlet solution operator in different points
$z, w \in \C \setminus \spec \HDir$.  For $f \in \HS$ we set
\begin{equation}
  \label{eq:def.dsol.u}
  U(z,w) f := f + (z-w) \RDir (z) f.
\end{equation}
Recall the convention $\RDir(z) f = (\HDir-z)^{-1} f^\Dir \oplus 0$ if
$f=f^\Dir \oplus f^\orth \in \HS^{0,\Dir} \oplus
(\HS^{0,\Dir})^\orth=\HS$ if $\ker \Gamma$ is not dense in $\HS$ (see
\Defenum{bd2}{bd2.dir}), so that we can also write
\begin{equation*}
  \map{U(z,w)= (\HDir - w) \RDir(z) \oplus \id_{(\HS^{0,\Dir})^\orth}}
  \HS \HS.
\end{equation*}
It is clear from the spectral calculus that this operator
extends respectively restricts to a topological isomorphism $\map{U(z,w)^{k,\Dir
    \to k,\Dir}}{\HS^{k,\Dir}}{\HS^{k,\Dir}}$ for all $k$ with norm
bounded by $C^\Dir(z,w)$, where
\begin{equation}
  \label{eq:c.dn}
  C^\Dir(z,w)
  := \bignorm{(\HDir-w)(\HDir-z)^{-1}}
  = \sup_{\lambda \in \spec \HDir} \frac{\abs{\lambda-w}}{\abs{\lambda-z}}
  \le 1 + \frac {\abs {z-w}}{d(z,\spec \HDir)}.
\end{equation}
Moreover, an easy calculation using the resolvent equality
$\RDir(z)-\RDir(w)=(z-w)\RDir(z)\RDir(w)$ shows that $U(z,w)
U(w,z)f=f$ for $f \in \HS$, i.e., the inverse of $U(z,w)$ is $U(w,z)$.
In particular, we have
\begin{subequations}
  \begin{equation}
    \label{eq:u.zw.iso}
    \frac 1 {C^\Dir(w,z)}  \norm[\HS] f
    \le \norm[\HS]{U(z,w) f}
    \le C^\Dir(z,w) \norm[\HS] f
  \end{equation}
  for all $f \in \HS$.  If $z,w$ are both real, then 
  \begin{equation}
    \label{eq:u.zw.iso2}
    \frac 1 {C^\Dir(w,z)}  \normsqr[\HS] f
    \le \iprod[\HS]{U(z,w) f} f
    \le C^\Dir(z,w) \normsqr[\HS] f
  \end{equation}
  for all $f \in \HS$, as $U(z,w)$ is self-adjoint in this case.
\end{subequations}
If we restrict $U(z,w)$ to functions $f \in \HS^1$, then we also have
$U(z,w) f \in \HS^1$ since $\RDir(z) f \in \HS^{2,\Dir} \subset
\HS^1$.  Therefore, $U(z,w)$ is also an operator in $\HS^1$ denoted by
\begin{equation*}
  \map{U(z,w)^{1 \to 1}}{\HS^1}{\HS^1}.
\end{equation*}

The following theorem shows that this operator is a topological
isomorphism respecting the solution space splitting; and therefore
also relates solution operators at different points:
%----------------------------------------------------------------------
\begin{theorem}
%%   \addtocounter{equation}{-1} %
  \label{thm:dsol.z}
  Let $z, w \in \C \setminus \spec \HDir$.
  \begin{subequations}
    \begin{myenumerate}
    % \item
    %   \label{dsol.u.i}
    \item
      \label{dsol.z.iii}
      We have $U(z,w)^{1 \to 1}\SDir(w) = \SDir(z)$ or equivalently
      \begin{equation}
        \label{eq:dsol.zw}
        \map{\SDir(z) - \SDir(w)
          = (z-w) \RDir(z)^{1\to 1} \SDir(w)
          = (z-w)  \RDir(w)^{1\to 1}\SDir(z)} {\HSaux^{1/2}}{\HS^1}.
      \end{equation}
      Moreover, the projection onto $\LS^1(z)$ in $\HS^1=\HS^{1,\Dir}
      \dplus \LS^1(z)$ is given by $P(z)=\SDir(z) \Gamma$.
    \item
      \label{dsol.u.ii}
      The map
      \begin{equation*}
        \map{U(z,w)^{1 \to 1}} {\HS^{1,\Dir} \dplus \LS^1(w)}
        {\HS^{1,\Dir} \dplus \LS^1(z)}
      \end{equation*}
      is a topological isomorphism (extending $U(z,w)^{1,\Dir \to
        1,\Dir}$), respecting the splitting, and its inverse
      is $U(w,z)^{1 \to 1}$. Moreover, $U(z,w)^{1 \to 1}$ has norm bounded by 
      \begin{equation}
        \label{eq:c1.dn}
        C^1(z,w) := \norm[1 \to 1]{U(z,w)^{1 \to 1}}
        \le 1 + \abs{z-w} 
          \sup_{\lambda \in \spec \HDir}
           \frac{(\lambda+1)^{1/2}}{\abs{\lambda-z}}.
      \end{equation}

    \item
      \label{dsol.z.iia}
      The solution operator $\map{\SDir(z)}{\HSaux^{1/2}}{\LS^1(z)}$
      is a topological isomorphism with left inverse given by
      $\Gamma$.  Moreover,
      \begin{equation}
        \label{eq:dsol.z.norm}
        \frac 1 {C^1(-1,z)} \norm[\HSaux^{1/2}] \phi
        \le \norm[\HS^1]{\SDir(z) \phi} \le 
        C^1(z,-1) \norm[\HSaux^{1/2}] \phi
      \end{equation}
      with $C^1(z,w)$ defined in~\eqref{eq:c1.dn}.

    \item
      \label{dsol.z.v}
      The solution operator as function $z \mapsto S(z)$ is
      holomorphic and the $k$-th derivative is given by
      $\map{S^{(k)}(z)=k!\RDir(z)^k\SDir(z)}{\HSaux^{1/2}}{\HS^{2k,\Dir}
        \hookrightarrow \HS^1}$ ($k \ge 1$).
    \end{myenumerate}
  \end{subequations}
\end{theorem}
%----------------------------------------------------------------------
\begin{proof}
  \itemref{dsol.z.iii}~Let $f:= \SDir(w) \phi \in \LS^1(w)$ and $g :=
  U(z,w)^{1 \to 1} f = f + (z-w) \RDir(z)f$.  We have to show that $g$
  is the weak solution of the Dirichlet problem in $z$.  The fact that
  $\Gamma g= \Gamma f = \phi$ is obvious since $\Gamma \RDir(z)=0$.
  It remains to show that $g \in \LS^1(z)$.  Let $u \in \HS^{1,\Dir}$,
  then
  \begin{equation*}
    (\qf h - z \qf 1)(g,u)
    = (\qf h - z \qf 1)(f,u)
    + (z-w) (\qf h - z \qf 1)(\RDir(z)f,u).
  \end{equation*}
  But the latter summand equals $(z-w)\iprod f u$, so that $(\qf h - z
  \qf 1)(g,u)= (\qf h - w \qf 1)(f,u)=0$ since $f \in \LS^1(w)$.  In
  particular, we have shown that $g \in \LS^1(z)$.  Moreover,
  $P(z)=\SDir(z) \Gamma$ follows from~\eqref{eq:dsol.zw} with $w=-1$
  (see~\eqref{eq:dec.z.proj}).

  \itemref{dsol.u.ii} We have $U(z,w) U(w,z)f=f$ for $f \in \HS^1$ since
  this equality is already true for $f \in \HS$.  The norm bound can
  be seen by the estimate
  \begin{align*}
    \norm[1 \to 1]{U(z,w)^{1 \to 1}}
    &= \bignorm{1 + (z-w)(\HNeu+1)^{1/2} 
                   \RDir(z)(\HNeu+1)^{-1/2}}\\
    &\le \bignorm{1+ (z-w) (\HDir+1)^{1/2}\RDir(z)},
  \end{align*}
  and this norm bound shows that $U(z,w)^{1 \to 1}$ is a topological
  isomorphism from $\HS^1$ onto $\HS^1$ with inverse given by $U(z,w)^{1 \to
    1}$.  That $U(z,w)^{1\to1}$ respects the splitting actually can be
  seen as follows: From~\eqref{dsol.z.iii} we know that $U(z,w)^{1 \to
    1} \LS^1(w) \subset \LS^1(z)$, and by the definition of $U(z,w)^{1
    \to 1}$, we have $U(z,w)^{1 \to 1} \HS^{1,\Dir}\subset
  \HS^{1,\Dir}$.

  \itemref{dsol.z.iia}~That $\SDir(z)$ is a topological isomorphism
  follows already from the fact that $\Gamma$ restricted as map
  $\LS^1(z) \to \HSaux^{1/2}$ is bounded and bijective.  Finally, the
  norm bounds on $\SDir(z)\phi$ follow easily from
  $\SDir(z)=U(z,-1)^{1\to 1} \SDir$.

  \itemref{dsol.z.v}~The formula for the derivative follows immediately
  from~\eqref{eq:dsol.zw}.
\end{proof}
%----------------------------------------------------------------------

%----------------------------------------------------------------------
\subsection{The \DtN\ form at arbitrary points}
\label{sec:dtn.z}
%----------------------------------------------------------------------
Let us now define a sesquilinear form which will be associated with
the \DtN\ operator at $z \in \C \setminus \spec \HDir$.  We have seen
in~\eqref{eq:why.dtn2} that this form and its associated operator is
indeed what we expect from a \DtN\ operator: Roughly speaking,
$\Lambda(z) \phi$ is the ``normal derivative'' $\Gamma'$ restricted to
the boundary of the Dirichlet solution associated with $\phi$ at $z$,
i.e., $\Lambda(z)=\Gamma'\SDir(z) \phi$.

Let us start first with what we call the \emph{\DtN\ form} later on:
%----------------------------------------------------------------------
\begin{theorem}
  \label{thm:dn.z.qf}
  Let $(\Gamma,\HSaux)$ be a boundary pair, $z, w \in \C \setminus
  \spec \HDir$, and $\lambda \in \R \setminus \spec \HDir$.
%   \addtocounter{equation}{-1} %
  \begin{subequations}
    \begin{myenumerate}
    \item
      \label{dn.z.qf.i}
      The expression
      \begin{equation}
        \label{eq:def.dn.z}
        \qf l_z(\phi,\psi):= (\qf h - z \qf 1) (\SDir(z) \phi, g),
      \end{equation}
      where $g \in \HS^1$ with $\Gamma g=\psi$, is well-defined (i.e.,
      independent of the choice of $g$) and defines a \emph{symmetric}
      family $(\qf l_z)_{z \in \C \setminus \spec \HDir}$ of
      sesquilinear forms $\map{\qf l_z}{\HSaux^{1/2} \times
        \HSaux^{1/2}} \C$ (see \Sec{prelim}).

    \item
      \label{dn.z.qf.ii}
      We have
      \begin{equation}
        \label{eq:dn.qf.zw}
        \qf l_z(\phi,\psi)-\qf l_w(\phi,\psi)
        = -(z - w) \iprod{\SDir(z) \phi} {\SDir(\conj w) \psi}
        = -(z - w) \iprod{\SDir(w) \phi} {\SDir(\conj z) \psi}.
      \end{equation}

    \item
      \label{dn.z.qf.iii}
      The sesquilinear form $\map{\qf l_z}{\HSaux^{1/2} \times
        \HSaux^{1/2}} \C$ is bounded, i.e., $\abs{\qf l_z(\phi,\psi)}
      \le L(z) \norm[1/2] \phi \norm[1/2] \psi$, where
      \begin{equation}
        \label{eq:def.m.z}
        L(z) := 1 + \abs{z+1} \norm[1/2 \to 1]{S(z)}
        \le 1 + \abs{z+1} C^1(z,-1).
      \end{equation}

    \item
      \label{dn.z.qf.iv}
      The quadratic form $\map{\qf l_\lambda}{\HSaux^{1/2}} \C$ (see
      \Sec{prelim}) is non-negative for $\lambda \le 0$ and closed for
      $\lambda<0$ as form in $\HSaux$ with $\dom \qf
      l_\lambda=\HSaux^{1/2}$.

    \item
      \label{dn.z.qf.vi}
      If $\lambda_1 \le \lambda_2$ and $\lambda_1, \lambda_2 \notin
      \spec \HDir$ then $\qf l_{\lambda_1} \ge \qf l_{\lambda_2}$.
    
    \item
      \label{dn.z.qf.v}
      % was \label{dn.z.iii}
      The \DtN\ form is holomorphic in $z$, i.e., $z \mapsto \qf
      l_z(\phi,\psi)$ depends holomorphically on $z \in \C \setminus
      \spec \HDir$ for all $\phi,\psi \in \HSaux^{1/2}$.  Its
      derivatives (denoted by $\qf l_z'$ and $\qf l_z^{(k)}$) are
      bounded sesquilinear forms on $\HSaux^{1/2} \times \HSaux^{1/2}$
      given by
      \begin{equation*}
        \qf l_z'(\phi,\psi)
        = -\iprod[\HS]  {\SDir(z)\phi}{\SDir(\conj z)\psi}
        \qquadtext{resp.}
        \qf l_z^{(k)}(\phi,\psi)
        = - k! \bigiprod[\HS] {\RDir(z))^{k-1}\SDir(z)\phi}
                             {\SDir(\conj z)\psi}.
      \end{equation*}
      In particular, if $\lambda \in \R \setminus \spec \HDir$ then
      $\qf l_\lambda' \le 0$.

    \item
      \label{dn.z.qf.vii}
      % was \label{dn.z.iv}
      We have (see \Sec{prelim} for the notion $\Im \qf l_z$)
      \begin{equation*}
        (\Im \qf l_z)(\phi) = -(\Im z) \normsqr[\HS] {\SDir(z) \phi} \le 0
      \end{equation*}
      for $\phi \in \HSaux^{1/2}$ provided $\Im z \ge 0$.  In
      particular, $(-\qf l_z)_z$ is a form-valued Nevanlinna function.
    \end{myenumerate}
  \end{subequations}
\end{theorem}
%----------------------------------------------------------------------
\begin{proof}
  \itemref{dn.z.qf.i}~Let $g_1, g_2 \in \HS^1$ such that $\Gamma
  g_1=\Gamma g_2 =\phi$.  Then $g_1 - g_2 \in \HS^{1,\Dir}$, and
  \begin{equation*}
    (\qf h - z \qf 1) (\SDir(z) \phi, g_1-g_2)=0
  \end{equation*}
  since $\SDir(z) \phi \in \LS^1(z)$ is a weak solution by
  \Defs{weak.dsol.z}{dsol.z}.  The symmetry of the family is obvious.

  \itemref{dn.z.qf.ii}~Choosing $g=\SDir(\conj w) \phi$ and
  using~\eqref{eq:dsol.zw}, we have
  \begin{align*}
    \qf l_z(\phi,\psi)
    &= (\qf h - z\qf 1)(\SDir(z) \phi, g)\\
    &= (\qf h - w\qf 1)(\SDir(w) \phi, g)
      +  (z-w) (\qf h - w \qf 1) (\RDir(z) \SDir(w) \phi, g)
      - (z-w) \iprod {\SDir(z) \phi} g\\
    &= \qf l_w(\phi) - (z-w) \iprod {\SDir(z) \phi} g.
  \end{align*}
  Note that by this choice of $g$, the middle term in the second line
  vanishes (by definition of $\LS^1(w)$).

  For~\itemref{dn.z.qf.iii}, we set $w=-1$ in~\eqref{eq:dn.qf.zw} and
  obtain
  \begin{equation*}
    \abs{\qf l_z(\phi)}
    \le 
    \normsqr[1/2] {\phi} 
     + \abs{z+1} \norm[\HS]{\SDir(z) \phi}\norm[\HS]{\SDir \phi}.
  \end{equation*}
  Using the estimates $\norm[\HS]{\SDir(z) \phi} \le
  \norm[\HS^1]{\SDir(z)\phi}$ and $\norm[\HS]{\SDir \phi} \le
  \norm[\HS^1]{\SDir \phi} = \norm[1/2] \phi$, we obtain the desired
  bound.

  \itemref{dn.z.qf.iv} Obviously, $\qf l_\lambda (\phi) \ge 0$ as
  $(\qf h-\lambda \qf 1)(f) \ge 0$ for $\lambda \le 0$.  Moreover, if
  $\lambda<0$, then
  \begin{equation*}
    \min \{1,-\lambda \}(\qf h + \qf 1)(f) 
    \le (\qf h - \lambda \qf 1)(f)
    \le \max \{1,-\lambda \}(\qf h + \qf 1)(f) ,
  \end{equation*}
  hence $\qf l_\lambda(\phi)$ defines a squared norm equivalent with
  $\qf l_{-1}(\phi)=\normsqr[\HS^1]{\SDir \phi}$, hence $\qf
  l_\lambda$ is closed on $\HSaux^{1/2}=\dom \qf l_{-1}$.

  \itemref{dn.z.qf.vi} We have
  \begin{align*}
    (\qf l_{\lambda_1}-\qf l_{\lambda_2})(\phi)
    &= (\lambda_2 - \lambda_1) 
      \iprod[\HS]{\SDir(\lambda_1)\phi}{\SDir(\lambda_2)\phi}\\
    &= (\lambda_2 - \lambda_1) 
      \iprod[\HS]{\SDir(\lambda_1)\phi}
      {U(\lambda_2,\lambda_1)\SDir(\lambda_1)\phi}\\
    &\ge \frac{\lambda_2 - \lambda_1}{C^\Dir(\lambda_1,\lambda_2)} 
      \normsqr[\HS]{\SDir(\lambda_1)\phi}
  \end{align*}
  using~\eqref{eq:dn.qf.zw} for the first equality,
  \Thmenum{dsol.z}{dsol.z.iii} for the second and~\eqref{eq:u.zw.iso2}
  for the last inequality.  Since $\lambda_2-\lambda_1 \ge 0$, the
  result follows.

  \itemref{dn.z.qf.v}~and~\itemref{dn.z.qf.vii} follow
  from~\eqref{eq:dn.qf.zw} by letting $w \to z$ and $w=\conj z$,
  respectively.
\end{proof}
%----------------------------------------------------------------------

%----------------------------------------------------------------------
\begin{definition}
  \label{def:dn.z}
  We call the sesquilinear form $\qf l_z$ defined
  in~\eqref{eq:def.dn.z} the \emph{\DtN\ form at $z \in \C \setminus
    \spec \HDir$} associated with the boundary pair $(\Gamma,\HSaux)$
  and the quadratic form $\qf h$.

  By the previous proposition, $\qf l_z$ is bounded as form on
  $\HSaux^{1/2} \times \HSaux^{1/2}$.  We therefore can define an
  operator
  \begin{equation*}
    \map {\wLambda(z)} 
    {\HSaux^{1/2}} {\HSaux^{-1/2}},
    \qquad
    \phi \mapsto \qf l_z(\phi,\cdot),
  \end{equation*}
  called the \emph{weak \DtN\ operator}, i.e.,
  \begin{equation*}
    \iprod[-1/2,1/2] {\wLambda(z) \phi} \psi
    = \qf l_z(\phi, \psi)
    = (\qf h - z \qf 1)(\SDir(z) \phi, \SDir \psi),
  \end{equation*}
  or $\wLambda(z) = \SDir^*(\wHNeu-z)\SDir(z) \colon \HSaux^{1/2} \to
  \HS^1 \to \HS^{-1} \to \HSaux^{-1/2}$ in the scale of Hilbert
  spaces.
\end{definition}
%----------------------------------------------------------------------

We always have an \emph{associated operator} with $\qf l_z$, defined
by
\begin{equation}
  \label{eq:dn.ass.op}
  \dom \Lambda(z)
  := \bigset{\phi \in \HSaux^{1/2}}
  {\exists \eta \in \HSaux \;
    \forall \psi \in \HSaux^{1/2} \colon \;
    \qf l_z(\phi,\psi)= \iprod[\HSaux] \eta \psi}
\end{equation}
and $\Lambda(z) \phi := \eta$, and that the latter definition is
well-defined (since $\HSaux^{1/2}=\ran \Gamma$ is dense in $\HSaux$ by
definition of a boundary pair).  We call this operator the
\emph{(strong) \DtN\ operator} associated with a boundary pair.
Actually, $\Lambda(z)$ is the restriction of $\wLambda(z)$ to those
$\phi$ such that $\wLambda(z) \phi \in \HSaux$.  It is easily seen
that $\phi \in \dom \Lambda(z)$ can equivalently expressed by
\begin{equation}
  \label{eq:dn.ass.op2}
  \exists u \in \HS^1, \Gamma u = \phi \;\; \exists \eta \in \HSaux \;
    \forall v \in \HSaux^{1/2} \colon \quad
    (\qf h - z \qf 1)(u,v) = \iprod[\HSaux] \eta {\Gamma v}
\end{equation}
without referring to the solution operator (see also the text
after~\eqref{eq:g-ell} for this Arendt and ter Elst
approach in~\cite{arendt-ter-elst:12}).

We use the notation $\wLambda(z)$ when we want to stress that we mean
the \emph{weak} \DtN\ operator, and not the \emph{strong} \DtN\
operator.

Note that $\qf l_{-1}=\qf l$ where $\qf l$ is defined in
\Prpenum{s.dn.closed}{l.closed} and that $\Lambda$ is the operator
associated with $\qf l$ by \Def{dn}, therefore $\Lambda=\Lambda(-1)$.

%----------------------------------------------------------------------
\begin{remark}
  \label{rem:dtn-qf}
  \indent
  \begin{myenumerate}
  \item
%2    \label{dtn-qf.i}
    The \DtN\ form $\qf l_z$ with domain $\dom \qf l_z=\HSaux^{1/2}$ is
    closed for $z<0$ by \Thmenum{dn.z.qf}{dn.z.qf.iv}.  We will see in
    \Thm{ell.bd2.impl} that for so-called \emph{elliptically regular}
    boundary pairs, $\qf l_z$ is a closed (and sectorial) form for
    \emph{all} $z \in \C \setminus \spec \HDir$; in particular,
    $\Lambda(z)$ is the associated operator and hence closed.
    Moreover, the domain of $\Lambda(z)$ is independent of $z$ in this
    case.

  \item
    \label{dtn-qf.ii}
    In general, the forms $\qf l_z$ are not closed or sectorial resp.\
    bounded from below: In \Ex{dtn.unbdd.below}, we construct an
    example where $\qf l_z$ is unbounded from both sides for all $z>0$
    not in $\spec \HDir$.  In this situation, we cannot speak of the
    \emph{closure} of a form (in the classical sense) anymore, and
    more advanced techniques are necessary,
    see~\cite{mcintosh:69,mcintosh:70} or~\cite{gkmv:13} and
    references therein: A sesquilinear form $\map{\qf a}{\HSaux_0
      \times \HSaux_0} \C$ is said to be \emph{$0$-closed} iff $\qf a$
    is non-degenerative ($\qf a(\phi,\cdot)=0$ implies $\phi=0$ and
    $\qf a(\cdot,\psi)=0$ implies $\psi=0$) and if $\HSaux_0 \subset
    \HSaux$ is densely and continuously embedded.

    We can show here that $\qf l_z$ is non-degenerative if $z \notin
    \spec \HNeu$, an argument very similar to the proof of
    \Prpenum{dn.z.inv}{dn.z.inv.i} (using the non-degeneracy of the
    form $\qf h - z \qf 1$ as $z \notin \spec \HNeu$).  It follows
    from Theorem~3.2 of~\cite{mcintosh:69} that the associated
    operator $\Lambda(z)$ is closed and has a bounded inverse ($0
    \notin \spec{\Lambda(z)}$) (this is actually
    \Thmenum{krein1}{krein1.i0}).  We will show these facts
    independently in \Prp{dn.z.inv}.

  \item It is a priori not clear whether $\Lambda(\lambda)$ as
    operator associated with $\qf l_z$ is closed also for $\lambda \in
    \spec \HNeu$ (we need this fact in \Thm{krein2} in order to speak
    of the spectrum of $\Lambda(z)$).

    With a trick, we can show that $\Lambda(z)$ is closed for all $z
    \in \C \setminus \spec \HDir$ even for non-elliptically regular
    boundary pairs under the additional assumption that $\RNeu$ is
    compact, see \Thmenum{nd.comp}{comp.e}
  \end{myenumerate}
\end{remark}
%----------------------------------------------------------------------

We state more results on the strong \DtN\ operator later on (see
\Prp{dn.z.inv}, \Thm{ell.bd2.impl} and \Prp{bd2.b.z}).

%----------------------------------------------------------------------
\begin{remark}
  \label{rem:reconstr.bd2}
  \Thmenum{dn.z.qf}{dn.z.qf.vii} allows us to express the quadratic
  form $\qf q_z$ defined by $\qf q_z(\phi):= \normsqr[\HS]{\SDir(z)
    \phi}$ in terms of the \DtN\ form $\qf l_z$, namely, $\qf q_z= -
  \Im \qf l_z/\Im z$ for $z \in \C \setminus \R$.  This fact allows us
  to detect elliptic regularity and positivity (see \Sec{bd2.add} and
  \Def{ell.pos.intro}) from the family of \DtN\ forms. This fact is
  also useful in reconstructing a boundary pair from a given
  form-valued Nevanlinna function in the sense
  of~\cite{langer-textorius:77}.  We will treat this and related
  questions in a forthcoming publication.%INVERSE
\end{remark}
%----------------------------------------------------------------------

Let us close this subsection with the following result, needed in the
proof of \Thm{ell.bd2}
\itemref{ell.bd2.ix}~$\Rightarrow$~\itemref{ell.bd2.viii}:
%----------------------------------------------------------------------
\begin{proposition}
  \label{prp:dtn.sdir}
  Let $z \in \C \setminus (\spec \HDir \cup \spec \HNeu)$.  Then we have
  \begin{equation*}
    \map{(\wLambda(z) \Gamma \wRNeu(z))^*=\SDir(\conj z)}
    {\HSaux^{1/2}}{\HS^1}
  \end{equation*}
\end{proposition}
%----------------------------------------------------------------------
\begin{proof}
  Let $u \in \HS^{-1}$ and $\phi \in \HSaux^{1/2}$, then the result
  follows from
  \begin{align*}
    \iprod{\wLambda(z) \Gamma \wRNeu(z) u} \phi
    = \qf l_z(\Gamma \wRNeu(z) u, \phi)
    &= (\qf h - z\qf 1)(\SDir(z) \Gamma \wRNeu(z) u, 
         \SDir(\conj z) \phi)\\
    &= (\qf h - z\qf 1)(\wRNeu(z) u, \SDir(\conj z) \phi)
    = \iprod u {\SDir(\conj z) \phi}
  \end{align*}
  using the fact that $\SDir(z)\Gamma \wRNeu(z) u - \wRNeu(z)u \in
  \HS^{1,\Dir}$ and the definition of $\LS^1(z)$ for the third equality.
\end{proof}
%----------------------------------------------------------------------

%----------------------------------------------------------------------
\subsection{The \NtD\ operator}
%----------------------------------------------------------------------

Let us first show that the weak \DtN\ operator $\wLambda(z)$ is
invertible if $z$ is not in the Neumann spectrum, and that the
function $z \mapsto \Lambda(z)^{-1}$ extends continuously into the
Dirichlet spectrum $z \in \spec \HDir$.

%----------------------------------------------------------------------
\begin{proposition}
%%   \addtocounter{equation}{-1} %
  \label{prp:dn.z.inv}
  Let $z \in \C \setminus (\spec \HDir \cup \spec \HNeu)$.  Then we
  have:
  \begin{myenumerate}
  \item
    \label{dn.z.inv.i}
    The weak \DtN\ operator $\map
    {\wLambda(z)}{\HSaux^{1/2}}{\HSaux^{-1/2}}$ is bijective with
    inverse
    \begin{equation}
      \label{eq:dn.z.inv}
      \map{\wLambda(z)^{-1} = \Gamma \wRNeu(z) \Gamma^*}
      {\HSaux^{-1/2}}{\HSaux^{1/2}}.
    \end{equation}
  \item
    \label{dn.z.inv.ii}
    The operator-valued function $z \mapsto \wLambda(z)^{-1}$ extends
    into $z \in \spec \HDir$, and the value is again a bounded
    operator denoted by the same symbol $\map {\wLambda(z)^{-1}}
    {\HSaux^{-1/2}} {\HSaux^{1/2}}$.  The norm of $\wLambda(z)^{-1}$
    is bounded by $C^\Neu(z,-1)$ where $C^\Neu(z,w)$ is defined as
    in~\eqref{eq:c.dn} with $\HDir$ replaced by $\HNeu$.

  \item
    \label{dn.z.inv.iv}
    Denote by $\map{\Lambda(z)^{-1}} \HSaux \HSaux$ the operator
    $\wLambda(z)^{-1}$ restricted to $\HSaux$ and with range space
    $\HSaux$, then $\Lambda(z)^{-1}$ is the inverse of the strong \DtN\
    operator $\map{\Lambda(z)}{\dom \Lambda(z)} \HSaux$ and
    $\Lambda(z)^{-1}$ is bounded by $\normsqr[1 \to 0] \Gamma
    C^\Neu(z,-1)$.  Moreover, $\Lambda(z)$ is closed for all $z \in \C
    \setminus (\spec \HDir \cup \spec \HNeu)$ and $0 \notin \spec
    {\Lambda(z)}$.
  \end{myenumerate}
\end{proposition}
%----------------------------------------------------------------------
\begin{proof}
  \itemref{dn.z.inv.i}~Let $\phi \in \ker \wLambda(z)$, i.e., $\qf
  l_z(\phi,\eta)=0$ for all $\eta \in \HSaux^{1/2}$.  Therefore,
  \begin{equation*}
    0 = \qf l_z(\phi,\eta) = (\qf h - z \qf 1)(\SDir(z)\phi, g)
  \end{equation*}
  for all $g \in \HS^1$ with $\Gamma g = \eta$.  Since $z \notin \spec
  \HNeu$, the form $\qf h - z \qf 1$ is non-degenerative, and
  therefore $\SDir(z)\phi =0$, i.e., $\phi=0$.  In particular, we
  have shown that $\wLambda(z)$ is injective.

  For the surjectivity, let $\psi \in \HSaux^{-1/2}$.  Set $\phi:=
  \Gamma \wRNeu(z) \Gamma^* \psi$, then $\phi \in \HSaux^{1/2}$, and
  \begin{equation*}
    \iprod[-1/2,1/2]{\wLambda(z) \phi} \eta
    = \qf l_z(\phi,\eta)
    = (\qf h - z \qf 1)(\SDir(z) \Gamma \wRNeu(z) \Gamma^* \psi,g), 
  \end{equation*}
  where $g \in \HS^1$ with $\Gamma g=\eta$.  Moreover, $h=\wRNeu(z)
  \Gamma^* \psi \in \LS^1(z)$, since $(\qf h - z \qf 1)(h,f)=\iprod
  \psi {\Gamma f}=0$ for all $f \in \HS^{1,\Dir}$, hence $h=\SDir(z)
  \Gamma h$ by \Thmenum{dsol.z}{dsol.z.iii}, and we have
  \begin{equation*}
    \iprod[-1/2,1/2]{\wLambda(z) \phi} \eta
    = (\qf h - z \qf 1)(h, g)
    = \iprod[-1/2,1/2] \psi \eta,
  \end{equation*}
  i.e., we have shown that $\wLambda(z)\phi=\psi$ and $\phi=
  \Gamma \wRNeu(z) \Gamma^* \psi$, i.e., that $\wLambda(z)^{-1}
  = \Gamma \wRNeu(z) \Gamma^*$.

  \itemref{dn.z.inv.ii}~is obvious from the representation of
  $\wLambda(z)^{-1}$.

  \itemref{dn.z.inv.iv}~By~\itemref{dn.z.inv.i},
  $\map{\Lambda(z)}{\dom \Lambda(z)} \HSaux$ is bijective as a
  restriction of a bijective function and since $\dom
  \Lambda(z)=\set{\phi \in \HSaux^{1/2}}{\wLambda(z) \phi \in
    \HSaux}$.  Its inverse $\Lambda(z)^{-1}$ is bounded as map $\HSaux
  \to \HSaux$ provided $z \notin \spec \HNeu$.  The strong \DtN\
  operator $\Lambda(\lambda)$ is closed as its inverse is bounded.  In
  particular, it makes sense to speak of the spectrum of
  $\Lambda(\lambda)$, and by definition of the spectrum, if
  $\Lambda(\lambda)$ has a bounded inverse then $0 \notin \spec
  {\Lambda(\lambda)}$.
\end{proof}
%----------------------------------------------------------------------

%----------------------------------------------------------------------
\begin{definition}
  \label{def:nd.z}
  We call $\map {\Lambda(z)^{-1}} \HSaux \HSaux$ the \NtD\ operator.
\end{definition}
%----------------------------------------------------------------------
Let us now look at the \NtD\ operator at different points:
%----------------------------------------------------------------------
\begin{proposition}
  \label{prp:nd.z}
%   \addtocounter{equation}{-1} %
  Let $z,w \in \C \setminus \spec \HNeu$,
  \begin{myenumerate}
  \item
    \label{nd.z.i}
    We have
    \begin{equation}
      \label{eq:nd.zw}
      \map{\Lambda(z)^{-1}-\Lambda(w)^{-1}
        = (z - w) \Gamma \RNeu(w) (\Gamma \RNeu( \conj z))^*
        = (z - w) \Gamma \RNeu(z) (\Gamma \RNeu(\conj w))^*} 
              \HSaux \HSaux.
    \end{equation}

  \item
    \label{nd.z.iii}
    We have $\Im \Lambda(z)^{-1} := \frac 1 {2 \im} (\Lambda(z)^{-1} -
    \Lambda(\conj z)^{-1}) = (\Im z) (\Gamma \RNeu(z))(\Gamma
    \RNeu(z))^* \ge 0$.  In particular, the map
    $\map{\Lambda(\cdot)^{-1}} {\C \setminus \spec \HNeu}{\BdOp
      \HSaux}$, $z \mapsto \Lambda(z)^{-1}$ is an operator-valued
    Nevanlinna function.
  \end{myenumerate}
\end{proposition}
%----------------------------------------------------------------------
\begin{proof}
  \itemref{nd.z.i}~follows immediately from~\eqref{eq:dn.z.inv} and
  the resolvent equation.  \itemref{nd.z.iii}~is obvious
  from~\itemref{nd.z.i}.
\end{proof}
%----------------------------------------------------------------------

%----------------------------------------------------------------------
\begin{theorem}
  \label{thm:nd.comp}
% \item
%   \label{nd.z.iv}
  The following assertions are equivalent:
  \begin{myenumerate}
  \item
    \label{comp.a}
    $\map{\Lambda^{-1}}\HSaux \HSaux$ is compact,
  \item
    \label{comp.b}
    $\map{\Lambda(z)^{-1}} \HSaux \HSaux$ is compact for all $z \in \C
    \setminus \spec \HNeu$ (for some $z \in (-\infty,0] \setminus
    \spec \HNeu$)
  \item
    \label{comp.c}
    $\map \Gamma {\HS^1}\HSaux$ is compact
 \end{myenumerate}
 Assume additionally that $\RNeu$ is compact. Then any of the above
 condition is also equivalent with the following:
 \begin{myenumerate}
   \addtocounter{enumi}{3}
 \item 
   \label{comp.d}
   For all $z \in \C \setminus \spec \HDir$, there exists $a
   \ge 0$ such that $(\Lambda(z)+a)^{-1}$ is compact,
 \item
   \label{comp.e}
   For all $z \in \C \setminus \spec \HDir$, the operator $\Lambda(z)$
   is closed and has purely discrete spectrum.
 \end{myenumerate}
\end{theorem}
%----------------------------------------------------------------------
\begin{proof}
  %  \itemref{nd.z.iv},
  \itemref{comp.a}~$\Rightarrow$~\itemref{comp.b}\&\itemref{comp.c}:~We
  have the factorisation $\Lambda^{-1}= K^* K$ with $\map{K = (\Gamma
    \RNeu^{1/2})^*}\HSaux \HS$ by \Thmenum{dn}{dn.i} (or
  \Prp{dn.z.inv}).  Assume now that $\Lambda^{-1}$ is compact, then
  $K$ is compact, and therefore also $\Gamma = K^* (\HNeu+1)^{1/2}$
  (hence~\itemref{comp.c} is shown).  Moreover, $\map {(\Gamma
    \RNeu)^* = \RNeu^{1/2} K} \HSaux \HS$ is compact, too, and
  $\map{\Gamma\RNeu(z)}\HS\HSaux$ is bounded.  Hence,
  by~\eqref{eq:nd.zw}, we have
  \begin{equation*}
    \Lambda(z)^{-1}
    = \Lambda^{-1} + (z+1) (\Gamma \RNeu(z)) (\Gamma \RNeu)^*
  \end{equation*}
  which shows that $\Lambda(z)^{-1}$ is compact as operator $\HSaux
  \to \HSaux$ for any $z \in \C \setminus \spec \HNeu$.
  
  \itemref{comp.b}~$\Rightarrow$~\itemref{comp.a} is obvious if the
  statement is true for all $z \in \C \setminus \spec \HNeu$.  If it
  is only true for one $z \in (-\infty,0] \setminus \spec \HNeu$, then
  we can use a factorisation $\Lambda(z)^{-1}=K^*K$ similarly as in
  the proof of \itemref{comp.a}~$\Rightarrow$~\itemref{comp.b} to show
  the compactness of $\Lambda^{-1}$.

  \itemref{comp.c}~$\Rightarrow$~\itemref{comp.a}: The compactness of
  $\Lambda^{-1}$ follows from $\map {\Lambda^{-1}=\Gamma
    \Gamma^\oneadj} \HSaux \HSaux$ (see \Thmenum{dn}{dn.i}).

  For the assertions~\itemref{comp.d} and~\itemref{comp.e} we need
  some more notation: Denote by $\qf h_a$ the quadratic form $\qf
  h_a(u)=\qf h(u) + a \normsqr{\Gamma u}$.  It can be seen that $\qf
  h_a$ is closed, and its associated operator $\HNeu_a$ is
  non-negative; moreover $(\Gamma,\HSaux)$ is still a boundary pair
  associated with $\qf h_a$ (see \Sec{robin} for details).

  \itemref{comp.a}~$\Rightarrow$~\itemref{comp.d}: Let $z \notin \spec
  \HDir$.  By \Prp{robin.nspec}, there exists $a>0$ such that $z
  \notin \spec {\HNeu_a}$.  By assumption, $\Lambda^{-1}$ is compact,
  hence $\Lambda_a^{-1}=(\Lambda+a)^{-1} \le \Lambda^{-1}$ is also
  compact.  We can now apply
  \itemref{comp.a}~$\Rightarrow$~\itemref{comp.b} for the boundary
  pair $(\Gamma, \HSaux)$ associated with $\qf h_a$ and obtain that
  $\Lambda_a(z)^{-1}=(\Lambda(z)+a)^{-1}$ is compact.

  \itemref{comp.d}~$\Rightarrow$~\itemref{comp.a}: Set $z=-1$, then
  $\Lambda^{-1} = (1+a\Lambda^{-1})(\Lambda+a)^{-1}$ and this operator
  is compact, if $(\Lambda+a)^{-1}$ is.  The equivalence
  \itemref{comp.d}~$\Leftrightarrow$~\itemref{comp.e} is a general
  fact from operator theory.
\end{proof}
%----------------------------------------------------------------------

We want to remark that, in general, the compactness of $\RNeu$ and
$\Lambda^{-1}$ are independent of each other as the following tabular shows:

%----------------------------------------------------------------------
\begin{center}
  \begin{tabular}[h]{|c|c|c||l|l|}
    \hline
    \multicolumn{3}{|l||}{Operator is compact:}& Examples & Remarks\\
    $\RNeu$ & $\RDir$ & $\Lambda^{-1}$ & &\\
    \hline \hline
    \yes & \yes & \yes & \Sec{lapl.mfd} (compact manifold)  &\\
    \yes & \yes & \no &  \Ex{neu.comp.dtn.not} 
               (modified manifold example) &\\
%    \yes & \no  & \yes & \\ geht nicht
%    \yes & \no  & \no &  \\ geht nicht
    \no  & \yes & \yes & \Ex{dir.comp.neu.not} ($\alpha>2$) &
    Examples cannot be ell.\ reg.\\ 
    \no  & \yes & \no & \Ex{dir.comp.neu.not} ($\alpha=2$)&\\
    \no  & \no  & \yes & \Ex{cyl.comp} cylindrical manifold &\\
    \no  & \no  & \no &  \Ex{cyl.comp.mod} modified cylindrical manifold&\\
    \hline
  \end{tabular}
  \parbox{0.9\textwidth}{\yes --- yes; \no --- no.  Note that $\RNeu$
    compact implies $\RDir$ compact as $\RNeu \ge \RDir \ge
    0$. Moreover, for elliptically regular boundary pairs (see next
    section) the compactness of $\Lambda^{-1}$ and $\RDir$ implies the
    compactness of $\RNeu$ by the resolvent
    formula~\eqref{eq:krein.res}.}
\end{center}
%----------------------------------------------------------------------

%----------------------------------------------------------------------
% cccc
\section{Boundary pairs with additional properties}
\label{sec:bd2.add}
%
%----------------------------------------------------------------------

Let us now describe further properties of boundary pairs described in
terms of the Dirichlet solution operator and an associated quadratic
form.  It turns out that these properties allow us to relate the
concept of boundary pairs to other concepts such as boundary triples
(see \Sec{rel.bd3}).  Moreover, one of these properties, called
\emph{elliptic regularity}, allows us to show stronger results on
resolvent formulae and spectral characterisations (see
\Thmenum{krein.res}{krein.res.iii} and \Thm{krein2}).

%----------------------------------------------------------------------
\subsection{Elliptically regular boundary pairs}
%----------------------------------------------------------------------
\label{sec:ell.bd2}
An important case is when the solution form $\qf q$ (defined by $\qf
q(\phi):=\normsqr[\HS]{\SDir\phi}$ for $\phi \in \HSaux^{1/2}$) is
\emph{bounded} as form in $\HSaux$.  Surprisingly, this property has
already been recognised as important in a different context by Brasche
et al, and also by Arlinskii, see \Rem{eq.ell}.
 
One of the main consequences of the boundedness of $\qf q$ is that the
\DtN\ form $\qf l_z$ is sectorial, and hence, the strong \DtN\ operator
(the operator associated with the form $\qf l_z$) is closed and
sectorial.  Moreover, its domain $\dom \Lambda(z)$ is independent of
$z$ (see \Thm{ell.bd2.impl}).

%----------------------------------------------------------------------
\begin{definition}
  \label{def:ell.bd2}
  Let $(\Gamma,\HSaux)$ be a boundary pair.  We say that the boundary
  pair is \emph{elliptically regular} if the \emph{solution form} $\qf
  q$ is bounded, i.e., if there is a constant $C>0$ such that
  \begin{equation*}
    \qf q(\phi)
    = \normsqr[\HS]{\SDir \phi}
    \le C^2 \normsqr[\HSaux] \phi
  \end{equation*}
  for all $\phi \in \HSaux^{1/2}$.
\end{definition}
%----------------------------------------------------------------------

Let us first present a simple result for \emph{bounded} boundary
pairs:
%----------------------------------------------------------------------
\begin{proposition}
  \label{prp:ell.bd2.simple}
  Let $(\Gamma,\HSaux)$ be a boundary pair and let $\SDir(z)$ be the
  corresponding weak Dirichlet solution operator.  If the boundary
  pair $(\Gamma,\HSaux)$ is bounded, then it is elliptically regular.
\end{proposition}
%----------------------------------------------------------------------
\begin{proof}
  If the boundary pair is bounded, then $\Lambda$ is a bounded
  operator in $\HSaux$ by \Thmenum{dn}{dn.iii}.  Moreover,
  \begin{equation*}
    \qf q(\phi)
    = \normsqr[\HS]{\SDir \phi}
    \le \normsqr[\HS^1]{\SDir \phi}
    = \normsqr[\HSaux^{1/2}] \phi
    = \iprod[\HSaux]{\Lambda \phi} \phi
    \le \norm \Lambda \normsqr[\HSaux] \phi,
  \end{equation*}
  for $\phi \in \HSaux$, i.e., we can choose $C^2:=\norm \Lambda$.
\end{proof}
%----------------------------------------------------------------------

%----------------------------------------------------------------------
\begin{remark}
  \label{rem:ell.bd2}
  \indent
  \begin{myenumerate}
  \item Not all boundary pairs are elliptically regular, see the
    unbounded Jacobi operator examples in \Sec{jacobi} or the Zaremba
    problem in \Thm{zaremba2.bd2}.  Moreover, not all elliptic
    boundary pairs are bounded (see the manifold examples in
    \SecS{lapl.mfd}{dtn.mg}).

  \item The notion ``elliptically regular'' for boundary pairs is
    actually inspired by our basic example presented in detail in
    \Sec{lapl.mfd}, see in particular \Rem{why.ell.reg} and also
    \Rem{why.ell.reg.intro}.
  \end{myenumerate}
\end{remark}
%----------------------------------------------------------------------

The name ``elliptically regular'' is also justified by the following
definition and \Thmenum{ell.bd2}{ell.bd2.vii}:
%----------------------------------------------------------------------
\begin{definition}
  \label{def:dn.ell}
  We say that the sesquilinear form $\map{\qf l_z} {\HSaux^{1/2}
    \times \HSaux^{1/2}} \C$ is \emph{elliptic} in $\HSaux$ if there
  exist $\alpha>0$ and $\omega(z) \in \R$ such that
  \begin{equation*}
    (\Re \qf l_z)(\phi) + \omega(z) \normsqr[\HSaux] \phi
    \ge \alpha \normsqr[\HSaux^{1/2}] \phi
  \end{equation*}
  for all $\phi \in \HSaux^{1/2}$.
\end{definition}
% ----------------------------------------------------------------------

Let us now present some equivalent characterisations:
%----------------------------------------------------------------------
\begin{theorem}
  \label{thm:ell.bd2}
  Let $(\Gamma,\HSaux)$ be a boundary pair, then the following
  assertions are equivalent:
  \begin{myenumerate}
  \item 
    \label{ell.bd2.i}
    The boundary pair is elliptically regular.
  \item 
    \label{ell.bd2.iv}
    The solution form $\qf q_z$ (defined by $\qf
    q_z(\phi):=\normsqr{\SDir(z)\phi}$ for $\phi \in \HSaux^{1/2}$)
    extends to a bounded form, and is associated with a \emph{bounded}
    operator $Q(z)$ on $\HSaux$ for some (any) $z \in \C \setminus
    \spec \HDir$.
  \item
    \label{ell.bd2.v}
    The imaginary part $\Im \qf l_z := \frac 1 {2 \im} (\qf l_z - \qf
    l_z^*) \; {(=}-(\Im z) \qf q_z)$ of the \DtN\ form $\qf l_z$ is
    bounded, and hence associated with a \emph{bounded} operator
    $-(\Im z) Q(z)$ on $\HSaux$ for some (any) $z \in \C \setminus
    \R$.
  
  \item
    \label{ell.bd2.vi}
    The derivative $\qf l_\lambda'$ of the \DtN\ form is bounded, and
    hence associated with a \emph{bounded} operator on $\HSaux$ (given
    by $\Lambda'(\lambda):=-Q(\lambda)$) for some (any) $\lambda \in
    \R \setminus \spec \HDir$.

  \item
    \label{ell.bd2.vii}
    For $z \in \R$ in a neighbourhood of $-1$ (resp.\ for all $z \in
    \C \setminus \spec \HDir$), there exists $\omega(z)$ such that
    $\omega(-1)=0$ and $\limsup_{a \to -1} \frac{\omega(a)}{a+1}$ is
    finite, and the sesquilinear form $\qf l_z$ is elliptic (in the
    sense of \Def{dn.ell}) with constants $\alpha=1$ and $\omega(z)$.

  \item
    \label{ell.bd2.viii}
    We have $\Gamma(\dom \HNeu) \subset \HSaux^1 (=\dom \Lambda)$.

  \item
    \label{ell.bd2.ix}
    The operator $\Lambda \Gamma \RNeu$ maps $\HS$ into $\HSaux$ and
    is bounded as operator $\HS \to \HSaux$.
  \end{myenumerate}
\end{theorem}
%----------------------------------------------------------------------
\begin{remark}
  \label{rem:bd2.ell}
  We have another characterisation using the second boundary map
  $\Gamma'$ and some more rather obvious equivalent characterisations:
  \begin{myenumerate}
    \item 
      \label{bd3-ell}
      In \Rem{why.ell.reg.intro} we showed that if a boundary pair is
      elliptically regular, then $\Gamma'\RDir f \in \HSaux$ for all
      $f \in \HS$.  One can extend the definition of $\Gamma'$ to
      $\check \Gamma'$ such that Green's formula~\eqref{eq:green}
      remains true, but $\check \Gamma' f\in \HSaux^{-1/2}$ (hence
      $\iprod {\check \Gamma'f}{\Gamma g}$ is the dual pairing
      $\HSaux^{-1/2} \times \HSaux^{1/2}$).  Then we have the
      equivalent characterisation that $(\Gamma,\HSaux)$ is
      elliptically regular iff $\check \Gamma' u \in \HSaux$ for all
      $u \in \dom \HDir$, or, formulated in analogy
      with~\itemref{ell.bd2.viii},
      \begin{equation*}
        \check \Gamma'(\dom \HDir) \subset \HSaux.
      \end{equation*}

  \item[\ref{ell.bd2.iv}']
    \label{ell.bd2.ii}
    The weak Dirichlet solution operator $\SDir(z)$ extends to a
    bounded operator $\map{\eSDir(z)} \HSaux \HS$ for some (any) $z
    \in \C \setminus \spec \HDir$.

  \item[\ref{ell.bd2.iv}'']
    \label{ell.bd2.ii'}
    The dual $\map{\SDir(\conj z)^*}{\HS^{-1}}{\HSaux^{-1/2}}$ of the
    Dirichlet solution operator $\SDir(\conj z)$ restricts to a
    bounded operator $\HS \to \HSaux$ (denoted by $B(z)$) for some
    (any) $z \in \C \setminus \spec \HDir$.

  \item[\ref{ell.bd2.iv}''']
    \label{ell.bd2.ii''}
    There is a constant $c>0$ such that $\norm[\HSaux]{\Gamma h} \ge c
    \norm[\HS] h$ for all $h \in \LS^1$.
  \end{myenumerate}
\end{remark}
%----------------------------------------------------------------------
There are equivalent characterisations of elliptic regularity by other
authors:
%----------------------------------------------------------------------
\begin{remark}
  \label{rem:eq.ell}
  \indent
  \begin{enumerate}
  \item 
    \label{arlinskii}
    Arlinskii expressed in~\cite[Sec.~2.4]{arlinskii:00}) a condition
    under the name \emph{condition (e)}
    (in~\cite[p.~62]{arlinskii.in:12} it is called \emph{condition
      (F)}) which is equivalent to what we call ``elliptic
    regularity''.  Note that Arlinskii considers (in our notation)
    \emph{bounded} boundary pairs associated with the quadratic form
    of the \emph{Krein} extension (see \Sec{ext.th.intro}).
     
  \item
    \label{brasche}
    The characterisations~\Thmenums{ell.bd2}{ell.bd2.viii}{ell.bd2.ix}
    are due to Ben Amor and Brasche~\cite{benamor-brasche:08} (see
    also the publication~\cite[Thm.~2.7]{bbb:11} and the references
    therein).  They showed that elliptic regularity (more precisely,
    that \Thmenums{ell.bd2}{ell.bd2.viii}{ell.bd2.ix}) are also
    equivalent to
    \begin{equation*}
      \lim_{a \to \infty} 
      a \norm{\RNeu_a - \RDir} < \infty,
    \end{equation*}
    where $\RNeu_a=(\HNeu_a+1)^{-1}$ and where $\HNeu_a$ is the
    operator associated with the quadratic form $\qf h_a (f):= \qf
    h(f) + a \normsqr[\HSaux] {\Gamma f}$ for $a \ge 0$ (see
    \Sec{robin} for such Robin-type boundary conditions).
  \end{enumerate}
\end{remark}
%----------------------------------------------------------------------

%----------------------------------------------------------------------
\begin{proof}[Proof of \Thm{ell.bd2}]
  \itemref{ell.bd2.i}~$\Leftrightarrow$~\itemref{ell.bd2.iv}: This is
  a direct consequence of the estimate
  \begin{equation}
    \label{eq:dsol.z.norm.hs}
    \frac 1 {C^\Dir(-1,z)^2} \qf q(\phi)
    \le \qf q_z(\phi)
    \le C^\Dir(z,-1)^2 \qf q(\phi)
  \end{equation}
  for all $\phi \in \HSaux^{1/2}$: $\qf q_z$ is bounded iff $\qf q$ is
  bounded.

  For the equivalence of \itemref{ell.bd2.v} with
  \itemref{ell.bd2.iv}, note that $\Im \qf l_z = -(\Im z) \qf q_z$ by
  \Thmenum{dn.z.qf}{dn.z.qf.vii}.  Similarly, for \itemref{ell.bd2.vi}
  we note that $\qf l_\lambda'=-\qf q_\lambda$ by
  \Thmenum{dn.z.qf}{dn.z.qf.v}.
  
  \itemref{ell.bd2.i}~$\Rightarrow$~\itemref{ell.bd2.vii}~We have
  \begin{equation*}
    \Re \qf l_z(\phi) - \qf l(\phi)
    = - \Re \iprod{(z+1) U(z,-1) \SDir \phi} {\SDir \phi}
    \ge -\omega(z) \normsqr \phi
  \end{equation*}
  for $\phi \in \HSaux^{1/2}$ (see \Thmenum{dn.z.qf}{dn.z.qf.ii} and
  \Thmenum{dsol.z}{dsol.z.iii}), where
  \begin{equation}
    \label{eq:def.omega.z}
    \omega(z)
    :=  C^2
        \max \Bigl\{ \sup_{\lambda \in \spec \HDir}
               \Re \Bigl(\frac {(z+1)(\lambda+1)}{\lambda-z} \Bigr),
               0
             \Bigr\},
  \end{equation}
  where $C=\norm[\HSaux \to \HS]{\eSDir}$.  Note that the inequality
  holds by the spectral calculus.  Since the real part of the fraction
  as a function in $\lambda \in \spec \HDir$ is continuous and has a
  limit as $\lambda \to \infty$, it attains a maximum $C_+(z) \in \R$
  (and a minimum $C_-(z) \in \R$) for $z \in \C \setminus \spec
  \HDir$.  In particular, $\omega(z) = C^2 \max\{C_+(z),0\} < \infty$.

  If $z=a<0$, then $C_+(a)=-(a+1)$ and $C_-(a)=(a+1)/a$.  Moreover,
  $\omega(a)/(a+1) \le 0$ and the limes superior is
  finite.

  \itemref{ell.bd2.vii}~$\Rightarrow$~\itemref{ell.bd2.i}~From
  \Thmenum{dn.z.qf}{dn.z.qf.v} we know that $\qf l_{-1}' = - \qf
  q_{-1}$.  Moreover, from the assumption, we have $-(\qf
  l_a(\phi)-\qf l(\phi)) \le \omega(a) \normsqr \phi$ for $a<0$ near
  $-1$ and $\phi \in \HSaux^{1/2}$.  Therefore, we conclude
  \begin{equation*}
    0 \le \normsqr{\SDir \phi}
    = \qf q_{-1}(\phi)
    = -\lim_{a  \to -1} 
           \frac {\qf l_a(\phi) - \qf l(\phi)}{a+1}
    \le \limsup_{a \to -1} 
           \frac {\omega(a)}{a+1} \cdot  \normsqr \phi,
  \end{equation*}
  hence the boundary pair is elliptically regular with $C^2=\limsup_{a
    \to -1} \omega(a)/(a+1)$.

  \itemref{ell.bd2.viii}~$\Rightarrow$~\itemref{ell.bd2.ix}~By
  assumption, $\wLambda \Gamma \wRNeu(\HS) = \Lambda \Gamma \RNeu
  (\HS) \subset \HSaux$.  Assume that $f_n \to f$ in $\HS$ and $\psi_n
  := \Lambda \Gamma \RNeu f_n \to \psi$ in $\HSaux$, hence also in
  $\HSaux^{-1/2}$.  Moreover, $\psi_n = \wLambda \Gamma \wRNeu f_n \to
  \wLambda \Gamma \wRNeu f$ in $\HSaux^{-1/2}$ since $\wLambda \Gamma
  \wRNeu$ is bounded as operator $\HS \hookrightarrow \HS^{-1}
  \stackrel \wRNeu \to \HS^1 \stackrel \Gamma \to \HSaux^{1/2}
  \stackrel \wLambda \to \HSaux^{-1/2}$.  Since limits in
  $\HSaux^{-1/2}$ are unique, we have $\wLambda \Gamma \wRNeu f = \psi
  \in \HSaux$.  In particular, $\map{\Lambda \Gamma \RNeu} \HS \HSaux$
  is closed, hence bounded by the closed graph theorem.

  \itemref{ell.bd2.ix}~$\Rightarrow$~\itemref{ell.bd2.viii}~Since
  $\Lambda \Gamma (\dom \HNeu)=\Lambda \Gamma \RNeu(\HS) \subset
  \HSaux$ by assumption, we have $\Gamma(\dom \HNeu) \subset
  \HSaux^1=\dom \Lambda$.

  \itemref{ell.bd2.i}~$\Leftrightarrow$~\itemref{ell.bd2.ix}~We have
  $\map{(\wLambda \Gamma \wRNeu)^* = \SDir}{\HSaux^{1/2}}{\HS^1}$ by
  \Prp{dtn.sdir} for $z=-1$, hence $S$ extends to a bounded operator
  $\HSaux \to \HS$ (i.e., $\qf q$ is a bounded form) iff $\wLambda
  \Gamma \wRNeu$ restricts to a bounded operator $\HS \to \HSaux$.
\end{proof}
%----------------------------------------------------------------------

Here are some consequences of elliptic regularity:
%----------------------------------------------------------------------
\begin{theorem}
  \label{thm:ell.bd2.impl}
  Let $(\Gamma,\HSaux)$ be an elliptically regular boundary pair and
  $z \in \C \setminus \spec \HDir$, then the following assertions are
  true:
  \begin{myenumerate}
  \item
    \label{ell.bd2.impl.i}
    The norms $\norm[\qf l_z] \cdot$  and
    $\norm[\HSaux^{1/2}] \cdot$ are equivalent, i.e.,
    \begin{equation*}
      \normsqr[\HSaux^{1/2}] \phi
      \le \normsqr[\qf l_z] \phi
      := \Re \qf l_z(\phi) + \omega(z) \normsqr[\HSaux] \phi 
      \le( L(z)+  \omega(z) \normsqr[1 \to 0] \Gamma)
                     \normsqr[\HSaux^{1/2}] \phi,
    \end{equation*}
    where $L(z)$ is defined in~\eqref{eq:def.m.z}.

  \item
    \label{ell.bd2.impl.i'}
    The form $\qf l_\lambda$ resp.\ the associated operator
    $\Lambda(\lambda)$ is bounded from below for all $\lambda \in \R
    \setminus \spec \HDir$.

  \item
    \label{ell.bd2.impl.ii}
    The form $\qf l_z$ is closed and sectorial, i.e., $\qf l_z(\phi)
    \in \Sigma_\vartheta -\omega(z)$ for all $\phi \in \HSaux^{1/2}$,
    where
    \begin{equation}
      \label{eq:def.sector}
      \Sigma_\vartheta
      := \bigset{w \in \C} {\abs{\arg w} \le \vartheta}
    \end{equation}
    for $\vartheta = \vartheta_z := \arctan L(z)$.
  \item
    \label{ell.bd2.impl.iii}
    The associated operator family $(\Lambda(z))_{z \in \C \setminus
      \spec \HDir}$ is self-adjoint, i.e., $\Lambda(z)^*=\Lambda(\conj
    z)$.  In particular, $\Lambda(z)$ is closed and self-adjoint for
    $z \in \R \setminus \spec \HDir$.  Moreover, the domain is $\dom
    \Lambda(z)=\HSaux^1$, i.e., \emph{independent} of $z$, and
    $\Lambda(z)$ considered as operator $\map{\Lambda(z)^{1 \to 0}}
    {\HSaux^1} \HSaux$ is bounded.
  \item
    \label{ell.bd2.impl.iv}
    The operator $\Lambda(z)$ is sectorial, i.e., we have $\spec
    {\Lambda(z)} \subset \Sigma_\vartheta -\omega(z)$, i.e., the
    spectrum of the \DtN\ map is contained in the sector
    $\Sigma_\vartheta-\omega(z)$ for $\vartheta = \vartheta_z$.
  \item
    \label{ell.bd2.impl.v}
    The operator $\Lambda(z)$ is $m$-sectorial in the sense of Kato,
    i.e.,
    \begin{equation*}
      \norm[\Lin \HSaux]{(\Lambda(z) - w)^{-1}}
      \le \frac 1 {\abs{w + \omega(z)} \sin (\vartheta_0 - \vartheta)}
    \end{equation*}
    for $\vartheta_0 \in (\vartheta,\pi)$ and $w \in \C \setminus
    (\Sigma_{\vartheta_0} - \omega(z))$.
  \end{myenumerate}
\end{theorem}
%----------------------------------------------------------------------
\begin{proof}
  \itemref{ell.bd2.impl.i}~follows from the ellipticity of $\qf l_z$
  shown in \Thmenum{ell.bd2}{ell.bd2.vii},~\eqref{eq:hsaux.0.12} and
  \Thmenum{dn.z.qf}{dn.z.qf.iii}.  \itemref{ell.bd2.impl.i'}~follows
  immediately from~\itemref{ell.bd2.impl.i} and
  again~\eqref{eq:hsaux.0.12}.  \itemref{ell.bd2.impl.ii}~The
  closeness of $\qf l_z$ on $\HSaux^{1/2}$ follows
  from~\itemref{ell.bd2.impl.i}.  Moreover, for $\phi \in \HSaux^{1/2}
  \setminus \{0\}$ we have
  \begin{equation*}
    \frac{\abs{\Im \qf l_z(\phi)}}
         {\Re \qf l_z(\phi) + \omega(z) \normsqr \phi}
    \le \frac {L(z)} \alpha = L(z)
  \end{equation*}
  using again \Thm{dn.z.qf}.  In particular, $\qf l_z(\phi)$ lies in
  the sector $\Sigma_\vartheta-\omega(z)$.

  \itemref{ell.bd2.impl.iii}~Note that $\Lambda(z)$ is the operator
  associated with $\qf l_z$ in the sense of sesquilinear forms,
  see~\cite[Thm.~VI.2.1]{kato:66}; in particular, $\Lambda(z)$ is
  closed and sectorial, and $\dom \Lambda(z)$ is a form core (i.e.,
  dense in $\HSaux^{1/2}$.  Moreover, $\Lambda(z)$ is the (strong)
  operator associated with $\qf l_z$, see~\eqref{eq:dn.ass.op}
  
  For $\HSaux^1 = \dom \Lambda(z)$ we use the equality
  \begin{equation*}
    \qf l_z(\phi,\psi) - \qf l(\phi,\psi)
    = -(z+1) \iprod{\SDir(z) \phi} {\SDir \psi}.
  \end{equation*}
  for $\phi,\psi \in \HSaux^{1/2}$ (see~\eqref{eq:dn.qf.zw}).  The
  inclusion ``$\subseteq$'' follows from
  \begin{multline}
    \label{eq:ell.bd2.impl.iiia}
    \norm {\Lambda(z) \phi}
    = \sup_{\psi \in \HSaux^{1/2}}
       \frac{\qf l_z(\phi,\psi)}{\norm[\HSaux] \psi}
    \le \sup_{\psi \in \HSaux^{1/2}}
       \frac{\iprod[\HSaux] {\Lambda \phi} \psi 
         + \abs{z+1} \sqrt{\qf q_z(\phi) \qf q(\psi)}}
       {\norm[\HSaux] \psi}\\
    \le \norm{\Lambda \phi} 
     + \abs{z+1}\sqrt{C^\Dir(z,1)}C^2\norm[\HSaux] \phi
  \end{multline}
  for $\phi \in \dom \Lambda=\HSaux^1$ using~\eqref{eq:dsol.z.norm.hs}
  and the boundedness of $\qf q$.

  For the inclusion ``$\supseteq$'' we argue similarly.  The
  boundedness of $\Lambda(z)$ as operator $\HSaux^1 \to \HSaux$
  follows also from~\eqref{eq:ell.bd2.impl.iiia}.

  \itemref{ell.bd2.impl.iv}--\itemref{ell.bd2.impl.v} can be deduced
  similarly as in~\cite[Sec.~2]{mnp:13}.
\end{proof}
%----------------------------------------------------------------------

Note that if the boundary pair is not elliptically regular then $\qf
l_z$ is not necessarily closed on $\HSaux^{1/2}$; even worse, $\qf
l_z$ might not be bounded from below or sectorial, see \Rem{dtn-qf}.

We denote by $\map {B(z)} \HS \HSaux$ the adjoint of
$\map{\eSDir(\conj z)} \HSaux \HS$, i.e., the restriction of
$\map{\SDir(\conj z)^*}{\HS^{-1}}{\HSaux^{-1/2}}$ to $\HS$ for an
elliptically regular boundary pair.  The proof of the following is
straightforward from \Thms{dsol.z}{dn.z.qf}:
%----------------------------------------------------------------------
\begin{proposition}
  \label{prp:bd2.b.z}
  Assume that $(\Gamma,\HSaux)$ is an elliptically regular boundary
  pair, and that $z, w \notin \spec{H^\Dir}$, then the following
  assertions are true:
  \begin{myenumerate}
  \item
    \label{bd2.b.z.i}
    We have $\map{B(z) - B(w)=(z-w) B(w) \RDir(z)} \HS \HSaux$ and the
    operator is bounded.
  \item
    \label{bd2.b.z.ii}
    The map $z \to B(z)$ is holomorphic and the derivatives
    $\map{B^{(k)}(z)= k! B(z)\RDir(z)^k}\HS \HSaux$ are bounded.
  \item
    \label{bd2.b.z.iii}
    We have $\map{\Lambda(z) - \Lambda(w) = -(z-w) \eSDir(\conj z)^*
      \eSDir(w) = -(z-w) B(z) B(\conj w)^*} \HSaux \HSaux$ and the
    operator is bounded.
  \item
    \label{bd2.b.z.iv}
    The derivatives of $\Lambda(\cdot)$ are bounded, i.e.,
    \begin{equation*}
      \map{\Lambda^{(k)}(z)= - k! \eSDir(\conj z)^* \RDir(z)^k \eSDir(z)
        = - k! B(z) \RDir(z)^k B(\conj z)^*}\HS \HSaux;
    \end{equation*}
    in particular, $\Lambda'(z)= - \eSDir(\conj z)^* \eSDir(z) = - B(z)
    B(\conj z)^*$.
  \item
    \label{bd2.b.z.v}
    The imaginary part $\Im \Lambda(z) = -(\Im z) \eSDir(z)^*
    \eSDir(z) = -(\Im z) B(\conj z) B(\conj z)^*$ is bounded and
    non-positive for $\Im z \ge 0$.
  \item
    \label{bd2.b.z.vi}
    Assume that $0 \notin \spec{\Lambda(z)}$ and $0 \notin \spec
    {\Lambda(w)}$, then we have
    \begin{equation*}
      \Lambda(z)^{-1}-\Lambda(w)^{-1}
      =(z-w) \Lambda(z)^{-1}\eSDir(\conj w)^*\eSDir(z)\Lambda(w)^{-1}
      =\Lambda(z)^{-1}B(z)B(\conj w)^*\Lambda(w)^{-1}.
  \end{equation*}

  \end{myenumerate}
\end{proposition}
%----------------------------------------------------------------------
\begin{proof}
  Let us just give a proof for the last assertion: it follows from
  $\Lambda(z)^{-1}-\Lambda(w)^{-1}=-\Lambda(z)^{-1} (\Lambda(z) -
  \Lambda(w)) \Lambda(w)^{-1}$ and~\itemref{bd2.b.z.iii} and the fact
  that $\dom \Lambda(z)=\dom \Lambda(w)=\HSaux^1$.
\end{proof}
%----------------------------------------------------------------------

%----------------------------------------------------------------------
\begin{remark}
  \label{rem:non-ell.ntd}
  The last assertion is useful if we know that $\Lambda(\lambda)$ has
  a bounded inverse without a priori knowing that $\lambda \notin
  \spec \HNeu$, this is needed in the proof of
  \Thmenum{krein2}{krein2.ii}.  Note that the formula in
  \Prpenum{nd.z}{nd.z.i} requires that $z$ and $w$ are \emph{not} in
  the Neumann spectrum $\spec \HNeu$.  
\end{remark}
%----------------------------------------------------------------------

%----------------------------------------------------------------------
\subsection{Positive boundary pairs}
%----------------------------------------------------------------------
\label{sec:pos.bd2}

We have a sort of ``converse'' notion of elliptic regularity, namely,
that the solution form $\qf q$ is (uniformly) positive:

%----------------------------------------------------------------------
\begin{definition}
  \label{def:pos.bd2}
  We say that the boundary pair $(\Gamma,\HSaux)$ is \emph{positive},
  if the solution form $\qf q$ ($\qf q(\phi) := \normsqr[\HS]{\SDir
    \phi}$, $\phi \in \HSaux^{1/2}$) is uniformly positive, i.e., if
  there is a constant $c>0$ such that
  \begin{equation*}
    \qf q(\phi)
    = \normsqr[\HS]{\SDir \phi} 
    \ge c^2 \normsqr[\HSaux] \phi
  \end{equation*}
  for all $\phi \in \HSaux^{1/2}$.
\end{definition}
%----------------------------------------------------------------------

%----------------------------------------------------------------------
\begin{remark}
  \label{rem:pos.bd2}
  Not all boundary pairs are positive: a counterexample is given by
  the manifold model of \Sec{lapl.mfd}.  We have indicated in
  \Thm{bd3.intro} how this notion relates to \emph{ordinary boundary
    triples}.
\end{remark}
%----------------------------------------------------------------------

For the positivity of a boundary pair, we have the following
equivalent characterisations.  The proof is very much the same as the
one of \Thm{ell.bd2}, hence we omit it:
%----------------------------------------------------------------------
\begin{theorem}
  \label{thm:pos.bd2}
  Let $(\Gamma,\HSaux)$ be a boundary pair, then the following
  assertions are equivalent:
  \begin{myenumerate}
  \item 
    \label{pos.bd2.i}
    The boundary pair is positive.

  \item 
    \label{pos.bd2.iv}
    The solution form $\qf q_z$ is uniformly positive for some (any)
    $z \in \C \setminus \spec \HDir$.

 \item 
   \label{pos.bd2.iii}
   There is a constant $c(z)>0$ such that $\norm[\HS]{\SDir(z)} \ge
   c(z) \norm[\HSaux] \phi$ for all $\phi \in \HSaux^{1/2}$ for some
   (any) $z \in \C \setminus \spec \HDir$.

  \item 
    \label{pos.bd2.v}
    The form-valued function $z \to -\qf l_z$ is a uniformly strict
    Nevanlinna function, i.e., $-\Im \qf l_z$ is uniformly positive for
    some (any) $z \in \C$ with $\Im z>0$.

  \item
    \label{pos.bd2.vi}
    The negative derivative $-\qf l'_\lambda$ is a uniformly positive
    form for some (any) $\lambda \in \R \setminus \spec \HDir$.
  \end{myenumerate}
\end{theorem}
%----------------------------------------------------------------------
Operator-valued functions $-\Lambda(\cdot)$ with $-\Im \Lambda(z)$
being uniformly positive are called \emph{uniformly strict Nevanlinna}
functions in~\cite[p.~5354]{dhms:06}.  In~\itemref{pos.bd2.v}, we have
a form equivalent notion.

%----------------------------------------------------------------------
% dddd
\section{Resolvent formulae and spectral relations}
\label{sec:bd2.krein}
%
%----------------------------------------------------------------------

In this section we present some of our main results: a Krein-type
resolvent formula (\Thm{krein.res}) and spectral relations between the
Neumann operator and the family of \DtN\ operators
(\Thms{krein1}{krein2}).

%----------------------------------------------------------------------
\subsection{Resolvent formula for boundary pairs}
%----------------------------------------------------------------------

We will prove a resolvent formula in a rather abstract way.
To do so, we first need some technical preparation.  Denote by $\map
{\pi_z} {\HS^1}{\HS^{1,\Dir}}$ the canonical map associating with $f=w
+ h \in \HS^1=\HS^{1,\Dir} \dplus \LS^1(z)$ the component $w \in
\HS^{1,\Dir}$ (see~\eqref{eq:dec.ker.sol}) and by $\pi_z^*$ its dual.
We denote by $\wHNeu$, $\wHDir$ and $\wLambda(z)$ the operators
$\HNeu$, $\HDir$ and $\Lambda(z)$ extended to $\HS^1 \to \HS^{-1}$,
$\HS^{1,\Dir} \to \HS^{-1,\Dir}$, $\HSaux^{1/2} \to \HSaux^{-1/2}$.
respectively.  The natural inclusion $\embmap{\iota = \iota_{1,\Dir
    \to 1}}{\HS^{1,\Dir}}{\HS^1}$ is an isometry and induces an
operator $\map{\iota^*}{\HS^{-1}}{\HS^{-1,\Dir}}$.  Note that $\map
{\iota^* \wt f}{\HS^{1,\Dir}} \C$ is defined by $(\iota^* \wt f )(g) =
\wt f(\iota g)$ for $\wt f \in \HS^{-1}$ and $g \in \HS^{1,\Dir}$.  If
$f \in \HS$ is embedded in $\HS^{-1}$ via $\wt f= \iprod f \cdot$,
then
\begin{equation}
  \label{eq:iota.restr}
  (\iota^* \wt f)(g)
  = \iprod f {\iota g} 
  = \iprod f g
  = \wt f (g),
\end{equation}
i.e., $\iota^*$ restricted to $\HS$ acts as the identity on $\HS$.

We have now the following relation between the operators $\wHNeu$ and
$\wHDir$ and their resolvents extended to the scale of Hilbert spaces.
%----------------------------------------------------------------------
\begin{lemma}
  \label{lem:pi.iota.z}
  Let $z \in \C \setminus (\spec \HDir \cup \spec \HNeu$, then we have
  the following identities:
  \begin{myenumerate}
  \item
    \label{pi.iota.z.i}
    $\map{\iota^*(\wHNeu - z) = (\wHDir - z) \pi_z}
    {\HS^1}{\HS^{-1,\Dir}}$,
  \item
    \label{pi.iota.z.ii}
    $\map{\wRDir(z) \iota^* = \pi_z \wRNeu(z)}
    {\HS^{-1}}{\HS^{1,\Dir}}$,
    
  \item
    \label{pi.iota.z.iii}
    $\map{\iota \pi_z = \iota \wRDir(z) \iota^*(\wHNeu -
      z)}{\HS^1} {\HS^1}$ is the projection onto $\HS^{1,\Dir}$ with
    kernel $\LS^1(z)$, Moreover, the complementary projection is given
    by $\id_{\HS^1} - \iota \pi_z = P(z) (= S(z) \Gamma)$.
  \item
    \label{pi.iota.z.iv}
    $\map{P(z) \wRNeu(z) = \wRNeu(z) P(\conj z)^*}{\HS^{-1}}{\HS^1}$.
  \end{myenumerate}
\end{lemma}
%----------------------------------------------------------------------
\begin{proof}
  \itemref{pi.iota.z.i}~Let $f \in \HS^1$ and $g \in \HS^{1,\Dir}$.
  Then $f=w + h$ with $w = \pi_z f$ according to the decomposition
  $\HS^1=\HS^{1,\Dir} \dplus \LS^1(z)$.  Moreover,
  \begin{equation*}
    \iprod[\HS^{-1,\Dir},\HS^{1,\Dir}] {\iota^* (\wHNeu - z)f} g
    = (\qf h - z \qf 1)(f,g)
    = (\qf h - z \qf 1)(w,g)
    = \iprod[1,-1] {(\wHDir - z) \pi_z f} g,
  \end{equation*}
  since $(\qf h - z \qf 1)(h,g)=0$ ($h \in \LS^1(z)$ and $g \in
  \HS^{1,\Dir}$, see \Def{weak.dsol.z}).

  \itemref{pi.iota.z.ii}~follows from~\eqref{pi.iota.z.i} by multiplying
  with the resolvents.

  For~\eqref{pi.iota.z.iii}, note that $(\id_{\HS^1} - \iota \pi_z)f=
  h = P(z)f$ is the complementary projection onto $\LS^1(z)$.
  Moreover, we have $\iota \wRDir(z) \iota^*(\wHNeu-z)=\iota \pi_z
  \wRNeu(z)(\wHNeu-z)=\iota\pi_z$ by~\eqref{pi.iota.z.ii}.

  \itemref{pi.iota.z.iv}~follows from
  \begin{multline*}
    P(z)\wRNeu(z) 
    = (\id_{\HS^1}-\iota \pi_i) \wRNeu(z)
    = \wRNeu(z) - \iota \wRDir(z) \iota^*
    = \wRNeu(z) - (\wRDir(\conj z)\iota^*)^* \iota^*\\
    = \wRNeu(z) - (\pi_{\conj z} \wRNeu(\conj z))^* \iota^*
    = \wRNeu(z)(\id_{\HS^{-1}} - (\iota \pi_{\conj z})^*
    = \wRNeu P(\conj z)^*,
  \end{multline*}
  where we used~\eqref{pi.iota.z.ii} twice (once for $z$ in the second
  and once in its dual version for $\conj z$ in the fourth equality).
\end{proof}
%----------------------------------------------------------------------

We have now prepared all ingredients in order to prove one of the main
theorems for boundary pairs, a weak version of the so-called
\emph{Krein's resolvent formula}.  This formula allows us to detect
the Neumann spectrum as the ``zeros'' (\Thms{krein1}{krein2}) of the
\DtN\ operator. 
%----------------------------------------------------------------------
\begin{theorem}
  \label{thm:krein.res}
%   \addtocounter{equation}{-1} %
  Assume that $(\Gamma,\HSaux)$ is a boundary pair associated with the
  quadratic form $\qf h$ and that $z \in \C \setminus (\spec \HDir
  \cup \spec \HNeu)$.
  \begin{subequations}
    \begin{myenumerate}
    \item
      \label{krein.res.i}
      We have the following weak versions of Krein's resolvent
      identity:
      \begin{align}
        \nonumber
        \wRNeu (z) - \iota \wRDir(z) \iota^*
        &= P(z) \wRNeu(z)\\
        \label{eq:krein.res.w}
        &= \SDir(z) \wLambda(z)^{-1} \SDir(\conj z)^*
        \colon \HS^{-1}
        \stackrel {\SDir(\conj z)^*} \longrightarrow
        \HSaux^{-1/2}
        \stackrel {\wLambda(z)^{-1}} \longrightarrow
        \HSaux^{1/2}
        \stackrel {\SDir(z)} \longrightarrow 
        \HS^1,\\
        \label{eq:krein.res0}
        \RNeu (z) - \RDir(z)
        &= \SDir(z) \wLambda(z)^{-1} \SDir(\conj z)^*
        \colon \HS
        \stackrel {\SDir(\conj z)^*} \longrightarrow
        \HSaux^{-1/2}
        \stackrel {\wLambda(z)^{-1}} \longrightarrow
        \HSaux^{1/2}
        \stackrel {\SDir(z)} \longrightarrow 
        \HS.
      \end{align}
    \item
      \label{krein.res.iii}
      If, in addition, the boundary pair is elliptically regular, then
      \begin{equation}
        \label{eq:krein.res}
        \RNeu (z) - \RDir(z)
        = \eSDir(z) \Lambda(z)^{-1} \eSDir(\conj z)^*
        \colon
        \HS \stackrel {\eSDir(\conj z)^*} \longrightarrow 
        \HSaux \stackrel {\Lambda(z)^{-1}} \longrightarrow 
        \HSaux \stackrel {\eSDir(z)} \longrightarrow \HS,
      \end{equation}
      i.e., these operators do not leave the original Hilbert spaces
      $\HSaux$ and $\HS$.
    \end{myenumerate}
  \end{subequations}
\end{theorem}
%----------------------------------------------------------------------
\begin{proof}
  \itemref{krein.res.i}~We have
 \begin{align*}
   \wRNeu(z) - \iota \wRDir(z) \iota^*
   &= (\id_{\HS^1} - \iota \pi_z) \wRNeu(z)\\
   &= P(z) \wRNeu(z)
   = P(z)^2 \wRNeu(z)%\\
   = P(z) \wRNeu(z) P(\conj z)^*
   = S(z) \Gamma \wRNeu(z) \Gamma^* S(\conj z)^*
 \end{align*}
 using \Lem{pi.iota.z}.  Finally, in \Prp{dn.z.inv} we showed that
 \begin{equation*}
   \Gamma \wRNeu(z) \Gamma^*
   = \wLambda(z)^{-1}
 \end{equation*}
 as an operator $\HSaux^{-1/2} \to \HSaux^{1/2}$, and the resolvent
 formula~\eqref{eq:krein.res.w} follows, as well
 as~\eqref{eq:krein.res0}.

 \itemref{krein.res.iii}~is just a consequence
 of~\eqref{eq:krein.res.w} together with \Rem{bd2.ell} and
 \Prpenum{dn.z.inv}{dn.z.inv.iv}.
\end{proof}
%----------------------------------------------------------------------

%----------------------------------------------------------------------
\subsection{Operator pencils}
\label{sec:sp.rel.gen}
%----------------------------------------------------------------------
% 

For the spectral relations in \Thm{krein2}, we need some results on
operator pencils on $\HSaux$, which we define here in the form we need
it (see e.g.~\cite{tretter:00, eschwe-langer:04} and references
therein).  We restrict ourselves to the case where $T(z)-T(z_0)$ are
bounded operators for all $z$ (which corresponds to the elliptic
regular case in our application later on, see
\Remenum{op.pencil}{op.pencil.ii}):
%----------------------------------------------------------------------
\begin{definition}
  \label{def:op.pencil}
  Let $D \subset \C$ be open with $\conj D=D$.  We say that
  $T(\cdot)=\{T(z)\}_{z \in D}$ is a \emph{holomorphic self-adjoint
    operator pencil} if $T(z)^*=T(\conj z)$ and $z \mapsto
  T(z)-T(z_0)$ is holomorphic with values in the set of \emph{bounded}
  operators on $\HSaux$ for some $z_0 \in D$.

  The \emph{spectrum of the operator pencil $T(\cdot)$} is given by
  \begin{equation*}
    \spec {T(\cdot)} :=
    \set{z \in \C}{\text{$T(z)$ is not invertible}}.
  \end{equation*}
  We say that $\lambda$ is an \emph{eigenvalue} of $T(\cdot)$ (shortly
  $\lambda \in \spec[p]{T(\cdot)}$) if $T(\lambda)$ is not injective.
\end{definition}
%----------------------------------------------------------------------
Note that the usual spectrum of an operator $A$ is the operator pencil
spectrum of the operator pencil $z \mapsto A-z$.

Fix $\lambda \in D \cap \R$.  It follows that if
\begin{equation*}
  T(z) = A_0 - (z-\lambda) A_1 - (z-\lambda)^2A_2(z)
\end{equation*}
then $A_0=T(\lambda)$ is self-adjoint (and possibly unbounded), where
$A_1=-T'(\lambda)$ is self-adjoint and bounded and where $z \mapsto
A_2(z)$ is holomorphic with values in the bounded operators on
$\HSaux$.  Then it is obvious that $\lambda \in \spec[p]{T(\cdot)}$
iff $\lambda \in \spec[p]{A_0}$, but it is not clear whether $\lambda
\in \spec {T(\cdot)}$ iff $\lambda \in \spec {A_0}$ or whether
$\lambda$ is isolated in $\spec {T(\cdot)}$ iff $\lambda$ is isolated
in $\spec {A_0}$.  For the latter assertion, we need more assumptions:
%----------------------------------------------------------------------
\begin{proposition}
  \label{prp:op.pencil}
  Assume that $T(\cdot)$ is a holomorphic self-adjoint operator pencil
  on $D$ and that $\lambda \in D \cap \R$.  Assume that $A_1 :=
  -T'(\lambda)$ is (bounded and) uniformly positive (i.e., there are
  constants $0<c \le C < \infty$ such that $c\normsqr \phi \le
  \iprod{A_1 \phi} \phi \le C \normsqr \phi$).  Then the following
  assertions are equivalent:
  \begin{myenumerate}
  \item
    \label{op.pencil.iia}
    $\lambda$ is isolated in $\spec {T(\cdot)}$ and 
    \begin{equation}
      \label{eq:ass.op.pencil}
      T(z)^{-1}
      = \frac 1 {z - \lambda}
      T_\lambda(z) + \hat T_\lambda(z),
    \end{equation}
    where $T_\lambda(\cdot)$ and $\hat T_\lambda(\cdot)$ are
    holomorphic (bounded) operator functions near $z=\lambda$;
  \item
    \label{op.pencil.iib}
    $0$ is isolated in the individual operator spectrum $\spec
    {T(\lambda)}$.
  \end{myenumerate}
\end{proposition}
%----------------------------------------------------------------------
\begin{proof}
  Without loss of generality, we can assume that $\lambda=0$;
  otherwise consider $\wt T(z):=T(z+\lambda)$.

  Let us first consider the situation $A_2(z)=0$ for all $z \in D$. In
  this case, we do not need the expression~\eqref{eq:ass.op.pencil}
  for $T(z)^{-1}$.  Note first that $\ker T(0)=\ker A_0$, hence $0 \in
  \spec[p] {T(\cdot)}$ iff $0 \in \spec[p] {A_0}$.  Moreover, we have
  \begin{equation}
    \label{eq:op.pencil2}
    T(z)= A_0-zA_1 = A_1^{1/2}( A_1^{-1/2} A_0 A_1^{-1/2} - z) A_1^{1/2}.
  \end{equation}
  Now, $T(z)$ is invertible iff $A_1^{-1/2} A_0 A_1^{-1/2} - z$ is
  invertible, and $0$ is isolated in $\spec{A_1^{-1/2} A_0 A_1^{-1/2}}$ iff
  $0$ is isolated in $\spec {A_0}$ (see e.g.~\cite[Lem~3.1]{bgp:08}).

  Let us now consider the general situation and prove
  \itemref{op.pencil.iia}$\Rightarrow$\itemref{op.pencil.iib}: We have
  \begin{equation*}
    S(z) :=  T(z)^{-1} (A_0 - z A_1)
    = T(z)^{-1} (T(z) + z^2 A_2(z))
    = \id_\HSaux + z^2 T(z)^{-1} A_2(z).
  \end{equation*}
  Using~\eqref{eq:ass.op.pencil} we have
  \begin{equation*}
    z^2 T(z)^{-1} A_2(z)
    = (z T_0(z) + z^2 \hat T_0(z)) A_2(z)
  \end{equation*}
  and this operator family is bounded near $z=0$ (including $z=0$). In
  particular, if $\abs z$ is small enough, then $S(z)$ is invertible,
  hence $A_0-z A_1 = T(z) S(z)$ is invertible, too.  Therefore, we
  have shown that $0$ is isolated in the spectrum of the operator
  pencil $z \mapsto A_0 - z A_1$.  From the first part (the case
  $A_2(z)=0$), it follows that $0 \in \spec {A_0}$ is isolated.

  \itemref{op.pencil.iib}$\Rightarrow$\itemref{op.pencil.iia}: Let $0
  \in \spec{A_0}$ be isolated then $0$ is isolated in the spectrum of
  $z \mapsto A_0 - z A_1$, again by the first part (the case
  $A_2(z)=0$).  Using~\eqref{eq:op.pencil2} we have
  \begin{equation*}
    (A_0 - z A_1)^{-1}
    = A_1^{-1/2} (B - z)^{-1} A_1^{-1/2}
    = A_1^{-1/2} \Bigl(-\frac 1 z \1_{\{0\}}(B) + B_0^{-1} \Bigr) 
      A_1^{-1/2}
  \end{equation*}
  where $B:= A_1^{-1/2} A_0 A_1^{-1/2}$ and where $B_0$ is the
  restriction of $B$ onto $\ker B^\orth$.  In particular, we have a
  representation of the inverse as in~\eqref{eq:ass.op.pencil}, and in
  order to show that $T(z)$ has a bounded inverse for $0 < \abs z$
  small enough we can argue similarly as in
  \itemref{op.pencil.iia}$\Rightarrow$\itemref{op.pencil.iib}.
\end{proof}
%----------------------------------------------------------------------

%----------------------------------------------------------------------
\begin{remark}
  \label{rem:op.pencil}
  \indent
  \begin{myenumerate}
  \item The non-trivial assertion in \Prp{op.pencil} is the fact that
    $\lambda$ and $0$ are \emph{isolated} in the spectra, i.e., that
    $T(z)$ and $T(\lambda)-z$ are invertible for all $z \ne \lambda$
    near $\lambda$.

  \item
    \label{op.pencil.ii}
    In our application, the operator pencil will be $\Lambda(\cdot)$.
    If we assume that the boundary pair is elliptically regular, then
    $\Lambda(\cdot)$ is a holomorphic self-adjoint operator pencil on
    $D=\C \setminus \spec \HDir$ as in \Def{op.pencil} since
    $\Lambda(z)^*=\Lambda(\conj z)$
    (\Thmenum{ell.bd2.impl}{ell.bd2.impl.iii}) and
    \begin{align*}
%      \nonumber
      \Lambda(z) 
      &= \Lambda(\lambda) - (z-\lambda) \eSDir(\lambda)^* \eSDir(z)\\
%      \nonumber
      &= \Lambda(\lambda) - (z-\lambda) \eSDir(\lambda)^*
      \bigl( \eSDir(\lambda) + (z-\lambda) \RDir(z) \eSDir(\lambda) \bigr)\\
%      \label{eq:dn.z.hol}
      &= \Lambda(\lambda) - (z-\lambda) Q(\lambda)
      - (z-\lambda)^2 A_{2,\lambda}(z), 
      \qquad A_{2,\lambda}(z):= \eSDir(\lambda)^* \RDir(z) \eSDir(\lambda)
    \end{align*}
    by \Thm{dsol.z} and \Prp{bd2.b.z}.  In particular,
    $A_0=\Lambda(\lambda)$ and $A_1=Q(\lambda)$.  Moreover, $z \to
    A_{2,\lambda}(z)$ is holomorphic and $\norm{A_{2,\lambda}(z)} \le
    \normsqr{\eSDir(\lambda)}/d(z,\spec \HDir)$.  In particular,
    $\Lambda(z)-\Lambda(\lambda)$ is a bounded operator.

    The boundedness of $A_1=Q(\lambda)$ means that the boundary pair
    is elliptically regular and the uniform positivity of $A_1$ means
    that the boundary pair is positive.

  \item We are not aware of a counterexample of an operator pencil
    where $A_1=-T'(\lambda)$ is uniformly positive, but \emph{not}
    bounded.  However, we use \Prp{op.pencil} in
    \Thmenum{krein2}{krein2.v}, and \Ex{non-ell.bd2.dis.spec} is a
    counterexample to the assertion of the theorem (the boundary pair
    is positive, but not elliptically regular), see also
    \Remenum{krein2}{rem.krein2.iii}.
    
  \item Also, we cannot drop the \emph{uniform positivity} of $A_1$,
    as the counterexample given below in \Ex{op.pencil} shows
    (e.g. assuming that $A_1$ is bounded and injective, but \emph{not}
    uniformly positive).  This means that it is not enough for the
    application of \Prp{op.pencil} that the boundary pair is
    elliptically regular, but not positive (as in the case of our
    basic example, a Laplacian on a manifold with boundary).
    Nevertheless, we are not aware of a \emph{boundary pair} which is
    elliptically regular but not positive, such that the assertion of
    \Prp{op.pencil} is false for $T(\cdot)=\Lambda(\cdot)$.
  \end{myenumerate}
\end{remark}
%----------------------------------------------------------------------

The author is indebted to Michael Strauss for the following example:
%----------------------------------------------------------------------
\begin{example}
  \label{ex:op.pencil}
  We give here a counterexample for
  ``\itemref{op.pencil.iib}$\Rightarrow$\itemref{op.pencil.iia}'' in
  \Prp{op.pencil} violating the uniform positivity of $A_1$, even in
  the situation when $A_2(z)=0$: Let $T(z):= A_0 - z A_1$ with
  $A_0=A_0^*$ and $A_1=A_1^* \ge 0$ specified below:

  Let $0$ be an \emph{isolated} eigenvalue of $A_0$ of infinite
  multiplicity in the essential spectrum of $A_0$.  Denote by
  $(\phi_n)_n$ an orthonormal basis of the eigenspace $\ker A_0$.  Set
  \begin{equation*}
    A_1 := \1_{\R \setminus \{0\}}(A_0)
    + \sum_n \frac 1 n \iprod{\phi_n} \cdot \phi_n
  \end{equation*}
  then $A_1$ is bounded, non-negative and injective (but not
  \emph{uniformly} positive).  Moreover, $T(z) \phi_n = -(z/n)
  \phi_n$, hence $-z/n \in \spec{T(\cdot)}$ and $\norm{T(z) \phi_n}
  \to 0$ showing that $T(z)$ does not have a bounded inverse for any
  $z \in \C$.  In particular, $0 \in \spec {T(\cdot)}$, but $0$ is
  \emph{not} isolated.
\end{example}
%----------------------------------------------------------------------

%----------------------------------------------------------------------
\subsection{Spectral relations between the Neumann and the \DtN\
  operator family}
\label{sec:sp.rel}
%----------------------------------------------------------------------
We can now prove some important consequences of Krein's resolvent
formula and other formulae for the \DtN\ operator.  In particular, we
have the following spectral relations:
%----------------------------------------------------------------------
\begin{theorem}
  \label{thm:krein1}
%   \addtocounter{equation}{-1} %
  Assume that $(\Gamma,\HSaux)$ is a boundary pair associated with the
  quadratic form $\qf h$ and let $\lambda \in \C \setminus \spec
  \HDir$.  Then the following assertions are true:
  \begin{subequations}
    \begin{myenumerate}
    \item
      \label{krein1.i}
      The Dirichlet solution operator $S(\lambda)$ is a topological
      isomorphism from $\ker \Lambda(\lambda)$ onto $\ker (\HNeu -
      \lambda)$ with inverse $\Gamma$, i.e.,
      \begin{equation}
        \label{eq:krein1.ia}
        \ker (\HNeu - \lambda) = S(\lambda) \ker \Lambda(\lambda).
      \end{equation}
      In particular, we have the spectral relation
      \begin{equation}
        \label{eq:krein1.ib}
        \lambda \in \spec[p] \HNeu
             \quad \Leftrightarrow \quad
        0 \in \spec[p]{\Lambda(\lambda)}.
      \end{equation}
      for the point spectrum (the set of eigenvalues).  Moreover, the
      multiplicity of an eigenspace is preserved.

    \item
      \label{krein1.ia}
      Assume that $\map \RNeu \HS \HS$ and $\map \Gamma {\HS^1}
      \HSaux$ are compact operators, then the spectra of $\HNeu$,
      $\HDir$ and $\Lambda(\lambda)$ are purely discrete.  Moreover,
      the spectral relation~\eqref{eq:krein1.ib} is true for the
      discrete (hence the entire) spectrum, i.e.,~\eqref{eq:krein1.ib}
      holds with $\spec[p] \cdot$ replaced by $\spec[disc] \cdot$ or
      $\spec \cdot$.

    \item
      \label{krein1.i0}
      If $\lambda \notin \spec \HNeu$, then $\Lambda(\lambda)$ has a
      bounded inverse.  In particular, $\Lambda(\lambda)$ is closed
      and $0 \notin \spec {\Lambda(\lambda)}$.  (The converse is in
      general false, see \Remenum{krein2}{krein2.ii} and
      \Ex{non-ell.bd2.spec}, but true for \emph{isolated} eigenvalues,
      see~\itemref{krein1.i} and the next
      assertion~\itemref{krein1.v}).

    \item
      \label{krein1.v}
      Assume that $\lambda$ is isolated in $\spec \HNeu$, then there
      exist $r_1>0$ and $C_\lambda>0$ (depending only on $\lambda$,
      $\spec \HNeu$ and $r_1$) such that
      \begin{equation}
        \label{eq:krein1.x}
        \norm[\HSaux \to \HSaux] {\Lambda(z)^{-1}}
        \le \normsqr[1 \to 0] \Gamma 
           \frac {C_\lambda} {\abs{z - \lambda}}
      \end{equation}
      for all $z \ne \lambda$ and $\abs{z-\lambda} \le r_1$.
      Moreover, $\lambda$ is an isolated eigenvalue in the operator
      pencil spectrum $\spec{\Lambda(\cdot)}$, i.e.,
      \begin{equation*}
        \text{$\lambda$ is isolated in $\spec \HNeu$}
        \quad \Rightarrow \quad
        \text{$\lambda$ is isolated in the operator pencil spectrum 
           $\spec {\Lambda(\cdot)}$}.
      \end{equation*}
    \end{myenumerate}
  \end{subequations}
\end{theorem}
%----------------------------------------------------------------------
\begin{proof}
 \itemref{krein1.i}~Let $\phi \in \ker \Lambda(\lambda)$ and
  $h:=S(\lambda) \phi$, then
  \begin{equation*}
    (\qf h - \lambda \qf 1)(h, g)
    = \qf l_\lambda(\phi,\Gamma g)
    = \iprod{\Lambda(\lambda)\phi} {\Gamma g}
    =0
  \end{equation*}
  for $g \in \HS^1$ by the definition of $\qf l_\lambda$
  in~\eqref{eq:def.dn.z}, hence $h \in \dom \HNeu$ and
  $(\HNeu-\lambda) h =0$.

  On the other hand, if $h \in \ker (\HNeu-\lambda)$ then it is
  easily seen that $h \in \LS^1(\lambda)$.  Set $\phi:=\Gamma h$, then
  $h=S(\lambda) \phi$ and a similar calculation as above shows that
  $\phi \in \dom \Lambda(\lambda)$ and $\Lambda(\lambda) \phi =0$.

  The spectral equivalence~\eqref{eq:krein1.ib} for the point spectrum
  is obvious from~\eqref{eq:krein1.ia}, as well as the preserved
  multiplicity.

  \itemref{krein1.ia}~It follows from \Thm{nd.comp} that the spectrum
  of $\Lambda(\lambda)$ is discrete (see also
  \Prpenum{bd2.robin}{nd.z.v} for the discreteness of $\spec \HDir$).
  The spectral relation is then a consequence of
  part~\itemref{krein1.i}.

  \itemref{krein1.i0}~Assume that $\lambda \notin \spec \HNeu$, then
  $\Lambda(\lambda)^{-1} = \Gamma \RNeu(\lambda) \Gamma^*$ exists and
  is bounded as operator $\HSaux \to \HSaux$ by
  \Prpenum{dn.z.inv}{dn.z.inv.iv}.  In particular, $\Lambda(\lambda)$
  is closed and has a bounded inverse, hence $0 \notin
  \spec{\Lambda(\lambda)}$.

  \itemref{krein1.v} If $\lambda$ is isolated in $\spec \HNeu$, then
  there exists $r_1>0$ such that, for all $z \in \C$ with $0 <
  \abs{z-\lambda} \le r_1$, the resolvent $\RNeu(z)$ has a first order
  pole at $\lambda$, i.e.,
  \begin{equation*}
    \RNeu(z)
    = \frac 1 {\lambda - z} \1_{\{\lambda\}}(\HNeu) + R_\lambda(z),
  \end{equation*}
  where $R_\lambda(z)$ is the resolvent of $\HNeu \restr{\ker(\HNeu -
    \lambda)^\orth}$, and the same equation is true with $\RNeu(z)$,
  $\1_{\{\lambda\}}(\HNeu)$ and $\map{R_\lambda(z)}\HS \HS$ replaced
  by $\wRNeu(z)$, $\weakop \1_{\{\lambda\}}(\HNeu)$ and
  $\map{\wRNeu_\lambda(z)}{\HS^{-1}}{\HS^1}$.
  % and $\norm {R_z(\HNeu)} \le 1/r_1$.
  By \Prp{dn.z.inv}, we have
  \begin{equation}
    \label{eq:dtn.pole}
    \Lambda(z)^{-1}
    = \Gamma \wRNeu(z) \Gamma^*
    = \frac 1 {\lambda - z} \Gamma \weakop \1_{\{\lambda\}}(\HNeu) \Gamma^*
      + \Gamma \wRNeu_\lambda(z) \Gamma^*
  \end{equation}
  as operator in $\HSaux$.  In particular, $\Lambda(z)^{-1}$ has a
  first order pole at $z=\lambda$ and
  \begin{multline*}
    \norm{(\lambda-z)\Lambda(z)^{-1}}
    \le \normsqr[1 \to 0] \Gamma
    \bigl(\norm[-1 \to 1]{\weakop \1_{\{\lambda\}}(\HNeu)}
    + \abs{\lambda-z} \norm[-1 \to 1]{\wRNeu_\lambda(z)}
    \bigr) = \normsqr[1 \to 0] \Gamma C_\lambda(z),
    \\\text{where}\qquad
    C_\lambda(z):=
    \abs{\lambda+1} 
    + \abs{\lambda-z} \sup_{\lambda' \in \spec \HNeu \setminus \{\lambda\}}
    \frac {\abs{\lambda'+1}} {\abs{z - \lambda'}}
  \end{multline*}
  Note that $C_\lambda(z)< \infty$ since $\lambda$ is isolated in
  $\spec \HNeu$ and since the fraction tends to $1$ as $\lambda' \to
  \infty$.  Since $z \to C_\lambda(z)$ is continuous, $C_\lambda :=
  \sup C_\lambda(\set {z \in \C} {\abs{z - \lambda} \le r_1}) <
  \infty$.
  
  We have therefore shown that $z \notin \spec {\Lambda(\cdot)}$ for
  $0<\abs z \le r_1 $.  For the last assertion, it remains to show
  that $\lambda$ is an eigenvalue of $\Lambda(\lambda)$, but this
  follows immediately from~\itemref{krein1.i}.
\end{proof}
%----------------------------------------------------------------------

If we assume elliptic regularity for the boundary pair then we can
conclude the following \emph{stronger} spectral relations:
%----------------------------------------------------------------------
\begin{theorem}
  \label{thm:krein2}
%   \addtocounter{equation}{-1} %
  Assume that $(\Gamma,\HSaux)$ is an \emph{elliptically regular}
  boundary pair associated with the quadratic form $\qf h$ and let
  $\lambda \in \C \setminus \spec \HDir$.  Then the following
  assertions are true:
%  \begin{subequations}
    \begin{myenumerate}
    \item
      \label{krein2.ii}
      The spectral relation
      \begin{equation}
        \label{eq:krein2.ib}
        \lambda \in \spec \HNeu
             \quad \Leftrightarrow \quad
        0 \in \spec{\Lambda(\lambda)}.
      \end{equation}
      holds for the entire spectrum.

    \item
      \label{krein2.ia}
      Assume that $\map \Gamma {\HS^1} \HSaux$ is a compact operator
      (see \Thm{nd.comp} for equivalent characterisations) then
      $\essspec \HNeu=\essspec \HDir$, and $\spec \HNeu \setminus
      \spec \HDir$ consists of discrete eigenvalues of $\HNeu$, only.
      In particular, if $\lambda \notin \spec \HDir$, then the
      relation~\eqref{eq:krein2.ib} is true for the discrete spectrum,
      i.e., with $\spec \cdot$ replaced by $\spec[disc] \cdot$.

    \item
      \label{krein2.iii}
      An eigenvalue $\lambda$ is isolated in the spectrum of $\HNeu$
      iff $\lambda$ is isolated in the spectrum of the operator pencil
      $\Lambda(\cdot)$, i.e.,
      \begin{equation*}
        \text{$\lambda$ is isolated in $\spec \HNeu$}
        \quad \Leftrightarrow \quad
        \text{$\lambda$ is isolated in the operator pencil spectrum 
            $\spec {\Lambda(\cdot)}$}.
      \end{equation*}
      If one of the conditions is fulfilled, then the
      estimate~\eqref{eq:krein1.x} holds.
      
    \item
      \label{krein2.v}
      Assume additionally, that the boundary pair is \emph{positive},
      then
      \begin{equation*}
        \text{$\lambda$ is isolated in $\spec \HNeu$}
        \quad \Leftrightarrow \quad
        \text{$0$ is isolated in the individual spectrum 
          $\spec {\Lambda(\lambda)}$}
      \end{equation*}
      (i.e., isolated in the spectrum of the \emph{individual}
      operator $\Lambda(\lambda)$).  In particular, the spectral
      relation~\eqref{eq:krein2.ib} is also true for the discrete and
      essential spectrum, i.e., with $\spec \cdot$ replaced by
      $\spec[disc] \cdot$ resp.\ $\spec[ess] \cdot$.
    \end{myenumerate}
%  \end{subequations}
\end{theorem}
%----------------------------------------------------------------------
\begin{proof}
  \itemref{krein2.ii}~Note first that by elliptic regularity,
  $\Lambda(\lambda)$ is closed for all $\lambda \in \C \setminus \spec
  \HDir$ (i.e., also for $\lambda \in \spec \HNeu$) by
  \Thmenum{ell.bd2.impl}{ell.bd2.impl.iii}, and it makes sense to
  speak of $\spec{\Lambda(\lambda)}$.

  ``$\Rightarrow$'': We argue by contraposition: If $0 \notin
  \spec{\Lambda(\lambda)}$, then $\Lambda(\lambda)^{-1}$ exists and is
  bounded.  Moreover, the limit of
  $(\Lambda(z)^{-1}-\Lambda(\lambda)^{-1})/(z-\lambda)$ as $z \to
  \lambda$ exists in $\BdOp \HSaux$ by \Prpenum{bd2.b.z}{bd2.b.z.vi}
  (the boundary pair is elliptic!).  Note that we only know $\lambda
  \notin \spec \HDir$ here, but we do not know yet that $\lambda
  \notin \spec \HNeu$, therefore we cannot apply the formula of
  \Prpenum{nd.z}{nd.z.i} here.

  In particular, the map $\map {\Lambda(z)^{-1}}\HSaux \HSaux$ is
  holomorphic in $z \in \C \setminus \spec \HDir$ and also in
  $z=\lambda$.  Since also $\lambda \notin \spec \HDir$ and since
  again the boundary pair is elliptically regular, $\map{\eSDir(z)}
  \HSaux \HS$ is bounded and also depends holomorphically on $z$ near
  $\lambda$.  Therefore, the operator
  \begin{equation}
    \label{eq:krein.res'}
    \map{\RDir(z)+\eSDir(z)\Lambda(z)^{-1}\SDir(\conj z)^*}
    \HS \HS
  \end{equation}
  depends holomorphically on $z$ and is still a bounded operator as $z
  \to \lambda$.  By \Thmenum{krein.res}{krein.res.iii}, this operator
  equals $\RNeu(z)$ for $z \in \C \setminus \spec \HNeu$.  As
  $\RNeu(z)$ is then still a bounded operator on $\HS$ for $z \to
  \lambda$, we have $\lambda \notin \spec \HNeu$ (the resolvent set
  $\C \setminus \spec \HNeu$ is the \emph{maximal} set of holomorphy,
  see \cite[Thm. III.6.7]{kato:66}).

  ``$\Leftarrow$'' was already shown in \Thmenum{krein1}{krein1.i0}.

  \itemref{krein1.ia}~By \Thm{nd.comp}, $\Lambda^{-1}$ is compact, and
  by~\eqref{eq:krein.res} and the elliptic regularity, it follows that
  the resolvent difference $\RNeu - \RDir$ is compact, too.  In
  particular, the essential spectra agree.  The spectral relation
  follows from \Thmenum{krein1}{krein1.i}, since the spectrum now
  consists of eigenvalues of finite multiplicity only.

  \itemref{krein2.iii}~That $\lambda$ is an eigenvalue of $\spec
  \HNeu$ iff $\ker \Lambda(\lambda)$ is nontrivial follows already
  from \Thmenum{krein1}{krein1.i}.

  ``$\Rightarrow$'' follows from \Thmenum{krein1}{krein1.v}, but a
  shorter proof is as follows: If $\lambda$ is isolated in $\spec
  \HNeu$, then $\Lambda(z)^{-1}$ exists and is bounded by
  \Prpenum{dn.z.inv}{dn.z.inv.iv} for $z \in \C \setminus \spec \HDir$
  with $0<\abs{z-\lambda}$ small enough, i.e., $\lambda$ is isolated
  in $\spec{\Lambda(\cdot)}$.

  ``$\Leftarrow$'' Using~\eqref{eq:krein.res'}, Krein's resolvent
  formula~\eqref{eq:krein.res} (and the elliptic regularity), we
  conclude that $\RNeu(z)$ is defined for $z$ with $0 <
  \abs{z-\lambda}$ small enough, and also has a pole at $z=\lambda$,
  hence $\lambda \in \spec \HNeu$ is isolated.

  \itemref{krein2.v}~is a consequence of~\itemref{krein2.iii} and
  \Prp{op.pencil}.  For the representation~\eqref{eq:ass.op.pencil} we
  refer to~\eqref{eq:dtn.pole}.
\end{proof}
%----------------------------------------------------------------------

%----------------------------------------------------------------------
\begin{remark}
  \label{rem:krein2}
  Let us comment on the elliptic regularity condition and some other
  aspects:
  \begin{myenumerate}
  \item
    \label{rem.krein2.ii}
    The elliptic regularity condition for the spectral equivalence
    in~\Thmenum{krein2}{krein2.ii} is needed for the implication
    ``$\Rightarrow$'': We give a counterexample of a non-elliptically
    regular boundary pair in \Ex{non-ell.bd2.spec}, where $0 \in \spec
    \HNeu$ but $0 \notin \spec {\Lambda(0)}$.

    For the opposite implication it is enough to assume that
    $\Lambda(\lambda)$ is closed.

  \item Without the elliptic regularity assumption, the conclusion of
    \Thmenum{krein2}{krein2.ia} is also generally false: From Krein's
    resolvent formula~\eqref{eq:krein.res0} the compactness of
    $\Lambda^{-1}$ does not in general imply that $\RNeu - \RDir$ is
    compact: In \Ex{dir.comp.neu.not} we have a non-elliptically
    regular boundary pair for which $\RDir$ is compact, but $\RNeu$ is
    not, even though $\Lambda^{-1}$ may be be compact.

  \item
    \label{rem.krein2.iii}
    The implication ``$0 \in \spec {\Lambda(\lambda)}$ isolated
    $\Rightarrow$ $\lambda \in \spec \HNeu$ isolated'' in
    \Thmenum{krein2}{krein2.v} is generally false for non-elliptically
    regular boundary pairs: In \Ex{non-ell.bd2.dis.spec} we give an
    example of a boundary pair which is positive, but not elliptically
    regular.  For this example, $0 \in \spec {\Lambda(0)}$ is
    isolated, while $0 \in \spec \HNeu=[0,\infty)$ is not.

  \item Note that the spectral characterisations of
    \Thms{krein1}{krein2} are void if $\spec \HNeu = \spec \HDir =
    [0,\infty)$ as in the example of a non-compact cylindrical
    manifold in \Sec{cyl.mfd}.  In this case, the \DtN\ operator is not
    defined for $\lambda \in [0,\infty)$.

  \item
    \label{special.d2n}
    For bounded and positive boundary pairs (hence for ordinary
    boundary triples, see \Thm{bd3.intro}), there is also a
    characterisation for the absolutely and singular continuous
    spectrum (see~\cite{bgp:08}) if the \DtN\ operator has the special
    form
    \begin{equation*}
      \Lambda(z) = \frac {\dlapl - m(z)}{n(z)},
    \end{equation*}
    where $\dlapl$ is a \emph{bounded}, self-adjoint operator on
    $\HSaux$ and where $m,n$ are functions holomorphic on $\C
    \setminus \spec \HDir$.  We believe that this assertion remains
    true for elliptically regular and positive boundary pairs (and
    maybe even in more general settings), but where $\dlapl$ may be
    unbounded.  We hope to come back to the analysis of the absolutely
    continuous spectrum in a forthcoming publication.
  \end{myenumerate}
\end{remark}
%----------------------------------------------------------------------

%----------------------------------------------------------------------
% eeee
\section{Boundary pairs constructed from other boundary pairs}
\label{sec:bd2.constr}
%
%----------------------------------------------------------------------
In this section, we give classes of of boundary pairs constructed from
others, such as Robin-type perturbations (where we change the
quadratic form with which the boundary pair is associated), coupled
boundary pairs or the bounded modification of an unbounded boundary
pair (where we change the boundary space).

%----------------------------------------------------------------------
\subsection{Robin boundary conditions}
\label{sec:robin}
%----------------------------------------------------------------------
We start explaining how to use our concept of boundary pairs also for
more general ``boundary conditions'' than Neumann.  The basic idea is
to change the underlying quadratic form $\qf h$, but leave the
boundary pair $(\Gamma,\HSaux)$ as it is.  For simplicity, we consider
only constants $a$ and no operators on $\HSaux$, here.

Let $(\Gamma,\HSaux)$ be a boundary pair associated with a quadratic
form $\qf h$.  For $a \ge 0$, we define
\begin{equation}
  \label{eq:qf.robin}
  \qf h_a(f) := \qf h(f) + a \normsqr[\HSaux]{\Gamma f}
\end{equation}
for $h \in \dom \qf h_a:= \HS^1$.  Since
\begin{equation*}
  (\qf h + \qf 1)(f)
  \le \qf (h_a + \qf 1)(f)
  \le (1+a \normsqr \Gamma) (\qf h + \qf 1)(f),
\end{equation*}
the norms associated with $\qf h$ and $\qf h_a$
(see~\eqref{eq:norm.qf}) are equivalent, hence $\qf h_a$ is also a
closed quadratic form on $\HS^1$.  We will now derive the objects
arising from the boundary pair $(\Gamma,\HSaux)$ associated with $\qf
h_a$, denoted with a subscript $(\cdot)_a$:

%----------------------------------------------------------------------
\begin{proposition}
  \label{prp:bd2.robin}
  \indent
  \begin{myenumerate}
  \item
    \label{bd2.robin.i}
    The Dirichlet operator is unchanged, i.e.,
    $\HS^{1,\Dir}_a=\HS^{1,\Dir}=\ker \Gamma$ and $\HDir_a=\HDir$.
  \item
    \label{bd2.robin.ii}
    The Neumann operator $\HNeu_a$ has domain
      \begin{equation*}
        \dom \HNeu_a
        = \bigset{f \in \WS}
        {\Gamma'f+a\Gamma f=0},
      \end{equation*}
      where $\WS$ is the domain of $\Gamma'$ for the maximal boundary
      triple associated with $(\Gamma,\HSaux)$ (see \Def{ass.bd3} and
      \Sec{rel.bd3} for the notation).
  \item The range of the boundary map is unchanged, as well as the
    Dirichlet solution operator; i.e.,
    $\HSaux_a^{1/2}=\HSaux^{1/2}=\ran \Gamma$ and
    $\SDir_a(z)=\SDir(z)$.
  \item The \DtN\ form $\qf l_{z,a}$ for the boundary pair associated
    with $\qf h_a$ is given by
    \begin{equation*}
      \qf l_{z,a}(\phi) = \qf l_z(\phi) + a \normsqr[\HSaux] \phi
      \quadtext{or, in operator form,}
      \wLambda_a(z) = \wLambda(z)+a.
    \end{equation*}
  \item The boundary pair associated with $\qf h_a$ is elliptically
    regular resp.\ positive iff the boundary pair associated with $\qf
    h$ is.

  \item
    \label{nd.z.v}
    We have $\RNeu \ge \RNeu_a \ge \RDir$.  In particular, if $\map
    \RNeu \HS \HS$ is compact, then $\map {\RNeu_a} \HS \HS$ and $\map
    \RDir \HS \HS$ are also compact, and the eigenvalues fulfil
    \begin{equation*}
      \lambda_k(\HNeu) \le \lambda_k(\HNeu_a) \le \lambda_k(\HDir)
    \end{equation*}
    (labelled in increasing order respecting their multiplicity).
  \end{myenumerate}
\end{proposition}
%----------------------------------------------------------------------
\begin{proof}
  We only indicate some of the arguments here: For the domain
  inclusion ``$\subset$'' of~\itemref{bd2.robin.ii} note that $f \in
  \dom \HNeu_a$ implies that there is $h \in \HS$ such that
  \begin{equation*}
    \qf h(f,g) + a\iprod{\Gamma f}{\Gamma g} = \iprod h g
  \end{equation*}
  for all $g \in \HS^1$.  If we assume $g \in \dom \Hmin=\dom \HDir
  \cap \dom \HNeu$ in the last equation, then the boundary term
  vanishes and $\qf h(f,g)=\iprod f {\HNeu g}$, hence $f \in \dom
  \Hmax$ and $\Hmax f = h$.  Moreover, the defining equation for $\WS$
  of the associated maximal boundary triple (see the forthcoming
  publication~\cite{post:pre14a} for details)
  % $\WS^{1,\max}_0$ in~\eqref{eq:bd3.ex.w}
  is fulfilled.
  %(again, for $g \in \HS^{1,\Dir}$, the boundary term vanishes).
  Finally, comparing the above formula with Green's first
  identity~\eqref{eq:green}, we see that $\Gamma' f = -a \Gamma f$.
  Therefore, we have shown that $f \in \WS$.

  \itemref{nd.z.v}~For the last assertion note that $(0 \le) \qf h \le
  \qf h_a \le \qf h^\Dir$ in the sense of quadratic forms
  (cf.~e.g.~\cite[Sec.~4.4]{davies:95}).  Therefore, $\RNeu \ge
  \RNeu_a \ge \RDir (\ge 0)$, i.e., $\RNeu_a$ and $\RDir$ are also
  compact.
\end{proof}
%----------------------------------------------------------------------

The following proposition is useful when proving statements for
$\lambda$ \emph{inside} the Neumann spectrum $\HNeu$ stating that one
can always find an $a>0$ such that $\lambda$ is not in the spectrum of
$\HNeu_a$, even if $\lambda \in \spec \HNeu$, provided the Dirichlet
and Neumann spectra are purely discrete.
%----------------------------------------------------------------------
\begin{proposition}
  \label{prp:robin.nspec}
  Let $(\Gamma,\HSaux)$ be a boundary pair and $\lambda \in
  [0,\infty)$.  Assume in addition that $\RNeu$ is compact.  If
  $\lambda \notin \spec \HDir$ then there exists $a_0>0$ such that
  $\lambda \notin \spec {\HNeu_a}$ for all $a\ge a_0$.
\end{proposition}
%----------------------------------------------------------------------
\begin{proof}
  The operator
  \begin{equation*}
    \Lambda^{1/2} \Gamma \RNeu^{1/2}
    \colon \HS \stackrel {\RNeu^{1/2}} \longrightarrow
    \HS \stackrel \Gamma \longrightarrow
    \HSaux^{1/2} \stackrel {\Lambda^{1/2}} \longrightarrow
    \HSaux
  \end{equation*}
  is bounded and $\RNeu^{1/2}$ is compact as operator in $\HS$, since
  $\RNeu$ is compact.  In particular, the operator $\map{\Lambda^{1/2} \Gamma \RNeu
    = \Lambda^{1/2} \Gamma \RNeu^{1/2} \RNeu^{1/2}} \HS \HSaux$ is
  compact.  By~\cite[Thm.~2.6]{bbb:11}, we have $\norm{\RNeu_a -
    \RDir} \to 0$ as $a \to \infty$, i.e., $\HNeu_a$ converges in norm
  resolvent sense to $\HDir$.  This implies in particular, that if
  $\lambda \notin \spec \HDir$, then there exists $a_0>0$ such that
  $\lambda \notin \spec {\HNeu_a}$ for all $a \ge a_0$ (see
  e.g.~\cite[Thm~VIII.23]{reed-simon-1}).
\end{proof}
%----------------------------------------------------------------------

%----------------------------------------------------------------------
\subsection{Coupled boundary pairs}
\label{sec:coupl.bd2}
%----------------------------------------------------------------------
We present in this subsection two procedures of coupling boundary
pairs.  Such couplings have already been treated e.g.\ in
\cite{dhms:00}.

Assume that $(\Gamma_i,\HSaux)$ is a boundary pair associated with
$\qf h_i$ ($\dom \qf h_i = \HS_i^1$) in the Hilbert space $\HS_i$ for
$i=1,2$.  Note that the boundary space is the \emph{same} for both
boundary pairs.  We assume additionally that
\begin{equation}
  \label{eq:bd2.coupl}
  \HSaux^{1/2} := \ran \Gamma_1 \cap \ran \Gamma_2
  \quad\text{is dense in $\HSaux$.}
\end{equation}
We set $\HS:=\HS_1 \oplus \HS_2$ and $\HS^{1,\dec}:=\HS^1_1 \oplus
\HS^1_2$.  It follows easily from the boundedness of
$\map{\Gamma_i}{\HS_i^1} \HSaux$ that
\begin{equation}
  \label{eq:hs1.coupl}
  \HS^1 := \bigset {f \in \HS^{1,\dec}} {\Gamma_1 f_1 = \Gamma_2 f_2}
\end{equation}
is a closed subspace of $\HS^{1,\dec}$, and $\qf h := (\qf h_1 \oplus
\qf h_2) \restr{\HS^1}$ is a non-negative, closed form in $\HS$ with
associated operator $H$.  We call $\qf h$ the \emph{coupled form}
obtained from $\qf h_1$ and $\qf h_2$.

Set
\begin{equation*}
  \map \Gamma {\HS^1} \HSaux,\quad
  \Gamma f := \Gamma_1 f_1 = \Gamma_2 f_2.
\end{equation*}

%----------------------------------------------------------------------
\begin{proposition}
  \label{prp:bd2.coupl}
  Assume that $(\Gamma_i,\HSaux)$ are boundary pairs for $i=1,2$ such
  that the density condition~\eqref{eq:bd2.coupl} holds.  Then the
  following assertions are true:
  \begin{myenumerate}
  \item
    \label{bd2.coupl.i}
    The pair $(\Gamma,\HSaux)$ is a boundary pair associated with the
    coupled quadratic form $\qf h$, called here the \emph{(Neumann-)coupled
      boundary pair}.
  \item
    \label{bd2.coupl.ii}
    The Dirichlet operator associated with the coupled boundary pair
    is decoupled, i.e., $\HDir = \HDir_1 \oplus \HDir_2$, while the
    Neumann operator (the operator associated with $\qf h$) is (in
    general) coupled.  Moreover, the Dirichlet solution operator and
    the \DtN\ operator of the coupled boundary pair are given by
    \begin{equation*}
      \SDir(z)\phi = \SDir_1(z) \phi \oplus \SDir_2(z) \phi
      \quadtext{and}
      \wLambda(z) \phi
      = \wLambda_1(z) \phi + \wLambda_2(z) \phi
    \end{equation*}
    for $\phi \in \HSaux^{1/2}=\ran \Gamma$, where $z \in \C \setminus
    \spec \HDir = \C \setminus (\spec {\HDir_1} \cup \spec
    {\HDir_2})$.
  \item 
    \label{bd2.coupl.iii}
    We have
    \begin{equation*}
      \normsqr[\HSaux^{1/2}] \phi 
      := \normsqr[\HS^1]{S\phi}
      = \normsqr[\HS_1^1]{S_1\phi} + \normsqr[\HS_2^1]{S_2\phi}
      \ge \normsqr[\HSaux_i^{1/2}] \phi,
    \end{equation*}
    i.e., the embedding $\HSaux_i^{1/2} \hookrightarrow \HSaux^{1/2} =
    \ran \Gamma$ is bounded for $i=1,2$.  If in addition $\ran
    \Gamma_1=\ran \Gamma_2$, then the embedding is surjective, and the
    norms on $\HSaux^{1/2}$, $\HSaux_1^{1/2}$ and $\HSaux_2^{1/2}$ are
    mutually equivalent.
  \item
    \label{bd2.coupl.iv}
    If the boundary pairs $(\Gamma_i,\HSaux)$ are elliptically regular
    resp.\ positive, then the coupled boundary pair $(\Gamma,\HSaux)$
    is elliptically regular resp.\ positive.
  \item
    \label{bd2.coupl.v}
    Krein's resolvent formula in this context reads as
    \begin{equation}
      \label{eq:krein.res.coupl}
      \map{\wRNeu(z) = \iota_1 \wRDir_1(z) \iota_1^* 
        \oplus \iota_2 \wRDir_2(z))\iota_2^*
        + \SDir(z) \wLambda(z)^{-1} \SDir(\conj z)^*}
      {\HS^{-1}} {\HS^1}
    \end{equation}
    (with $\map{\wRDir_i(z)=(\wHDir_i -
      z)^{-1}}{\HS^{-1,\Dir}_i}{\HS^{1,\Dir}_i}$ and $\embmap
    {\iota_i} {\HS^{1,\Dir}_i} {\HS^1_i}$), i.e., the resolvent of the
    coupled operator can be expressed by operators of the individual
    boundary pairs only, namely, the direct sum of the Dirichlet
    resolvents and a coupling term.
  \end{myenumerate}
\end{proposition}
%----------------------------------------------------------------------
\begin{proof}
  \itemref{bd2.coupl.i}~The boundedness of $\Gamma$ is obvious, as well
  as the density of
  \begin{equation*}
    \HS^{1,\Dir} := \ker \Gamma 
    = \ker \Gamma_1 \oplus \ker \Gamma_2
    = \HS^{1,\Dir}_1 \oplus \HS^{1,\Dir}_2.
  \end{equation*}
  Moreover, $\ran \Gamma=\HSaux^{1/2}$ is dense in $\HSaux$ by
  assumption~\eqref{eq:bd2.coupl}. 

  \itemref{bd2.coupl.ii}~That $\HDir$ is decoupled is obvious, as well
  as the formula for the coupled Dirichlet solution operator.  The
  formula for the coupled solution operator is obvious.  The
  corresponding Neumann operator is (in general) \emph{coupled} (i.e.,
  not a direct sum of the individual Neumann operators) For the
  coupled \DtN\ operator, note that
  \begin{align*}
    \iprod{\wLambda(z) \phi} \psi
    &= (\qf h - z \qf 1)(\SDir(z),g)\\
    &= (\qf h_1 - z \qf 1)(\SDir_1(z),g_1)
    + (\qf h_1 - z \qf 1)(\SDir_1(z),g_2)%\\
    = \iprod{\wLambda_1(z) \phi} \psi
    + \iprod{\wLambda_2(z) \phi} \psi
   \end{align*}
   for $\phi \in \HSaux^{1/2}$ and any $g=g_1 \oplus g_2 \in \HS^1$
   with $\Gamma g = \psi$ (see~\eqref{eq:def.dn.z} and \Def{dn.z}).
   
   \itemref{bd2.coupl.iii}~The equivalence of the norms follows from the
   open mapping theorem (a bounded bijective operator has also a
   bounded inverse).

   \itemref{bd2.coupl.iv}~The last assertion is also obvious, using
   \Defs{ell.bd2}{pos.bd2}.   We have e.g.\
   \begin{equation*}
     \qf q(\phi) 
     = \normsqr[\HS] {\SDir \phi}
     = \normsqr[\HS_1] {\SDir_1 \phi} + \normsqr[\HS_1] {\SDir_1 \phi}
     = \qf q_1(\phi) + \qf q_2(\phi)
     \le (C_1 + C_2) \normsqr[\HSaux] \phi
   \end{equation*}
   if $\qf q_1$, $\qf q_2$ resp.\ $C_1$, $C_2$ are the solution forms
   resp.\ constants in the estimate of \Def{ell.bd2} for the
   individual boundary pairs.
\end{proof}
%----------------------------------------------------------------------

In many applications, the RHS of Krein's resolvent
formula~\eqref{eq:krein.res.coupl} in the coupled case can be
calculated explicitly, hence we have a formula for the resolvent of
the coupled operator (see \Rem{coupl.mfd} for an example).

There is another way of coupling two boundary pairs: Let $\qf h^\dec =
\qf h_1 \oplus \qf h_2$ with domain $\dom \qf h^\dec =
\HS^{1,\dec}=\HS_1^1 \oplus \HS_2^2$.  As boundary operator, we define
here
\begin{equation*}
  \wt \Gamma f := \Gamma_1 f_1 - \Gamma_2 f_2.
\end{equation*}
It is again easily seen that $(\wt \Gamma, \HSaux)$ is a boundary pair
associated with $\qf h^\dec$.  Then the associated Neumann operator is
$\wtHNeu = \HNeu_1 \oplus \HNeu_2$, hence decoupled.  Moreover, $\ker
\wt \Gamma$ equals $\HS^1$ defined in~\eqref{eq:hs1.coupl}, and the
Dirichlet operator $\wtHDir$ associated with this boundary pair is the
coupled operator.  We call this boundary pair the
\emph{Dirichlet-coupled boundary pair}, since the Dirichlet operator
is coupled here.

It is now straightforward to calculate the associated Dirichlet
solution operators and the \NtD\ operator of the coupled boundary pair
as
\begin{equation}
  \label{eq:bd2.coupl.dir}
  \wt S \phi
  = \SDir_1(\wLambda_1+ 
       \wLambda_2)^{-1} \wLambda_2 \phi 
  \oplus \SDir_2(\wLambda_1+ \wLambda_2)^{-1} 
      \wLambda_1 \phi 
  \quadtext{and}
  \weakop {\wt \Lambda}(z)^{-1} \phi
  = \wLambda_1(z)^{-1} \phi + \wLambda_2(z)^{-1} \phi
\end{equation}
for $\phi \in \wt \HSaux^{1/2} = \HSaux_1^{1/2} + \HSaux_2^{1/2}$.

%----------------------------------------------------------------------
\subsection{Direct sum of boundary pairs}
%----------------------------------------------------------------------
\label{sec:sum.bd2}

Another way of obtaining a new boundary pair from two boundary pairs
$(\Gamma_i,\HSaux_i)$ associated with $\qf h_i$ on $\HS_i$ ($i=1,2$)
is by taking the direct sum of all objects, i.e., $\HS:=\HS_1 \oplus
\HS_2$, $\HSaux:= \HSaux_1 \oplus \HSaux_2$, $\Gamma := \Gamma_1
\oplus \Gamma_2$ etc.  We call this boundary pair the \emph{direct
  sum} of the boundary pairs $(\Gamma_1,\HSaux_1)$ and
$(\Gamma_2,\HSaux_2)$.  The corresponding derived objects and the
properties of the direct sum can easily be derived; e.g. $\Lambda(z) =
\Lambda_1(z) \oplus \Lambda_2(z)$ and its spectrum is the union of the
spectra of $\Lambda_i(z)$.  Note that the direct sum is different from
the coupled pairs defined in \Sec{coupl.bd2}.

%----------------------------------------------------------------------
\subsection{Regularisation: Making a boundary pair bounded}
%----------------------------------------------------------------------
\label{sec:bdd.bd2}

Let us finally define a bounded boundary pair $(\wt \Gamma,\wt
\HSaux)$ constructed from an unbounded boundary pair $(\Gamma,\HSaux)$
associated with $\qf h$ as follows: We set
\begin{equation*}
  \wt \HSaux := \HSaux^{1/2} \quadtext{and}
  \map{\wt \Gamma} {\HS^1} {\wt \HSaux},
\end{equation*}
where $\wt\HSaux$ is endowed with the norm $\norm[\wt \HSaux]\phi =
\norm[\HSaux^{1/2}] \phi = \norm[\HS^1]{\SDir\phi}$, i.e., we just
change the range space of $\Gamma$, and obviously, $\ran \wt
\Gamma=\wt \HSaux$, i.e., $(\wt \Gamma, \wt \HSaux)$ is a
\emph{bounded} boundary pair.  For the new boundary pair, called the
\emph{bounded modification} or \emph{regularisation of
  $(\Gamma,\HSaux)$}, we have $\norm[1 \to 0]{\wt \Gamma} = 1$.
Moreover its weak Dirichlet solution operator and \DtN\ operator are
given as follows:

%----------------------------------------------------------------------
\begin{proposition}
  \label{prp:bd2.bdd}
  Assume that $(\Gamma,\HSaux)$ is an unbounded boundary pair
  associated with a quadratic form $\qf h$.  Denote by $(\wt \Gamma,
  \wt \HSaux)$ its bounded modification, given by $\wt \HSaux =
  \HSaux^{1/2}$, $\map {\wt \Gamma}{\HS^1}{\wt \HSaux}$, $\wt \Gamma
  f=\Gamma f$, where the objects without tilde refer to
  $(\Gamma,\HSaux)$ and the objects with tilde refer to $(\wt
  \Gamma,\wt \HSaux)$.  Then the following assertions are true:
  \begin{myenumerate}
  \item The Neumann and Dirichlet operators remain unchanged, i.e.,
    $\wtHDir=\HDir$ and $\wtHNeu=\HNeu$.

  \item We have
    \begin{equation*}
      \map{\wt \SDir(z)}{\wt \HSaux} {\HS^1}, \quad
      \wt \SDir(z) \phi = \SDir(z) \phi, \qquad
      \wt \Lambda = \id_{\wt \HSaux} \quadtext{and}
      \map{\wt \Lambda(z) = \wLambda^{-1}\wLambda(z)}{\wt \HSaux}{\wt \HSaux}.
    \end{equation*}
  
  \item The boundary pair $(\wt \Gamma, \wt \HSaux)$ is bounded and in
    particular elliptically regular.  Moreover $\map {\wt S(z)} {\wt
      \HSaux} \HS$ and $\map{\wt \Lambda(z)}{\wt \HSaux}{\wt \HSaux}$
    are bounded operators, the norm of the latter is bounded by $L(z)$
    (cf.~\eqref{eq:def.m.z}).
  \item If $(\Gamma, \HSaux)$ is not positive, then $(\wt \Gamma, \wt
    \HSaux)$ is not either.
  \end{myenumerate}
\end{proposition}
%----------------------------------------------------------------------

%----------------------------------------------------------------------
\begin{remark}
  \label{rem:why.bd2.useful}
  Note that although we could only work with bounded boundary pairs,
  there is not always an associated ordinary boundary triple (for this
  we need that the new (and hence the old) boundary pair is positive,
  see \Thm{bd3.intro}).  The bounded modification of an unbounded
  boundary pair is obviously elliptically regular (because it is
  bounded), but not necessarily positive (in particular the bounded
  modification is not positive if the original boundary pair is not).
  In \Ex{bdd.non-pos} we present a bounded modification of a boundary
  pair which is not positive.

  Moreover, the unbounded boundary pair is in many examples more
  ``natural'' like in the manifold example in \Sec{lapl.mfd} since the
  regularised boundary pair involves the \DtN\ operator in the norm of
  the new boundary space (see also the second last paragraph of
  \Sec{ext.th.intro}: ``Applications of \dots'').
\end{remark}
%----------------------------------------------------------------------

%----------------------------------------------------------------------
% ffff
\section{Examples}
\label{sec:examples}
%
%----------------------------------------------------------------------

All our examples
% (except the trivial ones of the next
% subsection) 
are of the form $\HS=\Lsqr {X,\mu}$ and $\HSaux := \Lsqr {Y,\nu}$
where $(X,\mu)$ and $(Y,\nu)$ are measure spaces and $Y \subset X$, as
explained in \Ex{basic.ex}.

%----------------------------------------------------------------------
\subsection{Examples with finite-dimensional boundary space}
\label{sec:2dim}
%----------------------------------------------------------------------

We treat here a simple example where $X=I$ is a compact interval and
$Y=\bd I$ consists of two points only.  The corresponding boundary
space is two-dimensional and the boundary pair is bounded and
positive, hence associated with an ordinary boundary triple (see
\Thm{bd3.intro}).

More precisely, let $I=[0,\ell]$ for some $\ell \in (0,\infty)$ and
set $\HS:= \Lsqr I$, $\HS^1 := \Sob X$, $\qf h(f) := \normsqr[\Lsqr
I]{f'}$.  As boundary operator, we choose $\Gamma f = (f(0),
f(\ell))$.  It follows now from standard assertions on Sobolev spaces
that $(\Gamma,\HSaux)$ is a boundary pair.  Moreover, the Neumann and
Dirichlet operators are the usual Neumann and Dirichlet Laplacians on
$[0,\ell]$, and the Dirichlet solution operator is given by
\begin{equation}
  \label{eq:dsol.2dim}
  \SDir(z) \phi
  = \phi_0 \frac{\sin (\sqrt z(\ell - s))}{\sin (\sqrt z \ell)}
  + \phi_1 \frac{\sin (\sqrt z s)}{\sin (\sqrt z \ell)}
\end{equation}
for $z \notin \spec \HDir = \set{k^2 \pi^2/\ell^2}{k=1,2,\dots}$,
where $\phi=(\phi_0,\phi_1) \in \C^2$ and where the complex square
root is suitably chosen.  If $z=0$, we use the continuous extension of
the above expressions.  The \DtN\ operator is represented by the matrix
\begin{equation}
  \label{eq:dtn.2dim}
  \Lambda(z) = \frac {\sqrt z}{\sin (\sqrt z \ell)}
  \begin{pmatrix}
    \cos (\sqrt z \ell) & -1\\
    -1 & \cos (\sqrt z \ell)
  \end{pmatrix}
\end{equation}
with eigenvalues $-\sqrt z \tan(\sqrt z \ell/2)$ and $\sqrt z \cot
(\sqrt z \ell/2)$.  For $z = -\kappa^2<0$ ($\kappa>0$), the former
eigenvalue, i.e., $\kappa \tanh (\kappa \ell/2)$, is smaller than the
latter one, i.e., $\kappa \coth (\kappa \ell/2)$.  The corresponding
eigenvectors are $(1,1)$ and $(-1,1)$.  The eigenvalues of
$\Lambda(0)$ are $0$ and $2/\ell$.  It follows that $\normsqr[1 \to 0]
\Gamma = (\inf \spec \Lambda)^{-1} = (\tanh
(\ell/2))^{-1}=\coth(\ell/2)$ ($\kappa=1$; see \Thmenum{dn}{dn.ii}.

The matrix $Q(z) = \SDir(z)^* \SDir(z)$ has the same eigenvectors and
the eigenvalues for $z=-\kappa^2$ ($\kappa>0$) are 
\begin{equation}
  \label{eq:sol.form.2dim}
  \xi_+^\ell(-\kappa^2) 
  = \frac{\tanh(\kappa\ell/2)} {2\kappa} 
       + \frac \ell{4\cosh^2(\ell \kappa/2)}
  \ge
  \xi_-^\ell(-\kappa^2) 
  = \frac{\coth(\kappa\ell/2)}{2\kappa} 
    - \frac \ell {4 \sinh^2(\ell \kappa/2)}
\end{equation}
with asymptotics $\xi_+^\ell(-1) \approx \ell/2$ and $\xi_-^\ell(-1)
\approx \ell/6$ as $\ell \to 0$ and $\xi_+^\ell(-1) \to 1/2$ and
$\xi_-^\ell(-1) \to 1/2$ as $\ell \to \infty$.  If $z=0$, then
$\xi_+^\ell(0)=\ell/2$ and $\xi_-^\ell(0)=\ell/6$.

We call $(\Gamma,\HSaux)$ the boundary pair associated with
$I=[0,\ell]$ and $\bd I=\{0,\ell\}$.

%----------------------------------------------------------------------
\subsection{Examples with Jacobi operators}
\label{sec:jacobi}
%----------------------------------------------------------------------

We present here a boundary pair which mainly serves as a ``zoo'' of
examples in which $X=[0,\ell)$ and $Y$ is a countable subset of $X$
accumulating only at $\ell \in (0,\infty]$.  It will turn out that the
associated \DtN\ operator (for certain real values of $z$) is actually
a Jacobi operator in $\lsqr \N$ acting as
\begin{equation}
  \label{eq:jacobi.op}
  (J\phi)_n = a_{n-1} \phi_{n-1} + b_n \phi_n +a_n \phi_{n+1}, \qquad
  n=1,2,\dots,
\end{equation}
and $\phi_0=0$.  Here, $a_n$, $b_n$ are suitable real-valued
sequences.  We call $J$ the Jacobi operator associated with $(a_n)_n$
and $(b_n)_n$. 

Note that if $a_n < 0$ and $b_n = -(a_n + a_{n-1})$, then we can
interpret $J$ as a discrete weighted Laplacian with corresponding form
$ \iprod{J\phi} \phi = \sum_{n=1}^\infty (-a_n) \abssqr{\phi_{n+1} -
  \phi_n}$, i.e., we can consider $-a_n$ as a weight of the edge from
vertex $n$ to $n+1$ of the half-line graph $\N$.  If $q_n := b_n + a_n
+ a_{n-1} \ne 0$ then we can interpret $(q_n)_n$ as a discrete
potential, and $J$ is a discrete Schr\"odinger operator with this
potential with associated form
\begin{equation}
  \label{eq:disc.schroe}
  \iprod{J\phi} \phi
  = \sum_{n=1}^\infty (-a_n) \abssqr{\phi_{n+1} - \phi_n}
  + \sum_{n=1}^\infty q_n \abssqr{\phi_n}.
\end{equation}

%----------------------------------------------------------------------
\begin{remark}
  \label{rem:delta.malamud}
  Recently, boundary triple methods have also been used
  in~\cite{kostenko-malamud:10} for the spectral analysis of
  Laplacians with infinitely many delta-interactions in dimension $1$
  which become arbitrarily close.  Kostenko and Malamud couple
  infinitely many boundary triples to a new boundary relation (a
  generalisation of boundary triples, see the text after
  \Thm{bd3.intro}).  Kostenko and Malamud provide a theory allowing
  them to describe self-adjoint extensions of a minimal operator of
  the coupled boundary relation, and these can be parametrised by
  Jacobi operators.  Moreover, they provide conditions under which the
  extension is bounded from below and under which resolvent
  differences are in Schatten-von Neumann classes.  In our context, we
  can understand these Laplacian with delta-interactions as Robin-type
  perturbations of the Neumann operator, and the \DtN\ operator is then
  a Jacobi operator.  We will not analyse such perturbations in this
  work, but Kostenko and Malamud's result could also be recovered by
  our boundary pair method.  There are similar results for Dirac-type
  operators with infinitely many delta-interactions in dimension $1$
  (\cite{cmp:13}) and also for infinitely many delta-interactions in
  dimension $3$ (\cite{malamud-schmuedgen:12}).
\end{remark}
%----------------------------------------------------------------------
Let $I:=[0,\ell)$ for some $\ell \in (0,\infty]$ and set $\HS := \Lsqr
I$.  As quadratic form, we choose $\qf h(f)=\normsqr[\Lsqr I] {f'}$
with domain
\begin{equation*}
  \HS^1 := \bigset{f \in \Sob X} {f(0)=0}.
\end{equation*}
As boundary $Y$, we choose a sequence of points $(x_n)_n$ such that
$x_0=0$, $\ell_n := x_{n+1} - x_n >0$ and $\lim_{n \to \infty} x_n =
\ell$.  We set $I_n := [x_n,x_{n+1}]$.  As boundary space and operator
we set $\HSaux := \lsqr \N$ and $(\Gamma f)_n := \rho_n^{1/2} f(x_n)$,
respectively, where $(\rho_n)_n$ is a sequence of positive numbers.
To simplify some estimates, and to assure that $(\Gamma,\HSaux)$ will
be a boundary pair associated with $\qf h$, we assume that there are
constants $\tau_+$, $\rho_\pm \in (0,\infty)$ such that
\begin{equation}
  \label{eq:jacobi.ass0}
  \tau_+ 
  := \sup_n \frac {\rho_n}2 \coth \Bigl(\frac{\ell_n}2\Bigr)
  < \infty
%  \lim_{n \to \infty} \ell_n = 0, \qquad
%  \ell_n \le \ell_+
  \qquadtext{and}
  \rho_- \le \frac {\rho_n}{\rho_{n+1}} \le \rho_+
\end{equation}
for all $n \in \N$.  From the first condition we can conclude the
following: since $\tanh y \le y$ for $y > 0$, we have $1/y \le
\coth(y/2)/2$ or
\begin{equation}
  \label{eq:quot.rho.ell}
  \frac{\rho_n}{\ell_n}  
  \le \frac {\rho_n}2 \coth (\ell_n/2) 
  \le \tau_+.
\end{equation}
The second condition of~\eqref{eq:jacobi.ass0} allows us to replace
$\rho_{n\pm1}$ by $\rho_n$ in estimates.  We also set
\begin{equation}
  \label{eq:jacobi.conseq}
  \ell_n^-:=\min\{\ell_n,1\},
  \quadtext{hence $\tau_+<\infty$ is equivalent with} 
  \sup_n \frac{\rho_n}{\ell_n^-}<\infty.
\end{equation}
%----------------------------------------------------------------------
\begin{proposition}
  \label{prp:jacobi}
  Assume that~\eqref{eq:jacobi.ass0} holds then the following
  assertions are true:
  \begin{myenumerate}
  \item
    \label{jacobi.i}
    The operator $\map \Gamma {\HS^1} \HSaux$ is bounded and
    $\normsqr[1 \to 0] \Gamma=2\tau_+$; moreover, $(\Gamma,\HSaux)$ is
    a boundary pair associated with $\qf h$.
  \item
    \label{jacobi.ia}
    The associated Neumann operator $\HNeu$ is the Laplacian with
    Dirichlet condition at $0$ and Neumann condition at $\ell$ (if
    $\ell<\infty$).  Its spectrum is purely discrete and given by
    $\set{(k+1/2)^2\pi^2/\ell^2}{k=0,1,\dots}$ if $\ell< \infty$, and
    purely absolutely continuous and given by $\spec \HNeu=[0,\infty)$
    if $\ell=\infty$.
  \item
    \label{jacobi.ii}
    The associated Dirichlet operator is given by $\HDir = \bigoplus_n
    {\Delta_{I_n}^\Dir}$, where $\Delta_{I_n}^\Dir$ denotes the
    Dirichlet operator on the interval $I_n$ acting as
    $\Delta_{I_n}^\Dir f = -f''$.  In particular, $\HDir$ is decoupled
    and has spectrum
    \begin{equation*}
      \spec \HDir
      = \clo{\bigset{(k \pi/\ell_n)^2} {k=1,2,\dots, \;\; n=0,1,\dots}}.
    \end{equation*}
    We can omit the closure if $\ell_n \to 0$.  If $\ell_n \to
    \infty$, then $\spec \HDir = [0,\infty)$.
  \item
    \label{jacobi.iii}
    Assume that $0 \notin \spec \HDir$, then the \DtN\ operator
    $\Lambda(0)$ is a Jacobi operator associated with
    \begin{equation*}
      a_n = a_n(0)
      = -\frac 1 {\ell_n} \cdot \frac 1 {(\rho_n \rho_{n+1})^{1/2}}
      \quadtext{and}
      b_n = b_n(0)
      = \Bigl(\frac 1 {\ell_{n-1}} +  \frac 1 {\ell_n} \Bigr) 
      \cdot \frac 1 {\rho_n}.
    \end{equation*}
  \item
    \label{jacobi.v}
    The boundary pair is bounded iff $\inf_n \ell_n^- \rho_n>0$.  If
    $\lim_n \ell_n=0$, then the boundary pair is unbounded.
  \item
    \label{jacobi.iv}
    The operator $Q$ associated with the solution form $\qf q$ is
    bounded from below and above by a constant times multiplication
    with $(\ell_n^-/\rho_n)_n$.
    % We have $\frac 1 6(1+\rho_-) \frac \ell \rho \le Q \le \frac 1
    % 2(1+\rho_+) \frac \ell \rho$, where $\ell/\rho$ denotes the
    % multiplication operator by $(\ell_n/\rho_n)$, and $Q$ is the
    % operator associated with $\qf q(\phi)=\normsqr[\HS]{\SDir \phi}$.
    % In particular, the boundary pair is positive.
    In particular, the boundary pair is positive.
  \item
    \label{jacobi.vi}
    The boundary pair is elliptically regular iff $\tau_-:= \inf_n
    \rho_n/\ell_n^- > 0$.
  \end{myenumerate}
\end{proposition}
%----------------------------------------------------------------------
\begin{proof}
  Let us denote the objects of the boundary pair associated with $I_n$
  and $\{x_n, x_{n+1}\}$ using the subscript $(\cdot)_{I_n}$ (see
  \Sec{2dim}).

  \itemref{jacobi.i}~We have
  \begin{equation*}
    \normsqr{\Gamma f}
    = \sum_{n=1}^\infty \rho_n \abssqr{f(x_n)}
    \le \sum_{n=1}^\infty \rho_n \coth(\ell_n/2) 
           \normsqr[\Sob {I_n}] f
    \le 2\tau_+ \normsqr[\Sob I] f
  \end{equation*}
  using the optimal bound $\abssqr{f(x_n)} \le \coth (\ell_n/2)
  \normsqr[\Sob {I_n}] f$ from the two-dimensional boundary pair
  $(\Gamma_{I_n},\C^2)$ in \Sec{2dim}.  That $2\tau_+$ is the optimal
  constant follows by a standard argument.  Moreover, it is easily
  seen that $\ker \Gamma = \bigoplus_n \Sobn {I_n}$ is dense in $\Lsqr
  I$ as well as $\ran \Gamma$ is dense in $\lsqr \N$ (the sequences
  with finite support are obviously in $\ran \Gamma$).

  \itemref{jacobi.ia}~is obvious.  \itemref{jacobi.ii}~The form of the
  associated Dirichlet operator is clear.  Note that the set 
  \begin{equation*}
    \set{(k \pi/\ell_n)^2}
    {k=1,2,\dots, n=0,1,\dots} \cap [0,\lambda]
  \end{equation*}
  is finite for any $\lambda >0$ if $\ell_n \to 0$, hence we can omit
  the closure in this case. If $\ell_n \to \infty$, then for given
  $\mu \ge 0$ and $\eps>0$ choose $n \in \N$ such that $\pi/\ell_n <
  \eps$.  Now choose $k \in \N$ such that $k\pi/\ell_n \le \mu <
  (k+1)\pi/\ell_n$.  Clearly, $\abs{\mu-k\pi/\ell_n}<\eps$, hence
  $\bigcup_n\sqrt{\spec \Delta_{I_n}^\Dir}$ is dense in $[0,\infty)$.
  As $\mu \mapsto \mu^2$ is a homeomorphism of $[0,\infty)$ the result
  follows.

  \itemref{jacobi.iii}~ The Dirichlet solution operator is given as
  follows: Let $h=\SDir(z) \phi$ for $\phi \in \HSaux^{1/2}$.  Then
  $h_n := h \restr {I_n} = \SDir_{I_n}(z) \Phi_n$, where $\Phi_n=(\wt
  \phi_n,\wt \phi_{n+1})$ and $\wt \phi_n= \rho_n^{-1/2}\phi_n$.
  Moreover, the \DtN\ operator is given by
  \begin{align*}
    \iprod[\lsqr \N] {\Lambda(z) \phi} \phi
    &= (\qf h - z \qf 1)(\SDir(z) \phi, \SDir \phi)
    = \sum_{n=0}^\infty (\qf h_{I_n} - z \qf 1)
                (\SDir_{I_n}(z) \Phi_n, \SDir_{I_n} \Phi_n)\\
    &= \sum_{n=0}^\infty \iprod[\C^2]{\Lambda_{I_n}(z) \Phi_n} {\Phi_n}
    = \sum_{n=1}^\infty 
        \bigl(
           a_{n-1}(z) \phi_{n-1} +
           b_n(z) \phi_n +
           a_n(z) \phi_{n+1}
        \bigr) \conj {\phi_n}
  \end{align*}
  for suitable $\phi \in \HSaux^{1/2}$, where
  \begin{equation}
    \label{eq:jacobi.a.b}
    a_n(z) = -\frac {\sqrt z}{\sin (\sqrt z \ell_n)}
            \cdot \frac 1 {(\rho_n \rho_{n+1})^{1/2}}
    \quadtext{and}
    b_n(z)
    = \sqrt z \bigl(\cot(\sqrt z \ell_{n-1}) + \cot(\sqrt z \ell_n)\bigr)
            \cdot \frac 1 {\rho_n}
  \end{equation}
  (see \eqref{eq:dtn.2dim}).  The formula for $z=0$ follows by taking
  $z \to 0$.

  % \itemref{jacobi.iiia} The asymptotics for $a_n(z)$, $b_n(z)$ and
  % $q_n(z)$ follow by a straightforward calculation.

  \itemref{jacobi.v}~The boundary pair is bounded iff $\Lambda$ is
  bounded (see \Thmenum{dn}{dn.iii}); and the Jacobi operator
  $\Lambda$ is bounded iff the Jacobi sequences $(a_n(-1))_n$ and
  $(b_n(-1))_n$ are both bounded.  Since $1/\sinh \ell_n \le
  \coth(\ell_n)$ and using the second condition
  of~\eqref{eq:jacobi.ass0}, we see that both Jacobi sequences are
  bounded iff $(b_n(-1))_n$ is bounded, i.e., iff $\sup_n (\coth
  \ell_n)/\rho_n <\infty$, hence iff $\inf_n \ell_n^-\rho_n>0$.

  Since $\ell_n \rho_n \le \tau_+ \ell_n^2$
  by~\eqref{eq:quot.rho.ell}, we have that $\lim_n \ell_n=0$ implies
  $\ell_n^-=\ell_n$ eventually and $\inf_n \ell_n^- \rho_n =0$, i.e.,
  the boundary pair is not bounded.

  \itemref{jacobi.iv} and~\itemref{jacobi.vi}: The solution form is
  given by
  \begin{equation*}
    \qf q(\phi)
    = \normsqr[\Lsqr I]{\SDir \phi}
    = \sum_{n=0}^\infty \normsqr[\Lsqr {I_n}] {\SDir_{I_n} \Phi_n}
    = \sum_{n=0}^\infty \iprod[\C^2] {Q_{I_n}(-1) \Phi_n} {\Phi_n}
  \end{equation*}
  and this is bounded from above and below by
  $(1+\rho_\pm)\sum_{n=0}^\infty \abssqr[\C^2]{\phi_n}
  \xi_\pm^{\ell_n}(-1)/\rho_n$ using~\eqref{eq:sol.form.2dim} of
  \Sec{2dim} and~\eqref{eq:jacobi.ass0}.  Since $\xi_\pm^\ell(-1)$ is
  of order $\ell$ as $\ell \to 0$ and $\xi_\pm^\ell(-1) \to 1/2$ as
  $\ell \to \infty$, we can bound $\qf q(\phi)$ from below and above
  by multiplying with $(\ell_n^-/\rho_n)_n$.

  \itemref{jacobi.iv}~In particular, we can bound $\qf q(\phi)$ from
  below by a positive constant times $(\inf_n
  \ell_n^-/\rho_n)\normsqr[\lsqr \N]\phi$.  But $\inf_n
  \ell_n^-/\rho_n>0$ by~\eqref{eq:jacobi.conseq}, hence the boundary
  pairs positive.

  \itemref{jacobi.vi}~Similarly, $\qf q(\phi)=\normsqr{\SDir\phi}$ is
  bounded from above by a constant times $(\sup_n
  \xi_+^{\ell_n}(-1)/\rho_n) \normsqr[\lsqr \N] \phi$.  Since
  $\xi_+^\ell(-1)$ has the same behaviour as $\xi_-^\ell(-1)$ for
  $\ell \to 0$ and $\ell \to \infty$, the solution form $\qf q$ is
  bounded iff $\sup_n \ell_n^-/\rho_n<\infty$.
\end{proof}
%----------------------------------------------------------------------

Let us now provide a list of examples:

%----------------------------------------------------------------------
\begin{example}[Unbounded, positive and elliptically regular boundary
  pair]
  \label{ex:bd2.ell.pos}
  Let $\ell_n$ and $\rho_n$ be of the same order ($0< \tau_- \le
  \rho_n/\ell_n \le \tau_+ < \infty$) and $\lim_n \ell_n =0$, then
  $\ell_n^-=\ell_n$ eventually and the boundary pair is unbounded and
  elliptically regular (and of course positive), see
  \PrpenumS{jacobi}{jacobi.v}{jacobi.vi}.  The Neumann operator in
  this case has purely discrete spectrum iff $\sum_n \ell_n < \infty$.

  This example shows that the spectral characterisation in
  \Thmenum{krein2}{krein2.v} can be actually used in a slightly wider
  class than ordinary boundary triples (see \Thm{bd3.intro}).
\end{example}
%----------------------------------------------------------------------

%----------------------------------------------------------------------
\begin{example}[Not elliptically regular, positive boundary pair]
  \label{ex:non-ell.bd2}
  If $(\ell_n)_n$ and $(\rho_n)_n$ are chosen such that we have
  $\sup_n \rho_n/\ell_n < \infty$, but $\inf_n \rho_n/\ell_n = 0$,
  then the boundary pair is not elliptic (in particular not bounded).
  For example, if $\rho_n=q^n$ ($0<q<1$) or $\rho_n=n^{-\gamma}$ and
  $\ell_n = n^{-\beta}$, $\gamma> \beta>0$, then the boundary pair is
  not elliptic.
\end{example}
%----------------------------------------------------------------------

For further examples, let us specify $\ell_n=n^{-\beta }$ and
$\rho_n=n^{-\gamma}$ with $\beta \le \gamma$ and $\gamma \ge 0$.  In
particular,~\eqref{eq:jacobi.ass0} is then fulfilled and
$(\Gamma,\HSaux)$ is a boundary pair by \Prpenum{jacobi}{jacobi.i}.
It is now a straightforward calculation using~\eqref{eq:jacobi.a.b} to
see that
\begin{equation*}
  -a_n(z)=\frac {n^\gamma\sqrt z}{\sin(n^{-\beta}\sqrt z)}(1+\err(1))
  \sim
  \begin{cases}
    n^\alpha,&\beta \ge 0, \quad z \in \C \setminus \spec \HDir,\\
    n^\gamma\e^{-n^{-\beta}}, & \beta<0, \quad z=-1,
  \end{cases}
  \qquad
  \alpha:=\beta+\gamma,
\end{equation*}
(where $a_n \sim b_n$ means that $a_n/b_n$ is bounded from above and
below by positive constants) and
\begin{equation*}
  q_n(z)=b_n(z)+a_n(z)+a_{n-1}(z)
  \sim
  \begin{cases}
    n^{\alpha-1}, & \beta\ge 1/2,\\ 
    & \text{or}\quad \beta>0,\\
    (-2z)n^{\alpha-2\beta} + \beta n^{\alpha-2\beta-1},&0<\beta < 1/2,\\
    n^{\alpha-\beta}=n^\gamma,& \beta \le 0,
  \end{cases}
  \begin{array}{l}
    z \in \C \setminus [0,\infty),\\
    z=0,\\
    z \in \C \setminus \spec \HDir,\\
    z<0.
  \end{array}
\end{equation*}
Moreover, the squared norm of $\HSaux^{1/2}$ is equivalent with
\begin{equation}
  \label{eq:sob.weight}
  \normsqr[\sob{\N,r,w}] \phi
  := \sum_{n \in \N}
    \bigl(\abssqr{(\de \phi)_n} r_n + \abssqr{\phi_n}w_n \bigr)
  \quadtext{with}
  r_n=-a_n(-1), \quad w_n=q_n(-1),  
\end{equation}
where $(\de \phi)_n=\phi_{n+1}-\phi_n$ is the discrete derivative.

%----------------------------------------------------------------------
\begin{example}[\DtN\ form unbounded from both sides]
  \label{ex:dtn.unbdd.below}
  For a non-elliptically regular boundary pair, the \DtN\ form can be
  unbounded from both sides: let $\ell_n = n^{-\beta}$ and
  $\rho_n=n^{-\gamma}$ with $\gamma> \beta>0$ and $0<\beta < 1/2$, and
  let $z \in (0,\infty)\setminus \spec \HDir$: Since the boundary pair
  is not bounded, $\qf l_\lambda$ is not bounded from above.  To see
  that $\qf l_z$ is not bounded from below, take
  $\phi^k=(\phi^k_n)_n$ with $\phi^k_n=1$ if $1 \le n \le k$ and $0$
  otherwise.  Then, using the asymptotics stated above, we obtain
  \begin{align*}
    \qf l_z(\phi^k)
    = \sum_{n \in \N} \bigl(-a_n(z)
      \abssqr{\phi^k_{n+1} - \phi^k_n}
      + q_n(z)\abssqr{\phi^k_n}
    \Bigr) \sim k^\alpha 
     - 2 z \sum_{n=1}^k n^{\alpha-2\beta}.
  \end{align*}
  Since $0<\beta<1/2$, the second sum, of order $k^{\alpha-2\beta+1}$,
  is dominant, and negative.  Moreover, $\normsqr[\HSaux] {\phi^k}=k$,
  and therefore $\qf l_\lambda(\phi^k)/\normsqr[\HSaux] {\phi^k} \sim
  -z k^{\alpha-2\beta}) \to -\infty$ as $k \to \infty$.  In
  particular, we have shown that $\qf l_\lambda$ is neither bounded
  from above nor from below (see also \Remenum{dtn-qf}{dtn-qf.ii}).
\end{example}
%----------------------------------------------------------------------

Let us now have another choice for $(\ell_n)_n$ and $(\rho_n)_n$.  In
particular, we want the corresponding Jacobi coefficients to have the
form $a_n=a_n(0)=-n^\alpha$ and $b_n=b_n(0) = -(a_n + a_{n-1})$.  Then
the corresponding Jacobi operator $J=\Lambda(0)$ is a \emph{pure}
(discrete) Laplacian, while for other values $\lambda \in \R \setminus
\spec \HDir$, $\Lambda(\lambda)$ is a discrete Schr\"odinger
operator with an additional potential of order $-\lambda
\ell_n/\rho_n$ (and this is of order $-\lambda n^{\alpha-2\beta}$, see
below), hence unbounded if $\alpha-2\beta=\gamma-\beta>0$.  Therefore,
we have another example where the form $\qf l_0$ is \emph{not} closed
on $\HSaux^{1/2}$ but only on a larger space $\HSaux^{1/2}_0
\supsetneq \HSaux^{1/2}$. % (see also \Ex{dtn.qf.0.not.closed}).

We use the ansatz $\ell_n = n^{-\beta} L_n^{-1}$ and $\rho_n =
n^{-\gamma} R_n^{-1}$ with $\alpha=\beta+\gamma>0$ and $\gamma>\beta$.
It can then be shown that for $a_n=n^\alpha$ and $b_n=-(a_n+a_{n-1})$
the sequences $(L_n)_n$ and $(R_n)_n$ defined above actually converge
to $1$ as $n \to \infty$.

This ansatz allows us to use known results on the spectrum of this
special Jacobi operator (see e.g.~\cite[Thm~1.1]{sahbani:08} and
references therein; as well as~\cite{janas-naboko:01} for the case
$\alpha=1$ and the general ideas of the spectral analysis).  The
spectrum of $J$ is purely discrete if $\alpha>2$, and absolutely
continuous if $0 < \alpha \le 2$.  If $\alpha < 2$ then $\spec J =
[0,\infty)$ and if $\alpha=2$ then $\spec J = [1/4,\infty)$.  In the
latter case ($\alpha=2$), the spectrum is \emph{purely} absolutely
continuous.

%----------------------------------------------------------------------
\begin{example}[Counterexamples to the compactness]
  \label{ex:dir.comp.neu.not}
  Here, we show that $\RNeu$ can be non-compact, while $\RDir$ is
  compact and $\Lambda(0)^{-1}$ can be either compact or non-compact.

  If we choose $\gamma \ge \beta=1$ then the Neumann operator $\HNeu$
  has purely absolutely continuous spectrum $[0,\infty)$ since $\sum_n
  \ell_n=\infty$, while the Dirichlet operator $\HDir$ has purely
  discrete spectrum.

  If $\gamma > \beta=1$, i.e., if $\alpha = \beta+\gamma > 2$ then the
  \DtN\ operator $\Lambda(0)=J$ has purely discrete spectrum.  By the
  monotonicity (\Thmenum{dn.z.qf}{dn.z.qf.vi}), $0 \le \qf l_0 \le \qf
  l=\qf l_{-1}$, and this inequality remains true for the closure of
  the form $\qf l_0$ (see~\cite[Sec.~4.4]{davies:95} for an order of
  quadratic forms).  In particular, the associated non-negative
  operators fulfil $0 \le \Lambda(0) \le \Lambda=\Lambda(-1)$, hence
  $\Lambda(0)^{-1} \ge \Lambda^{-1}\ge 0$, and $\Lambda^{-1}$ is also
  compact.  In this case, the boundary pair is not elliptically
  regular, and \Thmenum{krein2}{krein2.ia} is no longer true, as $\map
  \Gamma{\HS^1}\HSaux$ is compact by \Thm{nd.comp}, but $\essspec
  \HNeu=[0,\infty) \ne \emptyset = \essspec \HDir$.

  If $\gamma=\beta=1$, then $\alpha=2$ and the \DtN\ operator
  $\Lambda(0)$ has purely absolutely continuous spectrum
  $[1/4,\infty)$.  The boundary pair then is elliptically regular.
\end{example}
%----------------------------------------------------------------------

%----------------------------------------------------------------------
\begin{example}[Example violating the spectral
  relation~\Thmenum{krein2}{krein2.ii}]
  \label{ex:non-ell.bd2.spec}
  Choose $0 < \beta < 1/2$ and $\gamma=2-\beta>0$.  Then $\alpha=2$,
  and $\spec {\Lambda(0)}=[1/4,\infty)$, but the spectrum of the
  Neumann operator is $[0,\infty)$ (and purely absolutely continuous,
  see \Prpenum{jacobi}{jacobi.ia}); the Dirichlet spectrum is again
  discrete.  In particular, the implication ``$0 \in \spec \HNeu$
  $\Rightarrow$ $0 \in \spec {\Lambda(0)}$'' is not true (note that $0
  \notin \spec \HDir$).  Since $\beta < \gamma$, the boundary pair is
  not elliptic.  It can be seen as in \Ex{dtn.unbdd.below} that $\qf
  l_\lambda$ is even unbounded from below for all $\lambda>0$ (not in
  the Dirichlet spectrum).

  Let us illustrate what goes wrong in the proof of
  \Thmenum{krein2}{krein2.ii}: We argued by contraposition, so our
  assumption is $0 \notin \spec {\Lambda(0)}$ (which is true here).
  In order to show that $0 \notin \spec \HNeu$, we would have to show
  that $D(z):=\SDir(z)\wLambda(z)^{-1}\SDir(\conj z)^*$ is
  holomorphic in $z=0$ as function with values in $\BdOp \HSaux$.  But
  $D(0)$ is not bounded, as we need to use the weak version of
  $\Lambda(z)^{-1}$.
\end{example}
%----------------------------------------------------------------------

%----------------------------------------------------------------------
\begin{example}[Example violating the spectral
  relation~\Thmenum{krein2}{krein2.v}]
  \label{ex:non-ell.bd2.dis.spec}
  We can actually modify \Ex{non-ell.bd2.spec} such that the
  implication ``$0 \in \disspec {\Lambda(0)}$ $\Rightarrow$ $0 \in
  \disspec \HNeu$'' is false, although the boundary pair is positive
  (but \emph{not} elliptically regular): Take the direct sum
  $(\Gamma,\HSaux)$ (see \Sec{sum.bd2}) of the boundary pair of the
  previous example (denoted now $(\Gamma_1,\HSaux_1)$) and any
  boundary pair $(\Gamma_2,\HSaux_2)$ such that $0$ is a simple and
  isolated eigenvalue in $\spec{\Lambda_2(0)}$ and $\spec {\HNeu_2}$
  (e.g., the boundary pair on $[0,1]$ as in \Sec{2dim}).  Then $0$ is
  a discrete eigenvalue of $\spec {\Lambda(0)}=\spec{\Lambda_1(0)}
  \cup \spec {\Lambda_2(0)}$, but $0$ is not isolated in $\spec
  \HNeu=\spec {\HNeu_1} \cup \spec{\HNeu_2}=[0,\infty)$.
\end{example}
%----------------------------------------------------------------------

%----------------------------------------------------------------------
\subsection{Laplacian on Lipschitz domains}
\label{sec:lapl.mfd}
%----------------------------------------------------------------------

We consider now a compact $d$-dimensional Riemannian manifold $X$ with
its natural $d$-dimensional volume measure $\mu$.  Moreover, we assume
that $X$ has a Lipschitz boundary $Y=\bd X$ in the following sense:
let $\wt X$ be a complete smooth Riemannian manifold.  A
\emph{(smooth) manifold with Lipschitz boundary} or a \emph{Lipschitz
  domain} $X$ in $\wt X$ is the closure $X$ of an open subset of $\wt
X$ such that $\bd X$ can locally be written as graph of a Lipschitz
function (for details see~\cite[App.~A]{mitrea-taylor:99}).  It can be
shown that $\bd X$ has a natural measure $\nu$, the
$(d-1)$-dimensional Hausdorff measure.

%----------------------------------------------------------------------
\begin{remark}
  \label{rem:lipschitz.mfd}
  A manifold with Lipschitz boundary $X$ is locally $\Contspace
  [\infty]$-diffeomorphic with the set below the graph of a Lipschitz
  function, hence all results on Lipschitz domains in $\R^d$, which
  are invariant under such diffeomorphisms, remain true on a smooth
  manifold with Lipschitz boundary.  If we choose the half-space
  $\R^{d-1}\times \R$ as local model space (i.e., if $X$ has a local
  parametrisation into the half-space), the transition functions are
  only bi-Lipschitz.  Therefore, we can only carry over results for
  \emph{smooth} boundaries onto smooth manifolds with Lipschitz
  boundary which are invariant under bi-Lipschitz transformations.
  
  The boundary $\bd X$ is only a \emph{Lipschitz manifold}, i.e., the
  transition maps are only bi-Lipschitz functions, and no longer
  smooth.  A very nice introduction to Lipschitz manifolds (and even
  differential forms in this context) is given
  in~\cite[App.~A]{mmt:01} (we refer also to the discussion
  in~\cite[Sec.~1.2.1]{grisvard:85}).
\end{remark}
%----------------------------------------------------------------------
Denote by $\Ci X$ the space of functions, which are smooth on the
interior $\ring X := X \setminus \bd X$ such that all derivatives
extend continuously onto $X$.  We set $\HS:= \Lsqr X$ (with respect to
the volume measure $\mu$).  Moreover, $\HS^1:=\Sob X$ denotes the
completion of $\Ci X$ with respect to the norm given by $\normsqr[\Sob
X] u := \normsqr[\Lsqr X] u + \normsqr[\Lsqr X] {\de u}$, where $\de
u$ denotes the exterior derivative of $u$.  We consider the form $\qf
h$ given by $\qf h(u):=\normsqr{\de u}$, $u \in \HS^1$.

We set $\HSaux := \Lsqr Y$ (with respect to the $(d-1)$-dimensional
Hausdorff measure $\nu$ on $Y=\bd X$).  Moreover, for smooth functions
$u$ we set $\Gamma u := u \restr {\bd X}$.  For the definition of
fractional Sobolev spaces $\Sob[s]{\bd X}$ ($0 \le s \le 1$), see
e.g.~\cite[Sec.~1.3.3]{grisvard:85}
or~\cite[App.~A]{gesztesy-mitrea:09}).

Our main result here is the following:
%----------------------------------------------------------------------
\begin{theorem}
  \label{thm:lip.mfd.bd2}
  The boundary pair $(\Gamma, \HSaux)$ associated with the quadratic
  form $\qf d$ is unbounded with $\HSaux^{1/2}=\Sob[1/2]{\bd X}$,
  elliptically regular and not positive.  The Dirichlet and Neumann
  operators $\HDir$ and $\HNeu$ are the usual Dirichlet and Neumann
  Laplacians $\Delta^\Dir_X$ and $\Delta^\Neu_X$ on $X$, respectively
  (with the sign convention $\Delta^\Dir_X,\Delta^\Neu_X\ge 0$).
  Moreover, the \DtN\ operator $\Lambda(z)$ ($z \notin \spec
  {\Delta^\Dir}$) has the usual interpretation, i.e., $\psi=\Lambda(z)
  \phi$ iff $\psi$ is the normal derivative of the solution of the
  Dirichlet problem $(\Delta - z) h=0$ and $h \restr {\bd X} = \phi$
  (provided $\phi$ is smooth enough).  Finally, $\HNeu$, $\HDir$ and
  $\Lambda(z)$ all have compact resolvents.
\end{theorem}
%----------------------------------------------------------------------
\begin{proof}
  For the boundedness of $\map \Gamma{\Sob X}{\Lsqr {\bd X}}$, note
  that $\Sob X$ and $\Lsqr {\bd X}$ are both invariant under
  bi-Lipschitz transformations, hence the boundedness follows from the
  corresponding result for smooth boundaries.  Moreover, smooth
  functions with support away from $\bd X$ are in $\ker \Gamma =:
  \Sobn X$, and also dense in $\HS=\Lsqr X$, hence $\ker \Gamma$ is
  dense in $\HS$.  In addition, $\Gamma(\Ci X)$ is dense in $\Lsqr{\bd
    X}$.  In particular, $(\Gamma,\HSaux)$ is a boundary pair.  It is
  also well-known, that the range of the Sobolev trace map $\Gamma$ is
  $\Sob[1/2]{\bd X} \subsetneq \Lsqr {\bd X}$ (see
  e.g.~\cite[Thm.~1.5.1.3]{grisvard:85}), hence the boundary pair is
  unbounded.

  In order to show the elliptic regularity, we have to check that the
  solution operator $\map S {\Sob[1/2]{\bd X}} {\Sob X}$ extends to
  the corresponding $\Lsqrspace$-spaces, i.e., to
  \begin{equation}
    \label{eq:dsol.lsqr}
    \map \eSDir {\Lsqr{\bd X}} {\Lsqr X};
  \end{equation}
  for Lipschitz domains in Riemannian manifolds, this has been shown
  in~\cite[Thm.~4.1]{mitrea-taylor:03} (see also \cite[Prop.~3.7 and
  its proof]{mitrea-taylor:05}).
  
  If the boundary pair was positive, then $\map \SDir {\HSaux^{1/2}}
  {\LS^1}$ would extend to a topological isomorphism $\map \eSDir
  \HSaux {\LS^0}$ by \Rem{bd2.ell}~(\ref{ell.bd2.iv}') and
  \Thmenum{pos.bd2}{pos.bd2.iii}, where $\LS^0$ is the closure of
  $\LS^1$ in $\HS$.  It can be seen that $\LS^0=\ker(\Hmax+1)$, where
  $\Hmax=\Deltamax$ is the Laplacian in the distributional sense ($u
  \in \dom \Deltamax$ iff $u$ and $\Deltamax u$ are in $\Lsqr X$).  In
  particular, any $h=\eSDir \phi \in \ker (\Hmax+1)$ would have a
  boundary value $\phi \in \Lsqr Y$ which is known not to be true.

  The compactness of the resolvents is a standard fact for
  (pseudo-)differential operators on compact manifolds.
\end{proof}
%----------------------------------------------------------------------

We call $(\Gamma,\HSaux)$ the \emph{boundary pair associated with the
  manifold $X$ and boundary $\bd X$}.

%----------------------------------------------------------------------
\begin{remark}
  \label{rem:why.bd2.simpler}
  Note that in the context of non-smooth domains, questions of
  regularity for the operator are rather delicate.  For more details
  on Sobolev spaces and elliptic boundary value problems on Lipschitz
  domains we refer e.g.\ to~\cite{jerison-kenig:95,mitrea-taylor:99,
    mmt:01, mitrea-taylor:03, mitrea-taylor:05,
    gesztesy-mitrea:09,gesztesy-mitrea:11,behrndt-micheler:14} and
  references therein.  Our approach only needs the first order spaces,
  as we only have to check that the solution operator extends to an
  operator on the corresponding $\Lsqrspace$-spaces, and we believe
  that this is generally simpler to check.  In this context, the
  solution operator (at $z=0$) is also called \emph{Poisson operator}.
\end{remark}
%----------------------------------------------------------------------

%----------------------------------------------------------------------
\begin{remark}
  \label{rem:why.ell.reg}
  The notion ``elliptically regular'' for a boundary pair has its
  motivation from this manifold example: The boundary triple
  $(\Gamma,\Gamma',\HSaux)$ associated with the boundary pair
  $(\Gamma,\HSaux)$ (see \Sec{rel.bd3}) is called \emph{elliptically
    regular} if $\dom \HDir \subset \WS$ and $\dom \HNeu \subset \WS$
  (in \cite[Def.~3.4.21]{post:12} we actually used additional
  assumptions about the range of the boundary maps $\Gamma$ and
  $\Gamma'$); and a (maximal) boundary triple is elliptically regular
  iff the corresponding boundary pair is.  Here, $\WS$ is a space on
  which $\Gamma'$ is defined and bounded ($\map {\Gamma'}{\WS}
  \HSaux$) and on which Green's identity~\eqref{eq:green} holds.  If
  we assume (for simplicity) that $\bd X$ is smooth then we can choose
  $\WS = \Sob[2] X$.  The condition $\dom \HDir \subset \WS$ is then
  equivalent to an ``elliptic regularity estimate''
  \begin{equation}
    \label{eq:ell.reg.mfd}
    \norm[{\Sob[2]X}] u
    \le C \bigl( \norm[\Lsqr X] {u}
    + \norm[\Lsqr X] {\laplacianD X u}
    \bigr)
  \end{equation}
  for all $u \in \dom \laplacianD X\cap \Sob[2] X$ and similarly for
  the Neumann operator $\HNeu=\laplacianN X$.
  
  Moreover, we have indicated in \Rem{why.ell.reg.intro} that a
  boundary pair is elliptically regular iff a ``normal derivative''
  $\Gamma'$ can be defined such that Green's identity~\eqref{eq:green}
  holds and such that $\Gamma'u \in \HSaux$ for all $u \in \dom
  \HDir$.  We will treat such questions, namely boundary triples
  associated with quadratic forms and the relation with boundary pairs
  in a forthcoming publication~\cite{post:pre14a} (see
  also~\cite[Sec.~3.4]{post:12}).
\end{remark}
%----------------------------------------------------------------------

Krein's resolvent formula now is valid for the Neumann and Dirichlet
Laplacian, i.e.,
\begin{equation}
  \label{eq:krein.lip}
  (\Delta^\Neu_X-z)^{-1} - (\Delta^\Dir_X - z)^{-1}
  = \eSDir(z) \Lambda(z)^{-1} \eSDir(\conj z)^*,
\end{equation}
Moreover, the (extension of the) solution operator $\map
{\eSDir(z)}{\Lsqr {\bd X}} {\Lsqr X}$ is usually called
\emph{Poisson} operator in this context.  In addition, we have the
characterisation of the spectrum
\begin{equation}
  \label{eq:spec.rel.mfd}
  \lambda \in \spec {\Delta^\Neu} \quad\Leftrightarrow\quad
  0 \in \spec {\Lambda(\lambda)}
\end{equation}
provided $\lambda \notin \spec {\Delta^\Dir}$.  Since the spectrum of
$\Delta^\Neu$ is purely discrete, and since $\map \Gamma {\Sob
  X}{\Lsqr {\bd X}}$ is a compact operator, the spectra of
$\Delta^\Dir$ and $\Lambda(\lambda)$ are purely discrete, too (see
\Prpenum{bd2.robin}{nd.z.v} and \Thm{nd.comp}).  Moreover, the
multiplicities of the eigenvalues are preserved (\Thm{krein1}).

%----------------------------------------------------------------------
\begin{remark}
  Most of our results extend to the case when $X$ is non-compact but
  $\bd X$ is compact, e.g.\ products (see \Sec{cyl.mfd} and also
  \Rem{coupl.mfd}) or warped products $X = [0,\infty) \times Y$ with
  metric $g=\dd s^2 + r(s)^2 h$, where $(Y,h)$ is a compact Riemannian
  manifold.  The only problem here is that the essential spectra of
  the Dirichlet and Neumann operator are the same
  (\Thmenum{krein2}{krein2.ia}), and in many cases just the entire
  half-axis $[0,\infty)$.  Nevertheless, the \DtN\ operator might be
  extended analytically into $[0,\infty)$; we come back to this
  situation in a forthcoming publication.
\end{remark}
%----------------------------------------------------------------------

Let us illustrate how coupling of boundary pairs can be used in the
manifold case
%----------------------------------------------------------------------
\begin{remark}
  \label{rem:coupl.mfd}
  A prominent example of a coupled boundary pair (see \Sec{coupl.bd2})
  we have in mind is a smooth manifold $X=X_1 \cup X_2$ without
  boundary such that $Y=X_1 \cap X_2$ is a smooth submanifold of
  co-dimension $1$, $X_1$ is a compact manifold with boundary $Y$ and
  $X_2=I \times_r Y$ is a warped product over an interval $I$, i.e., a
  manifold with metric $g=\dd s^2 + r(s)^2 h$ ($\map r I
  {(0,\infty)}$, $h$ a metric on $Y$).  For a warped product, we have
  explicit formulae for the solution and the \DtN\ operators (in terms
  of solutions of some ODEs related with $r$).  As boundary pairs we
  now choose $(\Gamma_i,\HSaux)$ associated with the quadratic forms
  $\qf h_i(u)=\normsqr[X_i]{\de u}$, $u \in \HS_i^1 = \Sob {X_i}$,
  where $\HSaux = \Lsqr Y$ and $\Gamma_i u = u \restr Y$.  The coupled
  form and operator (i.e., the Neumann operator) is now the form and
  Laplacian on the \emph{entire} manifold $X$.  Moreover, for the
  boundary pairs $(\Gamma_1,\HSaux)$ on the compact part of the
  manifold one can derive explicit formulae for the Dirichlet solution
  operators and \DtN\ maps, as well as for the (possibly non-compact)
  warped product.  Hence we have rather explicit formulae for the
  resolvent of the entire Laplacian on $X$ in terms of simpler
  building blocks.  We will come back to these ideas, treating also
  more complicated coupled structures, in a forthcoming publication.
\end{remark}
%----------------------------------------------------------------------

%----------------------------------------------------------------------
\subsubsection*{Regularisation of the manifold boundary pair}
%----------------------------------------------------------------------

%----------------------------------------------------------------------
\begin{example}[A bounded, but not positive boundary pair]
  \label{ex:bdd.non-pos}
  Let $(\wt \Gamma,\wt \HSaux)$ be the \emph{bounded modification} or
  \emph{regularisation} of the above boundary pair $(\Gamma,\HSaux)$
  associated with $X$ (see \Sec{bdd.bd2}): it follows from
  \Prp{bd2.bdd} that $(\wt \Gamma, \wt \HSaux)$ is not positive,
  although bounded.
\end{example}
%----------------------------------------------------------------------

%----------------------------------------------------------------------
\begin{example}[Non-compact \DtN, but compact Dirichlet and Neumann
  operator]
  \label{ex:neu.comp.dtn.not}
  If we assume (with the notation of the previous example) that the
  manifold $X$ is compact, then $\wtRNeu=\RNeu$ and $\wtRDir=\RDir$ are
  compact.  But since $\wt \Lambda=\id_{\wt \HSaux}$ and since $\wt
  \HSaux=\Sob[1/2]{\bd X}$ is infinite-dimensional, $\wt \Lambda^{-1}$
  is not compact.
\end{example}
%----------------------------------------------------------------------

%----------------------------------------------------------------------
\subsection{Laplacian on a non-compact cylindrical manifold}
\label{sec:cyl.mfd}
%----------------------------------------------------------------------

Let us consider here a simple example in which the space $X$ is a
product manifold $X=[0,\infty) \times Y$ with corresponding product
metric $g=\dd s^2 + h$, where $(Y,h)$ is a compact Riemannian manifold
without boundary.  Similar cases were considered e.g.\
in~\cite[Sec.~2.5]{gorbachuk-gorbachuk:91} or \cite[Ex.~6.8]{dhms:06}.
 
In this case, we have again $\HS=\Lsqr X$, $\HS^1=\Sob X$, $\qf
h(u)=\normsqr{\de u}$ and $\HSaux= \Lsqr{Y,h}$.  Identifying a
function $\map u X \C$ with the corresponding vector-valued function
$s \mapsto u(s)$ on $[0,\infty)$, we set $\Gamma u = u(0)$, $u \in
\HS^1$.  It can be seen similarly as before that $\Gamma$ is bounded
and that $(\Gamma,\HSaux)$ is an unbounded, elliptically regular, but
not positive boundary pair.

This example can be seen as a vector-valued version of the interval
case in \Sec{2dim} (except that $I=[0,\infty)$ is non-compact here and
has only one boundary point).  Namely, we can write
\begin{equation*}
  \qf h(u)
  =\int_I
  \bigl(
    \normsqr[\Lsqr {Y,h}]{u'(s)} 
    + \normsqr[\Lsqr {Y,h}] {\de_Y {u(s)}}
  \bigr)
  \dd s,
\end{equation*}
where $\de_Y \phi$ denotes the exterior derivative on $Y$.  Moreover,
all objects can be calculated rather explicitly using separation of
variables (denoting the eigenvalues and eigenfunctions of the
Laplacian on $Y$ by $\kappa_k \ge 0$ and $\Phi_k$, respectively).  For
example, we have
\begin{equation*}
  \SDir(z) \phi = \sum_k f_{z,k} \otimes \Phi_k,
\end{equation*}
where $f_{k,z}(s)=\exp(\im s \sqrt{z - \kappa_k})$ (the square root is
cut along the positive real line).  Moreover,
\begin{equation*}
  \Lambda(z) \phi
  = \sum_k \iprod[\Lsqr {Y,h}] \phi {\Phi_k} f_{z,k} \otimes \Phi_k
  = -\im \bigl(\sqrt{z - \laplacian Y}\bigr) \phi
\end{equation*}
(see~\cite[Sec.~3.5]{post:12} for details, e.g., the type of
convergence of the sums).  In particular, for $z=-1$ we have
$\Lambda=\sqrt{\laplacian Y + 1}$.

Similarly, we can treat more general cases like \emph{warped} products
(i.e., $X=I \times Y$ with metric $g=\dd s^2 + r(s)^2 h$ for some
function $\map r I {(0,\infty)}$).  We will come back to this point in
a forthcoming publication.

%----------------------------------------------------------------------
\begin{example}[$\RNeu$, $\RDir$ non-compact, $\Lambda^{-1}$ compact]
  \label{ex:cyl.comp}
  The compactness of $\Lambda^{-1}$ does not imply the compactness of
  $\RNeu$: In the example above, the Neumann and Dirichlet operators
  are the Laplacians on the non-compact cylinder $X=\R_+ \times Y$
  with Neumann resp.\ Dirichlet conditions at $\bd X=\{0\}\times Y$,
  hence their resolvents are not compact.  On the other hand,
  $\Lambda^{-1}=(\laplacian Y+1)^{-1/2}$ is compact.
\end{example}
%----------------------------------------------------------------------

%----------------------------------------------------------------------
\begin{example}[$\RNeu$, $\RDir$ non-compact, $\Lambda^{-1}$ non-compact]
  \label{ex:cyl.comp.mod}
  If we take the bounded modification of the above example, we obtain
  the case where none of the operators $\RNeu$, $\RDir$ and
  $\Lambda^{-1}(=\id_{\wt \HSaux})$ is compact.
\end{example}
%----------------------------------------------------------------------

%----------------------------------------------------------------------
\subsection{\DtN\ operator supported on a metric graph: leaky graphs
  and photonic crystals}
\label{sec:dtn.mg}
%----------------------------------------------------------------------
Let us consider here a case of a \DtN\ operator defined on a singular
space $Y$, where $Y$ is a metric graph embedded in a $2$-dimensional
Riemannian manifold $X$, i.e., $Y = \bigcup_{e \in E} Y_e$, and
each $Y_e$ is a closed one-dimensional (smooth) submanifold in $X$,
called \emph{edge segment}.  We assume for simplicity that $X$ is
compact, but under suitable uniformity assumptions the results below
remain true; e.g.\ if $X$ is a covering manifold with compact
quotient.

We call the closure of each connected component of $X \setminus Y$ a
\emph{face} of $Y$ in $X$, and label the faces by $(X_f)_{f \in F}$.
We assume that each face is compact in $X$, and that the boundary of
each face, consisting of the adjacent edges $E_f$, is Lipschitz (if
$X$ is non-compact, one needs e.g.\ that the Lipschitz constants are
globally bounded).  Let $(\Gamma_f,\HSaux_f)$ be the boundary pair
associated with the manifold $X_f$ and boundary $\bd X_f$.  Note that
each function $\phi_f \in \HSaux_f:=\Lsqr{\bd X_f}$ decomposes into its
components $\phi_f = (\phi_{e,f})_{e \in E_f}$ of the adjacent edge
segments $Y_e$, i.e., $\phi_{e,f} \in \Lsqr{Y_e}$.

A global boundary map is now defined on
\begin{equation}
  \label{eq:h1.coupl}
  \HS^1
  := \bigset{u \in \bigoplus_{f \in F} \Sob{X_f}}
            {(\Gamma u)_{e,f_1}=(\Gamma u)_{e,f_2}
            \text{ whenever $X_{f_1} \cap X_{f_2} = Y_e$}}
\end{equation}
by $\Gamma u := u \restr Y$ ($u \in \HS^1$).  This map is well-defined
since the boundary values of $u$ from different sides on an edge agree
by definition.  It is not difficult to see that $\HS^1 = \Sob X$, and
that $\map \Gamma {\HS^1} {\HSaux := \Lsqr Y}$ is bounded, since
$\Gamma$ is the restriction of the direct sum of the boundary maps
$\map{\Gamma_f}{\Sob{X_f}} {\Lsqr{\bd X_f}}$ to $\HS^1$ after suitable
identifications.

As quadratic form, we consider $\qf h(u) := \normsqr{\de u}$, $u \in
\HS^1=\Sob X$.
%----------------------------------------------------------------------
\begin{proposition}
  \label{prp:dtn.mg.bd2}
  The boundary pair $(\Gamma, \HSaux)$ is unbounded, elliptically
  regular and not positive.  The Neumann operator $\HNeu$ is the
  Laplacian on $X$, and the Dirichlet Laplacian is given by the direct
  sum of the Dirichlet Laplacians on $X_f$, i.e.,
  \begin{equation*}
    \HDir = \bigoplus_{f \in F} \Delta_{X_f}^\Dir
  \end{equation*}
  and $\HDir$ is in particular decoupled.  The \DtN\ operator
  $\Lambda(z)$ acts as follows: if $\phi$ is a (suitably smooth)
  function on $Y$, then $\psi=\Lambda(z) \phi$ is given on $Y_e$ as
  the sum of the normal derivatives of the solutions of the Dirichlet
  problem on the two adjacent faces of $e$ (i.e., $\psi_e$ is the jump
  in the derivative when crossing $Y_e$ form one face to the other).
\end{proposition}
%----------------------------------------------------------------------
\begin{proof}
  We omit the details here, since we will consider these questions in
  a forthcoming publication.  We only indicate how to prove the
  elliptic regularity of the boundary pair: this can be seen by noting
  first that the solution form is given by
  \begin{equation*}
    \qf q(\phi) = \sum_{f \in F} \qf q_f(\phi_f),
  \end{equation*}
  where $\qf q_f$ is the solution form of the boundary pair
  $(\Gamma_f,\HSaux_f)$ associated with the face $X_f$.  As these
  boundary pairs are all elliptically regular (the face $f$ is assumed
  to have a Lipschitz boundary), each of the solution forms is
  bounded by $C_f^2$ in $\HSaux_f=\Lsqr{\bd X_f}$, the entire solution
  form is then bounded by $\max_{f \in F} C_f^2$, and hence
  $(\Gamma,\HSaux)$ is elliptically regular.
\end{proof}
%----------------------------------------------------------------------

%----------------------------------------------------------------------
\begin{remark}[Leaky graphs and photonic crystals]
  \label{rem:leaky.graphs}
  Let $\qf h_a$ be the Robin-type perturbation of the form $\qf h$
  (i.e., $\qf h_a(u):= \qf h(u) + a \normsqr{\Gamma u}$ for $a \ge 0$,
  see \Sec{robin}).  Then $\qf h_a$ is non-negative and closed, and we
  can consider $(\Gamma,\HSaux)$ associated with the form $\qf h_a$.
  The associated (Neumann) operator $H_a$ then has a Robin-type
  boundary condition of the type $\Gamma' u + a \Gamma u=0$, where
  $\Gamma' u$ on $Y_e$ is the sum of the normal (outwards) derivative
  of $u$ on the two adjacent faces.

  The resolvent difference of $H_a$ and $\HDir$ (the latter is still
  decoupled) can be expressed by a Krein-type formula, and $H_a$
  converges to $\HDir$ in norm resolvent sense as $a \to \infty$.
  Moreover, the associated \DtN\ operator of $(\Gamma,\HSaux)$
  associated with $\qf h_a$ is $\Lambda_a(z)=\Lambda(z)+a$.

  This situation is closely related to a model called \emph{leaky
    graph} (see the overview article~\cite{exner:08}).  Note that in
  the situation of a leaky graph, one has $X=\R^2$ and $a < 0$ (this
  needs some modifications of our arguments).  Moreover, some faces
  may be non-compact with finitely many adjacent edges, some of them
  having infinite length.  One is interested e.g.\ in the asymptotic
  behaviour of the negative eigenvalues as $a \to -\infty$.  If the
  graph is just a curve embedded in $\R^2$, then the asymptotics are
  typically of the form $\lambda_k(a)=-a^2/4 + \mu_k + \Err(\abs
  a^{-1} \ln \abs a)$ where $\mu_k$ is the $k$-th eigenvalue of a
  Schr\"odinger operator on the curve with (negative) potential given
  by the curvature of the curve in the plane.  \emph{Only very few
    results are known} if the curve is replaced by a graph.  In
  particular, it would be very interesting to relate $\mu_k$ with the
  above defined \DtN\ operator (or any other related one) on the graph
  (see~\cite[Sec.~7.13]{exner:08}).

  In~\cite{kuchment-kunyansky:02} (see also the references therein),
  Kuchment and Kunyansky consider the above-mentioned operator $H_a$
  appearing as the limit operator in the analysis of \emph{photonic
    crystals} supported on a \emph{periodic} (hexagonal) lattice
  $\Gamma$ embedded in $X=\R^2$.  As in our approach, they reduce the
  eigenvalue problem for $H_a$ to an eigenvalue problem for
  $\Lambda(z)$.  Then they investigate the nature of the operator
  $\Lambda(z)$ on $Y$, and try find a good candidate of a
  \emph{differential} operator on $Y$ being close to the
  \emph{pseudo-differential} operator $\Lambda(z)$.  This problem is
  still not yet completely understood.  \emph{We believe that our
    method helps to analyse these problems further.}
\end{remark}
%----------------------------------------------------------------------

%----------------------------------------------------------------------
\subsection{Laplacian with mixed boundary conditions: the Zaremba
  problem}
\label{sec:zaremba}
%----------------------------------------------------------------------

%----------------------------------------------------------------------
\subsubsection*{An elliptically regular Zaremba problem}
%----------------------------------------------------------------------
Let $X$ be a compact Riemannian manifold with smooth boundary $\bd
X$.  Let $Y \subset \bd X$ be a compact submanifold of the same
dimension as $\bd X$, with \emph{smooth} boundary in $\bd X$, and let
$Z := \clo{\bd X \setminus Y}$.  We call the Laplacian on $X$ with
Dirichlet condition on $Z$ and Neumann condition on $Y$ the
\emph{Zaremba Laplacian}, denoted by $\Delta_X^Z$ (in particular,
$\Delta^\Dir = \Delta_X^{\bd X}$ and $\Delta^\Neu =
\Delta_X^\emptyset$).

Let us first compare the Zaremba Laplacian with the Dirichlet Laplacian
on $X$.  We will see that we can again treat this problem with our
boundary pair method. 

Set $\HS:=\Lsqr X$ and set
\begin{equation*}
  \HS^1 := \Sobx Z X 
  := \bigset {u \in \Sob X} {u \restr Z = 0},
  \qquad
  \qf h(u) := \normsqr{\de u}.
\end{equation*}
As boundary operator we choose $\Gamma u := u \restr Y$.

%----------------------------------------------------------------------
\begin{theorem}
  \label{thm:zaremba.bd2}
  The boundary pair $(\Gamma, \HSaux)$ associated with the form $\qf
  h$ is unbounded, elliptically regular and not positive.  The Neumann
  operator $\HNeu$ is the Zaremba Laplacian $\Delta_X^Z$ on $X$ with
  Dirichlet condition on $Z$ and Neumann condition on $Y$.  Moreover,
  the Dirichlet operator is the Laplacian $\Delta_X^\Dir=\Delta_X^{\bd
    X}$ with (pure) Dirichlet condition on $\bd X$.

  The range of the boundary map is $\HSaux^{1/2}=\set{\phi \in \Lsqr
    Y}{\wt \phi \in \Sob[1/2]{\bd X}}$, where $\wt \phi=\phi \oplus 0$
  is the extension of $\phi \in \Lsqr Y$ by $0$ on $Z$.  The Dirichlet
  solution operator is given by $\SDir(z) \phi = \SDir_{(X,\bd X)}(z)
  \wt \phi$, where $\SDir_{(X,\bd X)}(z)$ is the Dirichlet solution
  (Poisson) operator for the boundary pair associated with $X$ and the
  \emph{entire} boundary $\bd X$.

  Finally, the Zaremba Laplacian and the \DtN\ operator $\Lambda(z)$
  ($z \notin \spec {\Delta_X^\Dir}$) have discrete spectrum.
\end{theorem}
%----------------------------------------------------------------------
\begin{proof}
  Clearly, $\map \Gamma \HS^1 {\HSaux:=\Lsqr Y}$ is bounded since $u
  \mapsto u \restr {\bd X}$ is bounded, as well as the restriction map
  $\Lsqr {\bd Y} \to \Lsqr Y$.  Moreover, that $(\Gamma,\HSaux)$ is an
  unbounded boundary pair follows from the existence of $\wt \phi \in
  \Sob[1/2]{\bd X} \setminus \Lsqr{\bd X}$ such that $\wt \phi \restr
  Z=0$.
  
  The assertion on $\HSaux^{1/2}$ and the Dirichlet solution operator
  is easily seen by noting that for $\phi \in \HSaux^{1/2}$ there
  exists $u \in \Sobx Z X$ such that $\phi = u \restr Y$.  In
  particular, the extension by $0$ is just $\wt \phi = u \restr{\bd
    X}$, and hence in $\Sob[1/2]{\bd X}$.  The elliptic regularity
  follows from
  \begin{equation*}
    \qf q(\phi)
    = \normsqr[\Lsqr X]{\SDir \phi}
    = \normsqr[\Lsqr X]{\SDir_{(X,\bd X)} \wt \phi}
    \le C^2 \normsqr[\Lsqr{\bd X}]{\wt \phi}
    = C^2 \normsqr[\Lsqr Y]{\wt \phi}
  \end{equation*}
  for all $\phi \in \HSaux^{1/2}$, where we used the elliptic
  regularity of the boundary pair associated with the manifold $X$ and
  boundary $\bd X$ (see \Sec{lapl.mfd}).  The non-positivity can be
  seen as in \Sec{lapl.mfd}.

  As $\Sobn X \subset \Sobx Z X \subset \Sob X$ we have
  $\HDir=\Delta_X^\Dir \ge \HNeu=\Delta_X^Z \ge \Delta_X^\Neu$, where
  the latter is the Laplacian on $X$ with Neumann boundary conditions
  on the entire boundary $\bd X$.  The latter has compact resolvent,
  so the same is true for the Zaremba and Dirichlet Laplacian $\HNeu$
  and $\HDir$, respectively.  Moreover, $\Gamma$ is a compact
  operator, since $\Gamma u = \wt \Gamma u \restr Y$, and $\map{\wt
    \Gamma}{\Sob X} {\Lsqr {\bd X}}$ is compact.  In particular,
  $\Lambda(z)$ has discrete spectrum by \Thm{nd.comp}.
\end{proof}
%----------------------------------------------------------------------

Krein's resolvent formula here relates the resolvent of the Zaremba
Laplacian with the pure Dirichlet Laplacian
\begin{equation}
  \label{eq:krein.zaremba}
  (\Delta_X^Z-z)^{-1} - (\Delta_X^\Dir - z)^{-1}
  = \eSDir(z) \Lambda(z)^{-1} \eSDir(\conj z)^*.
\end{equation}
Since the boundary pair is elliptic, the operators on the RHS all act
in the Hilbert spaces $\HSaux=\Lsqr Y$ and $\HS=\Lsqr X$.

%----------------------------------------------------------------------
\begin{remark}
  \label{rem:zaremba.reg}
  The domain of the Zaremba Laplacian is contained in $\Sob[3/2-\eps]
  X$ for all $\eps>0$, but not contained in $\Sob[3/2] X$ itself.  The
  latter can be seen in the following situation where $X=[0,\infty)
  \times \R$ (or some bounded subset containing $0$) and $u(x,y):= \Im
  \sqrt{x + \im y}$ (see~\cite{shamir:68}); $u$ fulfils a Dirichlet
  condition on the positive $x$-axis $Z$ and a Neumann condition on
  the negative $x$-axis $Y$, and $\Delta u = 0$.  Moreover, $\normder
  u \restr Z$ is not in $\Lsqr Y$, and hence $u \notin \Sob[3/2] X$.
  At first sight surprisingly, the Zaremba problem is less regular for
  smooth boundaries than for certain boundaries with corners (see the
  discussion in \cite[Sec.~4.3]{grubb:11} and
  also~\cite{mitrea-mitrea:07}).
\end{remark}
%----------------------------------------------------------------------

%----------------------------------------------------------------------
\subsubsection*{A non-elliptically regular Zaremba problem}
%----------------------------------------------------------------------
If we use the pure Neumann Laplacian as reference operator (by
choosing $\HS^1 := \Sob X$) and again, $\Gamma u = u \restr Y$, then
the Dirichlet operator $\HDir$ is the Zaremba Laplacian $\Delta_X^Y$,
now with Dirichlet condition on $Y$ and Neumann condition on $Z$, and
the Neumann operator $\HNeu$ is the pure Neumann Laplacian
$\Delta_X^\Neu$.
%----------------------------------------------------------------------
\begin{theorem}
  \label{thm:zaremba2.bd2}
  The boundary pair $(\Gamma, \HSaux)$ associated with the form $\qf
  h$ corresponding to the pure Neumann Laplacian on $X$ is unbounded,
  and \emph{not} elliptically regular.  The Dirichlet operator $\HDir$
  is the Zaremba Laplacian $\Delta_X^Y$, now with Dirichlet condition
  on $Y$ and Neumann condition on $Z$.  The boundary map range space
  $\HSaux^{1/2}$ is $\Sob[1/2] Y := \set{\psi \restr Y}{\psi \in
    \Sob[1/2] {\bd X}}$.  Finally, the Zaremba Laplacian and the \DtN\
  operator $\Lambda(z)$ ($z \notin \spec{\Delta_X^Y}$) have discrete
  spectrum.
\end{theorem}
%----------------------------------------------------------------------
\begin{proof}
  That the boundary pair is not elliptically regular can be seen as
  follows: As in \Rem{zaremba.reg}, where $X=[0,\infty) \times \R$,
  one can find functions $u \in \dom \HDir$ (the Zaremba domain) such
  that $\check \Gamma' u$ is not contained in $\HSaux$, and hence by
  \Remenum{bd2.ell}{bd3-ell} the boundary pair is not elliptic.

  The Neumann operator has discrete spectrum since $X$ is compact with
  smooth boundary.  The compactness of $\Lambda^{-1}$ can be seen as
  before.
\end{proof}
%----------------------------------------------------------------------
Krein's resolvent formula in this case still holds, but only in its
``weak'' form
\begin{equation}
  \label{eq:krein.zaremba2}
   (\Delta_X^\Neu - z)^{-1} - (\Delta_X^Y-z)^{-1}
  = \SDir(z) \wLambda(z)^{-1} \SDir(\conj z)^*,
\end{equation}
since now, the operators on the RHS map as $\HS \to \HSaux^{-1/2} \to
\HSaux^{1/2} \to \HS$.  Similar formulae have also been shown
in~\cite{grubb:11} and \cite{pankrashkin:06b}.  Moreover, since
$\RNeu$ and $\Lambda^{-1}$ are compact, we have the spectral relation
\begin{equation*}
  \lambda \in \spec {\laplacianN X}
    \qquad \Leftrightarrow \qquad
  0 \in \spec {\Lambda(\lambda)},
\end{equation*}
if $\lambda \notin \spec {\Delta_X^Y}$.  In other words, if the
spectrum and eigenfunctions of $\laplacianN X$ are known, we know the
$0$-eigenspace of the \DtN\ operator (cf.\
\Thmenum{krein1}{krein1.ia}).

%----------------------------------------------------------------------
\begin{remark}
  \label{rem:why.bd2}
  We would like to stress here that the boundary pair of
  \Thm{zaremba.bd2} can also be treated with the methods of
  quasi-boundary triples, according to
  \Thmenum{bd3.intro}{bd3.intro.i}.  On the other hand, the boundary
  pair of \Thm{zaremba2.bd2} does not correspond to a quasi boundary
  triple, and can hence be treated only by our boundary pair concept.
\end{remark}
%----------------------------------------------------------------------

%----------------------------------------------------------------------
\subsection{Example of a generalised boundary pair: discrete
  Laplacians}
\label{sec:large.bd2}
% ----------------------------------------------------------------------
Let us present here another class of examples; in this case, the
boundary pair is a generalised one, i.e., $\ker \Gamma$ is no longer
dense in $\HS$.

Let $(V,E,\bd)$ be a discrete graph, i.e., $V$ denotes the set of
vertices, $E$ the set of edges and $\map \bd E {V \times V}$ maps $e$
onto $(\bd_-e,\bd_+e)$, the \emph{initial} and \emph{terminal vertex}
of $e$; fixing therefore also an orientation.  Denote by $E_v$ the set
of edges $e$ adjacent with the vertex $v \in V$ (i.e., $e \in E_v$ iff
$v=\bd_+e$ or $v=\bd_-e$.  If $e \in E_v$, we denote by $v_e$ the
vertex on the other end of $e$. 

We assume for simplicity here that the graph is \emph{finite}.  Let
$\map \mu V {(0,\infty)}$ and $\map \rho E {(0,\infty)}$ be functions,
the \emph{vertex} and \emph{edge weights}.  Let
\begin{equation*}
  \HS := \lsqr{V,\mu}, \qquad
  \normsqr[\lsqr{V,\mu}] f 
    := \sum_{v \in V} \abssqr{f(v)} \mu(v),
\end{equation*}
and set
\begin{equation*}
  \qf h(f)
  := \sum_{e \in E} \abssqr{f(\bd_+e)-f(\bd_-e)}\rho(e)
\end{equation*}
with $\dom \qf h = \HS^1 = \HS$.  Since this form is bounded, we can
omit the subscripts $(\cdot)^1$ indicating the form domain.  The
Neumann operator $\HNeu$, i.e., the operator associated with $\qf h$
acts as
\begin{equation*}
  (H f)(v)
  = \frac 1{\mu(v)} \sum_{e \in E_v} \rho(e)\bigl(f(v)-f(v_e)\bigr).
\end{equation*}
If we choose $\mu(v)=1$ and $\rho(e)=1$ then we arrive at the
\emph{combinatorial} Laplacian; if we choose $\mu(v)=\deg v = \card
{E_v}$ and $\rho(e)=1$, then we arrive at the \emph{normalised}
Laplacian.

We now declare a subset of $V$ as \emph{boundary} of the graph, i.e.,
let $\bd V \subset V$ be the set of \emph{boundary vertices}.  The
vertices in its complement, $\ring V := V \setminus \bd V$, are called
\emph{inner vertices}.  We set
\begin{equation*}
  \HSaux := \lsqr{\bd V,\mu}, \qquad
  \Gamma f := f \restr {\bd V}.
\end{equation*}
Note that $\HS^\Dir := \ker \Gamma = \lsqr {\ring V,\mu}$ is
\emph{not} dense in $\HS=\lsqr {V,\mu}$.  Therefore, $(\Gamma,\HSaux)$
is a generalised bounded boundary pair associated with $\qf h$.  The
Dirichlet operator acts formally as $H$, but only on $\lsqr{\ring
  V,\mu}$; if $\embmap \iota {\lsqr{\ring V,\mu}}{\lsqr{V,\mu}}$
denotes the natural embedding, and $\pi:=\iota^*$ the corresponding
projection, then $\HDir = \pi \HNeu \iota$.  Note that this example
corresponds to $X=V$, $Y=\bd V$ and $\nu = \mu \restr {\bd V}$ in the
notation of \Ex{basic.ex}.

Before giving a formula for the Dirichlet solution operator, let us
represent the operator $\HNeu$ in block structure
\begin{equation*}
  \HNeu =
  \begin{pmatrix}
    A & B\\ B^* & D
  \end{pmatrix}
\end{equation*}
with respect to the splitting $\HS=\HSaux \oplus \HS^\Dir$, i.e.,
$\lsqr{V,\mu}= \lsqr{\bd V,\mu} \oplus \lsqr{\ring V,\mu}$.  Here,
$D=\HDir$ is the Dirichlet operator, and $\map A \HSaux \HSaux$, $\map
B {\HS^\Dir} \HSaux$.  Let $z \notin \spec D$, then $h \in \LS(z) :=
\LS^1(z)$ iff $(H-z)h \restr {\ring V}=0$.  Denote by $\map \Hmax \HS
{\HS^\Dir}$ the operator $\HNeu$ restricted to $\HS^\Dir$ (this is
actually consistent with $\Hmin:=\HDir \cap \HNeu$, the minimal
operator $\Hmin$, which acts as $\HS^\Dir \to \HS$, and $\Hmax =
(\Hmin)^*$).  Using the matrix decomposition, we have $\Hmax f = B^*
f_\bd + D f_0$, where $f=f_\bd \oplus f_0 \in \HSaux \oplus \HS^\Dir$.

Moreover,
\begin{align*}
  \qf h(f,g)
  = \iprod[\HS] {Hf} g
  &= \iprod[\HS^\Dir]{B^* f_\bd + D f_0} {g_0}
  + \iprod[\HSaux] {A f_\bd + B f_0} {g_\bd}\\
  &= \iprod[\HS^\Dir]{\Hmax f} {g_0}
  + \iprod[\HSaux] {\Gamma' f} {\Gamma g},
\end{align*}
which can be interpreted as Green's (first) identity~\eqref{eq:green},
where the ``normal derivative'' $\map {\Gamma'} \HS \HSaux$ is given
by $\Gamma' f:= A f_\bd + B f_0$, i.e.,
\begin{equation*}
  (\Gamma' f)(v)
  = \frac 1 {m(v)} \sum_{e \in E_v} \rho(e) \bigl(f(v)-f(v_e)\bigr),
  \quad v \in \bd V.
\end{equation*}

It is now easily seen that the Dirichlet solution operator is given by
\begin{equation*}
  S(z) \phi = \phi \oplus (-(D-z)^{-1} B^* \phi),
\end{equation*}
(since $(\Hmax - z) S(z)=0$ on $\HS^\Dir$ and $\Gamma S(z)
\phi=\phi$); note that the inverse exists since $z \notin \spec D$.
The \DtN\ operator is defined as
\begin{align*}
  \iprod[\HSaux]{\Lambda(z)\phi} \psi
  &=\iprod[\HS]{(\HNeu-z)S(z)\phi} g\\
  &=\iprod[\HSaux]{(A-z -B(D-z)^{-1}B^*)\phi} \psi
  +\iprod[\HS^\Dir]{B^* \phi - (D-z)(D-z)^{-1} B^*\phi} {g_0}
\end{align*}
where $g \in \HS$ is arbitrary with $g \restr {\bd V}=\psi$.  Since
the latter summand vanishes, we obtain for the \DtN\ operator
\begin{equation*}
  \map{\Lambda(z) = (A-z) - B(D-z)^{-1} B^*} \HSaux \HSaux.
\end{equation*}
Moreover, the interpretation is the same as in the manifold case: We
have $\Lambda(z) \phi = \Gamma' S(z) \phi$, i.e., the \DtN\ operator
associates to the boundary data $\phi$ the ``normal'' derivative of
the Dirichlet solution $h=S(z)\phi$.  Note that the \DtN\ operator can
also be understood as the \emph{Schur complement} of the block
operator $\HNeu-z$ with respect to the lower left ($\HS^\Dir \times
\HS^\Dir$)-block.

Finally, the spectral characterisation reads as follows: if $\lambda
\notin \spec \HDir$ (i.e., not an eigenvalue of $D$), then $\lambda
\in \spec \HNeu$ iff $0 \in \spec \Lambda(\lambda)$, i.e., if
$\Lambda(\lambda)$ is a singular matrix ($\det \Lambda(\lambda)=0$).
This fact can also be seen directly in this simple situation.

Finally, Krein's resolvent formula is just a variant of the inversion
of the block operator $\HNeu-z$, namely,
\begin{align*}
  \RNeu(z)-\iota^*\RDir(z)\iota
  &= \begin{pmatrix}
    A-z & B\\
    B^* & D-z
  \end{pmatrix}^{-1}
  - \begin{pmatrix}
    0 & 0\\
    0 & (D-z)^{-1}
  \end{pmatrix}\\
  &=
  \begin{pmatrix}
    \id_\HSaux\\
    -(D-z)^{-1}B^*
  \end{pmatrix}
  \bigr((A-z) -B (D-z)^{-1}B^*\big)^{-1}
  \bigr(\id_\HSaux, -B(D-z)^{-1}\bigl)\\
  &= S(z) \Lambda(z)^{-1} S(\conj z)^*
\end{align*}

If we allow infinite graphs, then we may also have unbounded forms
$\qf h$ (if, e.g., $\mu(v)=1$, $\rho(v)=1$ and $\deg v$ is unbounded
on the graph) and the above spectral characterisation and Krein's
formula become less obvious.  Such cases and even more general ones
(``discrete Dirichlet forms'') are considered in~\cite{hklw:12}.  We
can also use different weights for the boundary space and therefore
also have unbounded boundary pairs.  We hope to come back to the
unbounded case in a forthcoming publication.

%----------------------------------------------------------------------
% yyyy
%
% Bibliography
%
%----------------------------------------------------------------------

%----------------------------------------------------------------------
%\bibliographystyle{my-amsalpha}
%\bibliography{/home/post/Aktuell/BibTeX/literatur}

\begin{thebibliography}{DHMdS12}
\bibitem[AGrW14]{agw:14} H.~Abels, G.~Grubb and I.~G. Wood,
  \emph{Extension theory and {K}re\u\i n-type resolvent formulas for
    nonsmooth boundary value problems}, J. Funct. Anal.  \textbf{266}
  (2014), 4037--4100.

\bibitem[AS80]{alonso-simon:80}
A.~Alonso and B.~Simon, \emph{The {B}irman-{K}re\u\i n-{V}ishik theory of
  selfadjoint extensions of semibounded operators}, J. Operator Theory
  \textbf{4} (1980), 251--270.

\bibitem[AN70]{ando-nishio:70} T.~Ando and K.~Nishio, \emph{Positive
    selfadjoint extensions of positive symmetric operators}, T\^ohoku
  Math. J. (2) \textbf{22} (1970), 65--75.

\bibitem[AM12]{arendt-mazzeo:12} W.~Arendt and R.~Mazzeo,
  \emph{Friedlander's eigenvalue inequalities and the
    {D}irichlet-to-{N}eumann semigroup}, Commun. Pure
  Appl. Anal. \textbf{11} (2012), 2201--2212.

\bibitem[AtE11]{arendt-ter-elst:11}
W.~Arendt and A.~F.~M. ter Elst, \emph{The {D}irichlet-to-{N}eumann operator on
  rough domains}, J. Differential Equations \textbf{251} (2011), 2100--2124.

\bibitem[AtE12a]{arendt-ter-elst:12}
\bysame, \emph{Sectorial forms and degenerate differential operators}, J.
  Operator Theory \textbf{67} (2012), 33--72.

\bibitem[AtE12b]{arendt-ter-elst:12b}
W.~Arendt and A.~F.~M. ter Elst, \emph{From forms to semigroups}, Spectral
  theory, mathematical system theory, evolution equations, differential and
  difference equations, Oper. Theory Adv. Appl., vol. 221,
  Birkh\"auser/Springer Basel AG, Basel, 2012, pp.~47--69.

\bibitem[Ar96]{arlinskii:96} Y.~M. Arlinski{\u\i}, \emph{Maximal
    sectorial extensions and closed forms associated with them},
  Ukra\"\i n. Mat. Zh. \textbf{48} (1996), 723--738.

\bibitem[Ar99]{arlinskii:99} Y.~M. Arlinskii, \emph{On functions
    connected with sectorial operators and their extensions}, Integral
  Equations Operator Theory \textbf{33} (1999), 125--152.

\bibitem[Ar00]{arlinskii:00} Y.~Arlinskii, \emph{Abstract boundary
    conditions for maximal sectorial extensions of sectorial
    operators}, Math. Nachr. \textbf{209} (2000), 5--36.

\bibitem[Ar12]{arlinskii.in:12} Y.~Arlinski{\u\i}, \emph{Boundary
    triplets and maximal accretive extensions of sectorial operators},
  Ch.~3 in~\cite{hdss:12} (2012), 121--160.

\bibitem[BeLa07]{behrndt-langer:07}
J.~Behrndt and M.~Langer, \emph{Boundary value problems for elliptic partial
  differential operators on bounded domains}, J. Funct. Anal. \textbf{243}
  (2007), 536--565.

\bibitem[BeLa12]{behrndt-langer.in:12}
\bysame, \emph{Elliptic operators, {D}irichlet-to-{N}eumann
  maps and quasi boundary triples}, Ch.~6 in~\cite{hdss:12} (2012), 121--160.

\bibitem[BeLaL13]{bll:13a}
J.~Behrndt, M.~Langer, and V.~Lotoreichik, \emph{Schr\"odinger operators with
  {$\delta$} and {$\delta'$}-potentials supported on hypersurfaces}, Ann. Henri
  Poincar\'e \textbf{14} (2013), 385--423.

\bibitem[BeM14]{behrndt-micheler:14}
J.~Behrndt and T.~Micheler, \emph{Elliptic differential operators on
  {L}ipschitz domains and abstract boundary value problems}, J. Funct. Anal.
  \textbf{267} (2014), 3657--3709.

\bibitem[BeP]{behrndt-post:pre14}
J.~Behrndt and O.~Post, \emph{Convergence of the dirichlet-to-neumann operators
  on thin branched manifolds}, in preparation.

\bibitem[BedS09]{behrndt-de-snoo:09}
J.~Behrndt and H.~de~Snoo, \emph{On {K}re\u\i n's formula}, J. Math. Anal.
  Appl. \textbf{351} (2009), 567--578.

\bibitem[BBABr11]{bbb:11}
H.~BelHadjAli, A.~Ben~Amor and J.~F. Brasche, \emph{Large coupling
  convergence: overview and new results}, Partial differential equations and
  spectral theory, Oper. Theory Adv. Appl., vol. 211, Birkh\"auser/Springer
  Basel AG, Basel, 2011, pp.~73--117.

\bibitem[BABr08]{benamor-brasche:08}
A.~Ben~Amor and J.~F. Brasche, \emph{Sharp estimates for large coupling
  convergence with applications to {D}irichlet operators}, J. Funct. Anal.
  \textbf{254} (2008), 454--475.

\bibitem[Bir56]{birman:56}
M.~{\v{S}}. Birman, \emph{On the theory of self-adjoint extensions of positive
  definite operators}, Mat. Sb. N.S. \textbf{38(80)} (1956), 431--450.

\bibitem[BrD05]{brasche-demuth:05}
J.~Brasche and M.~Demuth, \emph{Dynkin's formula and large coupling
  convergence}, J. Funct. Anal. \textbf{219} (2005), 34--69.

\bibitem[BrMN02]{bmn:02}
J.~F. Brasche, M.~Malamud and H.~Neidhardt, \emph{Weyl function and spectral
  properties of self-adjoint extensions}, Integral Equations Operator Theory
  \textbf{43} (2002), 264--289.

\bibitem[BGrW09]{bgw:09}
B.~M. Brown, G.~Grubb and I.~G. Wood, \emph{{$M$}-functions for closed
  extensions of adjoint pairs of operators with applications to elliptic
  boundary problems}, Math. Nachr. \textbf{282} (2009), 314--347.

\bibitem[BGPa08]{bgp:08}
J.~Br\"uning, V.~Geyler and K.~Pankrashkin, \emph{Spectra of self-adjoint
  extensions and applications to solvable {S}chr\"odinger operators}, Rev.
  Math. Phys. \textbf{20} (2008), 1--70.

\bibitem[CMPc13]{cmp:13}
R.~Carlone, M.~Malamud and A.~Posilicano, \emph{On the spectral theory of
  {G}esztesy-\v {S}eba realizations of 1-{D} {D}irac operators with point
  interactions on a discrete set}, J. Differential Equations \textbf{254}
  (2013), 3835--3902.

\bibitem[Da95]{davies:95} E.~B. Davies, \emph{{Spectral theory and
      differential operators}}, Cambridge University Press, Cambridge,
  1995.

\bibitem[DHMdS00]{dhms:00}
V.~Derkach, S.~Hassi, M.~Malamud and H.~de~Snoo, \emph{Generalized resolvents
  of symmetric operators and admissibility}, Methods Funct. Anal. Topology
  \textbf{6} (2000), 24--55.

\bibitem[DHMdS06]{dhms:06} \bysame, \emph{Boundary relations and their
    {W}eyl families}, Trans. Amer. Math. Soc. \textbf{358} (2006),
  5351--5400.

\bibitem[DHMdS09]{dhms:09} \bysame, \emph{Boundary relations and
    generalized resolvents of symmetric operators},
  Russ. J. Math. Phys. \textbf{16} (2009), 17--60.

\bibitem[DHMdS12]{dhms.in:12} \bysame, \emph{Boundary triplets and
    {W}eyl functions. {R}ecent developments}, Ch.~7 in~\cite{hdss:12}
  (2012), 121--160.

\bibitem[DM95]{derkach-malamud:95} V.~Derkach and M.~Malamud,
  \emph{The extension theory of {H}ermitian operators and the moment
    problem}, J. Math. Sci. (New York) \textbf{73} (1995), 141--242.

\bibitem[Di96]{ding:96}
Z.~Ding, \emph{A proof of the trace theorem of {S}obolev spaces on {L}ipschitz
  domains}, Proc. Amer. Math. Soc. \textbf{124} (1996), 591--600.

\bibitem[EL04]{eschwe-langer:04} D.~Eschw{\'e} and M.~Langer,
  \emph{Variational principles for eigenvalues of self-adjoint
    operator functions}, Integral Equations Operator Theory
  \textbf{49} (2004), 287--321.

\bibitem[Ex08]{exner:08} P.~Exner, \emph{Leaky quantum graphs: a
    review}, Analysis on graphs and its applications,
  Proc. Sympos. Pure Math., vol.~77, Amer. Math. Soc., Providence, RI,
  2008, pp.~523--564.

\bibitem[GeMi08]{gesztesy-mitrea:08}
F.~Gesztesy and M.~Mitrea, \emph{Generalized {R}obin boundary conditions,
  {R}obin-to-{D}irichlet maps, and {K}rein-type resolvent formulas for
  {S}chr\"odinger operators on bounded {L}ipschitz domains}, Perspectives in
  partial differential equations, harmonic analysis and applications, Proc.
  Sympos. Pure Math., vol.~79, Amer. Math. Soc., Providence, RI, 2008,
  pp.~105--173.

\bibitem[GeMi09]{gesztesy-mitrea:09}
\bysame, \emph{Robin-to-{R}obin maps and {K}rein-type resolvent formulas for
  {S}chr\"odinger operators on bounded {L}ipschitz domains}, Modern analysis
  and applications. {T}he {M}ark {K}rein {C}entenary {C}onference. {V}ol. 2:
  {D}ifferential operators and mechanics, Oper. Theory Adv. Appl., vol. 191,
  Birkh\"auser Verlag, Basel, 2009, pp.~81--113.

\bibitem[GeMi11]{gesztesy-mitrea:11}
\bysame, \emph{A description of all self-adjoint extensions of the {L}aplacian
  and {K}re\u\i n-type resolvent formulas on non-smooth domains}, J. Anal.
  Math. \textbf{113} (2011), 53--172.

\bibitem[GG91]{gorbachuk-gorbachuk:91} V.~I. Gorbachuk and
  M.~L. Gorbachuk, \emph{Boundary value problems for operator
    differential equations}, Mathematics and its Applications (Soviet
  Series), vol.~48, Kluwer Academic Publishers Group, Dordrecht, 1991.

\bibitem[Gre87]{greiner:87}
G.~Greiner, \emph{Perturbing the boundary conditions of a generator}, Houston
  J. Math. \textbf{13} (1987), 213--229.

\bibitem[Gv85]{grisvard:85}
P.~Grisvard, \emph{Elliptic problems in nonsmooth domains}, Monographs and
  Studies in Mathematics, vol.~24, Pitman (Advanced Publishing Program),
  Boston, MA, 1985.

\bibitem[Gr68]{grubb:68}
G.~Grubb, \emph{A characterization of the non-local boundary value problems
  associated with an elliptic operator}, Ann. Scuola Norm. Sup. Pisa (3)
  \textbf{22} (1968), 425--513.

\bibitem[Gr70]{grubb:70}
\bysame, \emph{Les probl\`emes aux limites g\'en\'eraux d'un op\'erateur
  elliptique, provenant de la th\'eorie variationnelle}, Bull. Sci. Math. (2)
  \textbf{94} (1970), 113--157.

\bibitem[Gr08]{grubb:08}
\bysame, \emph{Krein resolvent formulas for elliptic boundary problems in
  nonsmooth domains}, Rend. Semin. Mat. Univ. Politec. Torino \textbf{66}
  (2008), 271--297.

\bibitem[Gr11]{grubb:11}
\bysame, \emph{The mixed boundary value problem, {K}rein resolvent formulas
  and spectral asymptotic estimates}, J. Math. Anal. Appl. \textbf{382} (2011),
  339--363.

\bibitem[GKMV13]{gkmv:13} L.~Grubi{\v{s}}i{\'c}, V.~Kostrykin,
  K.~A. Makarov and K.~Veseli{\'c}, \emph{Representation theorems for
    indefinite quadratic forms revisited}, Mathematika \textbf{59}
  (2013), 169--189.

\bibitem[HKLW12]{hklw:12}
S.~Haeseler, M.~Keller, D.~Lenz and R.~Wojciechowski, \emph{Laplacians on
  infinite graphs: {D}irichlet and {N}eumann boundary conditions}, J. Spectr.
  Theory \textbf{2} (2012), 397--432.

\bibitem[HdSS12]{hdss:12} S.~Hassi, H.~S.~V. de~Snoo and
  F.~H. Szafraniec (eds.), \emph{Operator methods for boundary value
    problems}, London Mathematical Society Lecture Note Series,
  vol. 404, Cambridge University Press, Cambridge, 2012.

\bibitem[JN01]{janas-naboko:01} J.~Janas and S.~Naboko, \emph{Spectral
    properties of selfadjoint {J}acobi matrices coming from birth and
    death processes}, Recent advances in operator theory and related
  topics ({S}zeged, 1999), Oper. Theory Adv. Appl., vol.  127,
  Birkh\"auser, Basel, 2001, pp.~387--397.

\bibitem[JK81]{jerison-kenig:81a} D.~S. Jerison and C.~E. Kenig,
  \emph{The {N}eumann problem on {L}ipschitz domains},
  Bull. Amer. Math. Soc. (N.S.) \textbf{4} (1981), 203--207.


\bibitem[JeKe95]{jerison-kenig:95} \bysame, \emph{The inhomogeneous
    {D}irichlet problem in {L}ipschitz domains},
  J. Funct. Anal. \textbf{130} (1995), 161--219.

\bibitem[Ka66]{kato:66}
T.~Kato, \emph{Perturbation theory for linear operators}, Springer-Verlag,
  Berlin, 1966.

\bibitem[KoM10]{kostenko-malamud:10}
A.~S. Kostenko and M.~M. Malamud, \emph{1-{D} {S}chr\"odinger operators with
  local point interactions on a discrete set}, J. Differential Equations
  \textbf{249} (2010), 253--304.

\bibitem[Kr47]{krein:47i}
M.~Krein, \emph{The theory of self-adjoint extensions of semi-bounded
  {H}ermitian transformations and its applications. {I}}, Rec. Math. [Mat.
  Sbornik] N.S. \textbf{20(62)} (1947), 431--495.

\bibitem[KuK02]{kuchment-kunyansky:02} P.~Kuchment and L.~Kunyansky,
  \emph{Differential operators on graphs and photonic crystals},
  Adv. Comput. Math. \textbf{16} (2002), 263--290, Modeling and
  computation in optics and electromagnetics.

\bibitem[LT77]{langer-textorius:77} H.~Langer and B.~Textorius,
  \emph{On generalized resolvents and {$Q$}-functions of symmetric
    linear relations (subspaces) in {H}ilbert space}, Pacific J.
  Math. \textbf{72} (1977), 135--165.

\bibitem[LiMa68]{lions-magenes:68} J.-L. Lions and E.~Magenes,
  \emph{Probl\`emes aux limites non homog\`enes et
    applications. {V}ol. 1}, Travaux et Recherches Math\'ematiques,
  No. 17, Dunod, Paris, 1968.

\bibitem[LySt83]{lyantse-storozh:83} V.~{\`E}. Lyantse and
  O.~G. Storozh, \emph{Metody teorii neogranichennykh operatorov},
  ``Naukova Dumka'', Kiev, 1983.

\bibitem[M92]{malamud:92b} M.~M. Malamud, \emph{Certain classes of
    extensions of a lacunary {H}ermitian operator}, Ukra\"\i
  n. Mat. Zh. \textbf{44} (1992), 215--233.

\bibitem[M10]{malamud:10} \bysame, \emph{Spectral theory of elliptic
    operators in exterior domains}, Russ. J. Math. Phys. \textbf{17}
  (2010), 96--125.

\bibitem[MMo02]{malamud-mogilevskii:02}
M.~M. Malamud and V.~I. Mogilevskii, \emph{Kre\u\i n type formula for canonical
  resolvents of dual pairs of linear relations}, Methods Funct. Anal. Topology
  \textbf{8} (2002), 72--100.

\bibitem[MSch12]{malamud-schmuedgen:12}
M.~M. Malamud and K.~Schm{\"u}dgen, \emph{Spectral theory of {S}chr\"odinger
  operators with infinitely many point interactions and radial positive
  definite functions}, J. Funct. Anal. \textbf{263} (2012), 3144--3194.

\bibitem[McI69]{mcintosh:69} A.~G.~R. McIntosh, \emph{Bilinear forms
    in {H}ilbert space}, J. Math. Mech.  \textbf{19} (1969),
  1027--1045.

\bibitem[McI70]{mcintosh:70} \bysame, \emph{Hermitian bilinear forms
    which are not semibounded}, Bull.  Amer. Math. Soc. \textbf{76}
  (1970), 732--737.

\bibitem[MiMi07]{mitrea-mitrea:07}
I.~Mitrea and M.~Mitrea, \emph{The {P}oisson problem with mixed boundary
  conditions in {S}obolev and {B}esov spaces in non-smooth domains}, Trans.
  Amer. Math. Soc. \textbf{359} (2007), 4143--4182.

\bibitem[MiMiT01]{mmt:01}
D.~Mitrea, M.~Mitrea, and M.~Taylor, \emph{Layer potentials, the {H}odge
  {L}aplacian, and global boundary problems in nonsmooth {R}iemannian
  manifolds}, Mem. Amer. Math. Soc. \textbf{150} (2001), x+120.

\bibitem[MiT99]{mitrea-taylor:99}
M.~Mitrea and M.~Taylor, \emph{Boundary layer methods for {L}ipschitz domains
  in {R}iemannian manifolds}, J. Funct. Anal. \textbf{163} (1999), 181--251.

% \bibitem[MiT00]{mitrea-taylor:00}
% \bysame, \emph{Potential theory on {L}ipschitz domains in {R}iemannian
%   manifolds: {S}obolev-{B}esov space results and the {P}oisson problem}, J.
%   Funct. Anal. \textbf{176} (2000), 1--79.

% \bibitem[MiT01]{mitrea-taylor:01}
% \bysame, \emph{Potential theory on {L}ipschitz domains in {R}iemannian
%   manifolds: {$L^P$} {H}ardy, and {H}\"older space results}, Comm. Anal. Geom.
%   \textbf{9} (2001), 369--421.

\bibitem[MiT03]{mitrea-taylor:03}
\bysame, \emph{Potential theory on {L}ipschitz domains in {R}iemannian
  manifolds: the case of {D}ini metric tensors}, Trans. Amer. Math. Soc.
  \textbf{355} (2003), 1961--1985.


\bibitem[MiT05]{mitrea-taylor:05}
\bysame, \emph{Sobolev and {B}esov space estimates for solutions to second
  order {PDE} on {L}ipschitz domains in manifolds with {D}ini or {H}\"older
  continuous metric tensors}, Comm. Partial Differential Equations \textbf{30}
  (2005), 1--37.

\bibitem[Mo12]{mogilevskii:12}
V.~Mogilevskii, \emph{Boundary pairs and boundary conditions for general (not
  necessarily definite) first-order symmetric systems with arbitrary deficiency
  indices}, Math. Nachr. \textbf{285} (2012), 1895--1931.

\bibitem[MuNP13]{mnp:13} D.~Mugnolo, R.~Nittka and O.~Post, \emph{Norm
    convergence of sectorial operators on varying {H}ilbert spaces},
  Oper. Matrices \textbf{7} (2013), 955--995.

\bibitem[Pa06]{pankrashkin:06b}
K.~Pankrashkin, \emph{Resolvents of self-adjoint extensions with mixed boundary
  conditions}, Rep. Math. Phys. \textbf{58} (2006), 207--221.

\bibitem[Pc01]{posilicano:01}
A.~Posilicano, \emph{A {K}re\u\i n-like formula for singular perturbations of
  self-adjoint operators and applications}, J. Funct. Anal. \textbf{183}
  (2001), 109--147.

\bibitem[Pc04]{posilicano:04}
\bysame, \emph{Boundary triples and {W}eyl functions for singular perturbations
  of self-adjoint operators}, Methods Funct. Anal. Topology \textbf{10} (2004),
  57--63.

\bibitem[Pc08]{posilicano:08}
\bysame, \emph{Self-adjoint extensions of restrictions}, Oper. Matrices
  \textbf{2} (2008), 483--506.

\bibitem[P12]{post:12}
O.~Post, \emph{Spectral analysis on graph-like spaces}, Lecture Notes in
  Mathematics, vol. 2039, Springer, Heidelberg, 2012.

\bibitem[P]{post:pre14a} \bysame, \emph{{Boundary triples associated
      with quadratic forms}}, (in preparation).

\bibitem[RS80]{reed-simon-1}
M.~Reed and B.~Simon, \emph{{Methods of modern mathematical physics I:
  Functional analysis}}, Academic Press, New York, 1980.

\bibitem[Ry07]{ryzhov:07}
V.~Ryzhov, \emph{A general boundary value problem and its {W}eyl function},
  Opuscula Math. \textbf{27} (2007), 305--331.

\bibitem[Sa08]{sahbani:08}
J.~Sahbani, \emph{Spectral theory of certain unbounded {J}acobi matrices}, J.
  Math. Anal. Appl. \textbf{342} (2008), 663--681.

\bibitem[Sch12]{schmuedgen:12}
K.~Schm{\"u}dgen, \emph{Unbounded self-adjoint operators on {H}ilbert space},
  Graduate Texts in Mathematics, vol. 265, Springer, Dordrecht, 2012.

\bibitem[Sh68]{shamir:68}
E.~Shamir, \emph{Regularization of mixed second-order elliptic problems},
  Israel J. Math. \textbf{6} (1968), 150--168.

% \bibitem[TMiV05]{mtv:05}
% M.~Taylor, M.~Mitrea, and A.~Vasy, \emph{Lipschitz domains, domains with
%   corners, and the {H}odge {L}aplacian}, Comm. Partial Differential Equations
%   \textbf{30} (2005), 1445--1462.

\bibitem[Tr00]{tretter:00}
C.~Tretter, \emph{Linear operator pencils {$A-\lambda B$} with discrete
  spectrum}, Integral Equations Operator Theory \textbf{37} (2000), 357--373.

\bibitem[Vi52]{vishik:52}
M.~I. Vishik, \emph{On general boundary problems for elliptic differential
  equations}, Trudy Moskov. Mat. Ob\v s\v c. \textbf{1} (1952), 187--246
  (English translation in Am. Math. Soc., Transl. (2) 24 (1963), 107--172).

\end{thebibliography}
%\printbibliography
%----------------------------------------------------------------------
\providecommand{\bysame}{\leavevmode\hbox to3em{\hrulefill}\thinspace}

\end{document}